\documentclass[11pt, letterpaper, reqno]{amsart}
\usepackage[utf8]{inputenc} 	
\usepackage{microtype} 			
\usepackage{geometry} 			
\usepackage{amsmath}			
\usepackage{amsthm}		 		
\usepackage{amssymb}	 		
\usepackage{slashed}
\usepackage{bm}					
\usepackage{mathrsfs}			
\usepackage{thmtools}
\usepackage{xcolor}				
\definecolor{darkyellow}{rgb}{0.7, 0.7, 0}
\definecolor{darkred}{rgb}{0.6, 0.1, 0.1}
\definecolor{darkblue}{rgb}{0.2, 0.2, 0.7}
\definecolor{darkgreen}{rgb}{0.1, 0.4, 0.1}
\definecolor{bettercyan}{rgb}{0.1, 0.4, 0.7}
\usepackage{tikz}				
\usetikzlibrary{patterns}
\usetikzlibrary{arrows}
\usepackage[all]{xy}			
\usepackage{graphicx}			
\usepackage{pgfplots}			
\usepackage{enumitem} 			
\usepackage{array}
\usepackage{lmodern}			
\usepackage[T1]{fontenc}		
\usepackage{cancel}
\usepackage[breaklinks=true, colorlinks=true, citecolor=bettercyan, linkcolor=darkred, urlcolor=magenta]{hyperref}
\usepackage{cleveref}
\usepackage[backend=bibtex, style=alphabetic, backref=true, firstinits=true, isbn=false, date=year]{biblatex}
\makeatletter
\renewbibmacro{in:}{}
\DeclareFieldFormat{pages}{#1}
\renewcommand*{\bibnamedash}{%
	\leavevmode\raise +0.6ex\hbox to 5.5ex{\hrulefill}.\space\space}

\InitializeBibliographyStyle{\global\undef\bbx@lasthash}

\newbibmacro*{bbx:savehash}{\savefield{fullhash}{\bbx@lasthash}}

\renewbibmacro*{author}{%
	\ifboolexpr{
		test \ifuseauthor
		and
		not test {\ifnameundef{author}}
	}
	{%
		\iffieldequals{fullhash}{\bbx@lasthash}
		{\bibnamedash\addcomma\space}
		{\printnames{author}}%
		\usebibmacro{bbx:savehash}%
		\iffieldundef{authortype}
		{}
		{%
			\setunit{\addcomma\space}%
			\usebibmacro{authorstrg}%
		}%
	}
	{\global\undef\bbx@lasthash}%
}
\makeatother

\numberwithin{equation}{section}

\newtheorem{propositionx}{Proposition}[section]
\newenvironment{proposition}
{\pushQED{\qed}\propositionx}
{\popQED\endpropositionx}
\newenvironment{propositionp}
{\pushQED{\qed}\propositionx}
{\popQED\endpropositionx}
\newtheorem*{theorem*}{Theorem}
\newtheorem{theorem}{Theorem}

\newtheorem{lemmax}[propositionx]{Lemma}
\newenvironment{lemma}
{\pushQED{\qed}\lemmax}
{\popQED\endlemmax}

\theoremstyle{remark}
\newtheorem{remarkx}[propositionx]{Remark}
\newtheorem*{remarkxa}{Remark}
\newtheorem{example}{Example}[section]
\newenvironment{remark}
{\pushQED{\qed}\remarkx}
{\popQED\endremarkx}
\newenvironment{remark*}
{\pushQED{\qed}\remarkxa}
{\popQED\endremarkxa}
\newcommand{\bbC}{\mathbb{C}}

\newcommand{\bbN}{\mathbb{N}}

\newcommand{\bbR}{\mathbb{R}}
\newcommand{\bbS}{\mathbb{S}}

\newcommand{\calA}{\mathcal{A}}
\newcommand{\calB}{\mathcal{B}}
\newcommand{\calC}{\mathcal{C}}
\newcommand{\calD}{\mathcal{D}}
\newcommand{\calE}{\mathcal{E}}
\newcommand{\calF}{\mathcal{F}}
\newcommand{\calG}{\mathcal{G}}

\newcommand{\calI}{\mathcal{I}}
\newcommand{\calJ}{\mathcal{J}}
\newcommand{\calK}{\mathcal{K}}

\newcommand{\calM}{\mathcal{M}}

\newcommand{\calR}{\mathcal{R}}
\newcommand{\calS}{\mathcal{S}}

\newcommand{\calV}{\mathcal{V}}

\newcommand{\calX}{\mathcal{X}}
\newcommand{\calY}{\mathcal{Y}}

\newcommand{\scrA}{\mathscr{A}}

\newcommand{\fraka}{\mathfrak{a}}

\newcommand{\frakm}{\mathfrak{m}}

\newcommand{\bfa}{\mathbf{a}}

\newcommand{\bfr}{\mathbf{r}}

\newcommand{\bfx}{\mathbf{x}}

\newcommand{\dd}{\,\mathrm{d}}

\makeatletter
\@namedef{subjclassname@2020}{%
	\textup{2020} Mathematics Subject Classification}
\makeatother

\title[Complete asymptotic analysis of low energy scattering]{Complete asymptotic analysis of low energy scattering for Schr\"odinger operators with a short-range potential}
\author{Ethan Sussman}
\date{September 4th, 2025 (Last update). November 14th, 2024 (Draft).}
\email{ethan.sussman@northwestern.edu}
\address{Department of Mathematics, Northwestern University, Illinois, USA}

\begin{document}
	
\begin{abstract}
	Recent work by Hintz--Vasy provides a partial asymptotic analysis of the low-energy limit of scattering for Schr\"odinger operators with a short-range potential.  
	Using a slight refinement of Hintz's algorithm, we complete the asymptotic analysis by providing full asymptotic expansions in every possible asymptotic regime. Moreover, the analysis is done in any dimension $d\geq 3$, for any asymptotically conic manifold, and we keep track of partial multipole expansions. Applications include full asymptotic analyses of the Schr\"odinger, wave, and Klein--Gordon equations, one of these being described in a companion paper. Using previous work, only partial asymptotic analyses were possible.
\end{abstract}

\maketitle

\tableofcontents

\section{Introduction}

This is a paper about full asymptotic expansions. If $f\in C^{K}(\bbR)$, then Taylor's theorem tells us that $f(x)$ has a partial asymptotic expansion as $x\to 0$:
\begin{equation}
	f(x) - \sum_{j=0}^K \frac{f^{(j)}(0)}{j!} x^j= o(x^{K}).
	\label{eq:Taylor} 
\end{equation}
If $f\in C^\infty(\bbR)$, then one has a \emph{full} asymptotic expansion:
\begin{equation}
	f(x) \sim \sum_{j=0}^\infty \frac{f^{(j)}(0)}{j!} x^j, 
\end{equation}
where the precise meaning of `$\sim$' is that \cref{eq:Taylor} holds for every $K\in \bbN$. 
Usually, knowing that a function is $C^0$ is more important than knowing that it is $C^1$, and knowing that it is $C^1$ is more important that it is $C^2$, and so on. 
However, keeping track of precise amounts differentiability can be tedious. If one can get away with working with $C^\infty$ functions, then it is often advisable to do so, even at loss of generality. For this reason, proving the existence of full asymptotic expansions can be worthwhile.

In PDE, asymptotic expansions provide precise statements about the behavior of solutions in various asymptotic regimes. Unlike in the simple one-dimensional setting, in higher dimensions variables can be taken to their extreme values \emph{jointly} in different ways. For example, when studying the Schr\"odinger equation on $\smash{\mathbb{R}^{1,3}_{t,\bfx}}$ with a time-independent short-range potential, very different behavior is seen 
\begin{itemize}
	\item as $t\to\infty$ if $\bfx$ is fixed vs. 
	\item if 
	$\lVert \bfx \rVert=vt$ for some $v>0$. 
\end{itemize}In the former regime, only bound states contribute at leading order, whereas the latter regime allows one to follow a wavefunction in the continuous spectrum of the Schr\"odinger operator as it disperses. The latter decays pointwise. Bound states do not. 
These are not the only regimes. For example, following the evolving wavefunction along the curve $\lVert \bfx \rVert=ct^k$ for $k>1$, $c>0$,
results in outrunning the wavepacket --- an accelerating observer sees rapid decay, at least if the initial data is also rapidly decaying and sufficiently smooth. A final asymptotic regime that needs to be distinguished is when $\lVert \bfx \rVert\sim \sqrt{t}$,
but this ``transitional regime'' turns out to only be relevant to understanding subleading terms in the asymptotic expansion. 

When proving results like those described above, it is of interest to show that they are complete. This means showing that the relevant asymptotic expansions agree at the interface between asymptotic regimes, i.e.\ ``at the corner.'' For example:
\begin{equation}
	f \in C^\infty([0,\infty)\times (0,\infty)) \cap C^\infty ((0,\infty)\times [0,\infty)) \not\Rightarrow f\in C^\infty([0,\infty)^2),
\end{equation}
as $f(x,y)=1/(x+y)$ is a counter-example; the problem is smoothness (in this case, even continuity) at the corner $(0,0)\in [0,\infty)^2$. In our companion paper \cite{Full}, we prove (using the results in the present paper) that the asymptotic regimes described above are complete for the Schr\"odinger equation. 

Asymptotics like those discussed in the previous paragraph are often deduced from results about \emph{low-energy scattering theory}. This subfield of scattering theory, which goes back to Wigner's work on scattering thresholds \cite{Wigner}, studies monochromatic waves (solutions of the Helmholtz equation) in the limit where the wavelength is taken to infinity, the deep infrared. In this limit, the Helmholtz equation behaves like Laplace's equation in bounded regions, but the effect of having a positive spectral parameter is still felt at large distances.
Via the Fourier transform, asymptotic expansions for low-energy scattering can be converted into asymptotic expansions on spacetime for the usual PDEs. In the mathematics literature, this strategy goes back to Jensen--Kato \cite{JensenKato} and has been used since to study wave propagation  in both the relativistic and non-relativistic settings. See e.g.\ \cite{Schlag1}\cite{Schlag2, Schlag3}\cite{Goldberg} for a small sample of the research literature regarding the Euclidean case.  

This paper continues a program recently undertaken by Hintz and Vasy \cite{VasyLA, Vasy, VasyLagrangian}\cite{HintzPrice}. Hintz's goal was to prove Price's law without symmetry assumptions. 
\emph{Price's law} describes the asymptotic tail of radiation on black hole spacetimes and is therefore significant in the mathematical analysis of general relativity. Because the spatial slices of black hole spacetimes are only \emph{asymptotically} Euclidean and not exactly Euclidean,\footnote{Black hole exteriors can be thought of as having two ends, the $r\to\infty$ end and the horizon. It turns out that the large end is the one more relevant for Price's law, which applies even without a horizon, as long as the large-$r$ tails of the metric perturbation are similar to Schwarzschild. So, we use ``Price's law'' to refer to $t^{-3}$ pointwise decay on asymptotically Schwarzschild spacetimes, as well as similar results regarding potential scattering.} Hintz's work required going beyond the Euclidean case. For this, he used Vasy's low-energy toolkit developed in \cite{Vasy, VasyLagrangian}. We will also use this toolkit.
Other works on the asymptotically Euclidean case include \cite{Hafner}\cite{GH1, GH2, GHK}\cite{Bouclet0,Bouclet}\cite{VasyWunsch}\cite{Tao}\cite{Waters}. We will say a bit more about some of these below.

Our aim here is to complete the Hintz--Vasy low-energy program. Hintz has proven only \emph{partial} asymptotic expansions at low-energy. For wave propagation, this yields only partial asymptotic expansions at large-$t$: the leading-order term for the wave equation (that is, Price's law) and the first few terms for the Schr\"odinger/Klein--Gordon equations. 
Here, we provide \emph{full} asymptotic expansions, in all possible asymptotic regimes (except for the high energy regime, which we consider already understood and therefore do not discuss further).

Low-energy scattering theory has intrinsic interest, but our real motivation lies in wave propagation. We discuss Price's law in \S\ref{sec:Price}, but applications are otherwise lacking here.
This is because the Schr\"odinger equation in $(1+3)$-dimensions is handled in our companion paper \cite{Full}. There, we cite the results here. Were we to only use pre-existing results -- and not the results below -- we would only be able to produce partial expansions. 
Likewise, if one takes the argument in \cite{HintzPrice} and replaces the partial low-energy analysis there with the full low-energy analysis here, then one gets a full asymptotic expansion for linear waves on asymptotically flat spacetimes, rather than only the leading terms. This reduces problems about linear wave propagation on particular spacetimes to a matter of computation, the computation of the terms below. We will indicate some examples at the end of  \S\ref{sec:Price}.

The spacetime compactifications in  \cite{HintzPrice}\cite{Full} -- the former concerning the wave equation, and the latter concerning the Schr\"odinger equation -- are different\footnote{Besides the lack of spacelike infinity in the latter, there is an additional parabolic scaling regime, which is ultimately related to the low-energy transitional regime `tf' in the present paper.}, which shows that going from the low-energy analysis to the spacetime analysis is non-trivial. It is substantially easier for wave than for Schr\"odinger. However, the same low-energy theory is the input for both. The only difference is the particular Fourier transform taken. 

\begin{remark*}
	 The Klein--Gordon equation (and other massive equations, like massive Dirac) can be treated along the same lines as the Schr\"odinger equation. It uses the same compactification within the lightcone. 
\end{remark*}

\begin{remark*}
	So far, we have only been discussing parabolic and hyperbolic PDE with \emph{time-independent} coefficients. The opposite extreme is PDE with time-dependent coefficients which settle down to their constant-coefficient analogues as $t\to\infty$. This latter setting has been studied by many authors. Works which produce asymptotics in this setting, using microlocal tools similar tools to those used here, include Baskin--Vasy--Wunsch \cite{BasinVasyWunsch} for wave, Sussman \cite{desc} for Klein--Gordon, and Gell-Redman--et.\ al.\ \cite{Parabolicsc, HassellNL} for Schr\"odinger. This list is not exhaustive.

	Ultimately, one would like to be able to study time-dependent coefficients which settle down to time-independent but non-constant coefficients. First steps in this direction have been undertaken by Hintz \cite{Hintz3b, Hintz3bwave} for the wave equation, using microlocal tools strictly more sophisticated than those applied here. Again, we are only discussing the microlocal literature; the non-microlocal literature on this topic is large, and these brief remarks are not an attempt to indicate anything about it.
\end{remark*}

\subsection{Main result}\label{subsec:main_result}
Let $X$ denote an asymptotically conic manifold, including the ur-example of Euclidean space.\footnote{We are using the term ``asymptotically conic manifold'' in the sense of Melrose \cite{MelroseGeometric}. For the reader not familiar with such terminological conventions, let us emphasize that $X$ is a compact manifold-\emph{with-boundary}. Indeed, any compact manifold-with-boundary can be endowed with a metric that makes it asymptotically conic. When we speak of ``Euclidean space'' in this paper, what we really mean is the radial compactification $X=\bbR^d\sqcup \infty \bbS^{d-1}$ that compactifies $\bbR^d$ by attaching one point ``at infinity'' for each ray in $\bbR^d$ emanating from the origin, in such a way so that $1/\langle r \rangle$ becomes a boundary-defining-function of $\infty \bbS^{d-1} = \partial X$. This is equivalent to stereographically compactifying to a hemisphere.}

The crux of Hintz's treatment of Price's law  (note that Hintz considers only the case $d=\dim X=3$, but a similar analysis applies for $d\geq 3$) is the proof that, given Schwartz forcing $f \in \calS(X)$, the resolvent output $R(E\pm i0) f$
admits a partial asymptotic expansion on some particular manifold-with-corners $X^+_{\mathrm{res}}$ compactifying the low-energy limit: 
\begin{equation}
X^+_{\mathrm{res}} \hookleftarrow \bbR^+_E \times X, \qquad X^+_{\mathrm{res}} \cap \operatorname{cl}_{X^+_{\mathrm{res}}} \{E<E_0\} \Subset X^+_{\mathrm{res}}\text{ for all }E_0>0.
\end{equation}
(The second of these equations is just saying that the low-energy end of $X^+_{\mathrm{res}}$ is compact, without saying anything about the high-energy end.)
Specifically, $X^+_{\mathrm{res}} = [ [0,\infty)_\sigma \times X ; \{0\}\times \partial X]$ is the result of blowing up the corner of $[0,\infty)_\sigma \times X$, where $\sigma$ is related to the energy $E$ by $E=\sigma^2$. 

This manifold-with-corners has three faces, bf (boundary face, the pure large-distance regime), tf (transition face), and zf (zero face, the pure low-energy regime) --- see \Cref{fig}.

\begin{figure}[h]
	\begin{tikzpicture}
		\fill[gray!5] (5,-1.5) -- (0,-1.5) -- (0,0)  arc(-90:0:1.5) -- (5,1.5) -- cycle;
		\draw[dashed] (5,-1.5) -- (5,1.5);
		\node[white] (ff) at (.75,.75) {ff};
		\node (zfp) at (-.35,-.75) {$\mathrm{zf}$};
		\node (flr) at (3,1.7) {$\mathrm{bf}$};
		\node (mf) at (.8,.6) {$\mathrm{tf}$};
		\draw (0,-1.5) -- (0,0)  arc(-90:0:1.5)  -- (5,1.5);
		\draw[dashed] (5,-1.5) -- (0,-1.5);
		\node[darkgray] at (3,0) {$X^+_{\mathrm{res}}$};
		\draw[->, darkred] (.1,-.1) -- (.1,-.75) node[right] {$1/r$};
		\draw[->, darkred] (.1,-.1) to[out=0, in=-150] (.8,.1) node[below right] {$\!\!\hat{r}=r \sigma$};
		\draw[->, darkred] (1.6,1.4) -- (2.4,1.4) node[below] {$\sigma$};
		\draw[->, darkred] (1.6,1.4) to[out=-90, in=57] (1.34,.6) node[right] {$1/\hat{r}$};
	\end{tikzpicture}
	\caption{The mwc $X^+_{\mathrm{res}}$, with the boundary hypersurfaces $\mathrm{bf} = \mathrm{cl}_{X^+_{\mathrm{res}}}\{(\sigma,\theta) : \sigma>0,\theta\in \partial X \}$,  $\mathrm{zf} = \mathrm{cl}_{X^+_{\mathrm{res}}}\{(0,x) : x\in X^\circ \}$, and $\mathrm{tf}$.}
	\label{fig}
\end{figure}
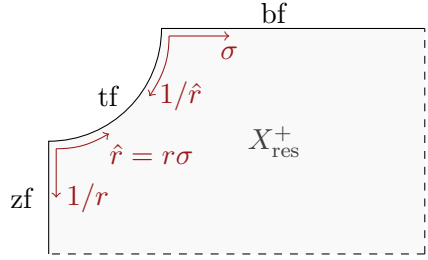

Our main theorem is the following. We phrase it in terms of a set of admissible one-parameter families $\{P(\sigma)\}_{\sigma\geq 0}$ of operators $P(\sigma) \in \operatorname{Diff}^2(X^\circ)$. The requirements for such a family to be admissible are enumerated in \S\ref{sec:op}. Roughly, the set of admissible families is slightly larger than the set of (appropriately conjugated) spectral families of short-range Schr\"odinger operators on $X$. 
We prove in \S\ref{sec:op} that Schr\"odinger operators with coefficients admitting full asymptotic expansions in the large-$r$ limit fall under our framework if there is no bound state or resonance at zero energy. This last clause generically holds.

The reason for the parenthetical ``appropriately conjugated'' above is that when we are studying a (short-range\footnote{ Long-range potentials, such as the $O(1/r)$ Coulomb potential, require serious modifications and we do not discuss them here.}) Schr\"odinger operator $P_0 = \triangle_g+V$, what we call $P(\sigma)$ is not the spectral family $P_0-\sigma^2$ but rather the result 
\begin{equation}
	P(\sigma) = e^{i\sigma r} r^{i\sigma\frakm/2} (P_0-\sigma^2) e^{-i\sigma r}   r^{-i\sigma \frakm/2}
	\label{eq:misc_yu1}
\end{equation}
of conjugating $P_0-\sigma^2$ by the exponentials\footnote{When $X=\overline{\bbR^d}$, what we call $r$ really only agrees with the Euclidean radial coordinate away from the origin. It is chosen so as to be smooth everywhere. On general asymptotically conic manifolds, $r=\rho^{-1}$, where $\rho$ is a boundary-defining-function.} 
\begin{equation} 
	e^{\pm i\sigma r} r^{\pm i\sigma\frakm/2}. 
\end{equation} 
(Here, $\frakm\in \bbR$ is proportional to the coefficient of the $O(1/r)$ term in the expansion of the $\dd r^2$ component of the metric.)
Studying $P(\sigma)$ is equivalent to studying $P_0-\sigma^2$, so this distinction is not important.
There are two reasons why we phrase things in terms of the conjugated spectral family rather than the spectral family itself:
\begin{itemize}
	\item  when studying wave propagation on stationary but non-static black hole spacetimes, such as the Kerr black hole as in \cite{HintzPrice}, the relevant $P(\sigma)$ do not arise by conjugating the spectral family of a Schr\"odinger operator, 
	\item  it is usually more convenient to work with $P(\sigma)$ than $P_0-\sigma^2$ (otherwise, we would have to carry many exponentials around), and we would like to reserve the simpler notation for the operator we talk about the most.
\end{itemize}
We invite the first-time reader to consider $P(\sigma)$ to be the conjugated spectral family of a Schr\"odinger operator (and to not worry about the ``conjugated'' modifier, which is just a matter of convenience).

Our main theorem also refers to the ``limiting resolvent'' $R(\sigma^2\pm i0)$. If $P(\sigma)$ is the conjugated spectral family of a Schr\"odinger operator, then this is just the usual limiting resolvent for the spectral family. The general definition appears in \S\ref{sec:op}; see \cref{eq:resolv_def}. What one does is invert $P(\sigma)$ between appropriate function spaces (see \Cref{prop:Vasy_absorption}) and then undo the conjugation above: 
\begin{equation}
	R(\sigma^2-i0)=
	e^{-i\sigma r}r^{-i\sigma\frakm/2}	P(\sigma)^{-1} e^{i\sigma r} r^{i\sigma \frakm/2}. 
\end{equation}
This approach to the limiting absorption principle goes back to Sommerfeld, but the form we use is from \cite{VasyLA}.

Our main theorem, proven in \S\ref{sec:arg}, is:
\begin{theorem}
	Let $X$ be an asymptotically conic manifold, and consider an admissible family $\{P(\sigma)\}_{\sigma\geq 0}$ of operators on it, as described in \S\ref{sec:op}, including 
	\begin{itemize}
		\item the assumption that $P(0)$ lacks a bound state or resonance at zero energy,	$\ker_{\calA^{d-2}(X)} P(0) = \{0\}$,\footnote{This requirement is generically satisfied.} 
		\item the requirement that the map in \cref{eq:Sommerfeld} is invertible.\footnote{This requirement is automatic if $P$ is symmetric. See \Cref{prop:Vasy_absorption}.}
	\end{itemize}
	For any $f\in \calS(X)$, the limiting resolvent output  $R(\sigma^2 \pm i0) f$ has the form 
	\begin{equation}
	R(\sigma^2 \pm i0) f = e^{\pm i \sigma (r +2^{-1} \frakm \log r)} u
	\label{eq:udef}
	\end{equation}
	for some function $u$ which is fully polyhomogeneous on $X^+_{\mathrm{res}}$. Here, $r$ is the reciprocal of a boundary-defining-function of $X$.\footnote{In the asymptotically Euclidean case, what we call ``$r$'' is really the Japanese bracket of the usual Euclidean radial coordinate. See the beginning of \S\ref{sec:op}.}
	\label{thm:A}
\end{theorem}

\begin{remark*}
	Roughly, this means that the resolvent output $R(\sigma^2 \pm i0) f$ has the form of an explicit oscillatory function times something \emph{non-oscillatory}, $u$. In fact, $u$ admits \emph{three} different asymptotic expansions:
	\begin{itemize}
		\item One expansion is in the region $r\sigma \gg 1$, and uniform as $\sigma\to 0^+$; this is the expansion at the face bf. As $\hat{r}=r\sigma\to\infty$, one sees the same decay rate that one sees in solutions of the Helmholtz equation for positive energy. This means $O(1/r^{(d-1)/2})$ decay, or more precisely $O(1/\hat{r}^{(d-1)/2})$ decay.
		\item One expansion is in powers of $1/(r+\sigma)$ with terms which are functions of $\sigma/(r+\sigma)$. This is the expansion at the face tf. The main index here comes from solving the PDE at $\sigma=0$, which we do in \S\ref{sec:0}. When $d=3$, this means $O(1/r)$ decay, same as at bf, \emph{because this is how the Green's function for the Laplacian decays}. In higher dimensions, 
		it is instead $O(1/r^{d-2})$. 
		
		The face tf only becomes essential \emph{at subleading order} in $\sigma$. As far as the first few terms in the low-energy expansion are concerned, you could blow this face down.
		\item  One expansion is in the region $r\sigma \ll 1$, essentially in powers of $\hat{r}=r\sigma$. This is the expansion at zf. 
		
		An important thing about the index sets appearing in this regime is the appearance of $\log \sigma$ terms at some point. This is the source of the main term in Price's law. The place in this paper where these log terms appear for the first time is \Cref{prop:Ntf_main}. Once they appear, they keep appearing.
	\end{itemize}
	Note that we have not one but \emph{two} low-energy expansions, one when $r=O(1/\sigma)$ and one when $r=\Omega(1/\sigma)$. The $d=3$ case is elaborated on in \S\ref{sec:Price}.
	See \cite[\S2]{Full} for a fuller unpacking of what this amounts to when studying the Schr\"odinger equation. 
\end{remark*}

Thus, the (limiting) resolvent output admits an atlas of asymptotic expansions;
see \Cref{thm:main} for a more precise version. The same holds for objects, like the spectral projector, that can be constructed out of the limiting resolvent.

Recall that polyhomogeneity is a generalization of smoothness which allows terms of the form $\rho^j (\log \rho)^k$ to appear in the $\rho\to 0^+$ asymptotic expansion, for any $(j,k)\in \bbC\times \bbN$, where $\rho=\rho_{\mathrm{f}}$ is a boundary-defining-function for some boundary hypersurface of our manifolds-with-corners. Thus, saying that $u$ is polyhomogeneous on $X^+_{\mathrm{res}}$ with index set $\calE=\calE_{\mathrm{f}}\subset \bbC\times \bbN$ is just saying that $u$ can be expanded in generalized Taylor series 
\begin{equation}
	u \sim \sum_{(j,k)\in \calE} u_{j,k} \rho^j (\log \rho)^k, \qquad u_{j,k}  = u_{j,k,\mathrm{f}}\in C^\infty(\mathrm{f}^\circ)
\end{equation}
at each boundary hypersurface $\mathrm{f} \in \{\mathrm{zf},\mathrm{tf},\mathrm{bf}\}$ of $X^+_{\mathrm{res}}$ (and that the Taylor series can be differentiated term-by-term, and that the different Taylor series at adjacent boundary hypersurfaces are consistent at the corners). Our proof of \Cref{thm:A} is constructive, in that it produces the terms in the asymptotic expansions and therefore allows them to be computed. We have not encapsulated the resultant formulae in a theorem. They are not clean.

One feature of the treatment of low-energy scattering here is an attention to partial \emph{multipole} expansions (with respect to some fixed boundary collar on $X$). Thus, when we prove the theorem above, we will actually show that the first few terms in the large radii expansions of $u$ involve only finitely many harmonics of the boundary Laplacian $\triangle_{\partial X}\in \operatorname{Diff}^2(\partial X)$. The quantitative meaning of ``few'' in the previous sentence depends on the decay rate of any terms in the given Schr\"odinger operator which break the symmetry associated to the preferred boundary collar. Proving this requires a bit of bookkeeping in the analysis of the null space of the given Schr\"odinger operator. This bookkeeping can be found in \S\ref{sec:0}, along with some additional motivation. The argument in that section is due to Melrose \cite{MelroseAPS}, in much greater generality, but there does not seem to exist in the literature an exposition with the level of detail required here.  

In his paper, Hintz uses a recursive algorithm similar to that used here for computing the first few terms in the low-energy asymptotics. Unfortunately, it is not entirely straightforward to extend his algorithm beyond the last order he considers. The reason it is not straightforward is that it is necessary to analyze the resolvent output on polyhomogeneous functions 
\begin{equation}
f \in \calA^{\calE,\calF,\calG}(X^+_{\mathrm{res}}) ,\quad \calE,\calF,\calG\subseteq \bbC\times \bbN,
\end{equation} 
where $\calE,\calF,\calG$ are appropriate index sets and $ \calA^{\calE,\calF,\calG}(X^+_{\mathrm{res}})$ is the set of functions admitting full asymptotic expansions on $X^+_{\mathrm{res}}$, with $\calE,\calF,\calG$ the index sets at bf, tf, and zf respectively; see \cite[Def.\ 2.13]{HintzPrice}.\footnote{Unfortunately, in our companion paper \cite{Full}, our notational conventions included listing index sets in the opposite order. Thus, in that paper, $\calA^{\calG,\calF,\calE}(X^{\mathrm{sp}}_{\mathrm{res}} \backslash \infty \mathrm{f})$ means what we call $\calA^{\calE,\calF,\calG}(X^+_{\mathrm{res}})$ here. We felt the former convention to be more natural when working with the additional face $\infty\mathrm{f} = X^{\mathrm{sp}}_{\mathrm{res}}\backslash X^+_{\mathrm{res}}$ (which corresponds to $E=\infty$), but in the present paper there is no such face (since we do not study the $E\to\infty$ limit). We decided to prioritize matching notational conventions with \cite{HintzPrice} rather than \cite{Full}.} Hintz only analyzes the case where $f$ is partially polyhomogeneous already on $[0,\infty)_\sigma \times X$. Here, we need to allow more singular behavior at the corner $\{0\}\times \partial X \subset [0,\infty)_\sigma \times X$.

The algorithm that we employ is actually a small modification of Hintz's. The details are described in \S\ref{sec:arg}, but let us emphasize that the key ideas are already contained in \cite{HintzPrice}. The main difference, besides producing full expansions, is that we use an inductive argument which does not worsen decay at each step, this being necessary to get sharp asymptotics at $\mathrm{bf}$. In contrast, the argument used by Hintz is simpler, but at the cost of producing a fictitious singularity at $\mathrm{bf}$.\footnote{The reason that Hintz produces a fictitious singularity at $\mathrm{bf}$ is that he applies $P(0)^{-1}$ globally, but $P(0)$ is only a good approximation for $P(\sigma)$ if $r\sigma \ll 1$, i.e.\ away from $\mathrm{bf}$. In contrast, whenever we apply $P(0)^{-1}$ below, there will be a cutoff localizing near $\mathrm{zf}$.} In addition, we work in greater generality, but this does not require new ideas, just more notation.

\begin{remark*}
	The work of Melrose--Sa Barreto \cite{MelroseSa} and Guillarmou--Hassell--Sikora \cite{GH1, GH2, GHK} is close in spirit to the work here. They use geometric microlocal tools to understand the Schwartz kernel $K_\pm(x,y;\sigma)$ of the limiting resolvents $R(\sigma^2 \pm i 0)$. While it should be possible in principle to get the asymptotics of 
	\begin{equation}
		R(\sigma^2 \pm i0)f (x) = \int_X K_\pm(x,y;\sigma) f(y) \dd \mathrm{Vol}_g( y)
	\end{equation}
	from those of the Schwartz kernel $K_\pm$ via a push-forward theorem, it is reasonably non-trivial. This is why \cite{HintzPrice} did not proceed in this way. One of the goals of the present work was to produce theorems about low-energy scattering in a ready-to-use form. 
\end{remark*}

\subsection{Method of proof and outline for rest of paper}

There are three additional sections of this paper besides those already mentioned: \S\ref{sec:bf}, \S\ref{sec:1}, \S\ref{sec:2}. These contain the asymptotic analyses at bf, zf, and tf, respectively. 
\begin{itemize}
	\item In \S\ref{sec:bf}, we discuss the limiting absorption principle for $\sigma$-dependent forcings with Schwartz behavior as $\sigma\to 0^+$. The low energy limit is singular, but the $O(\sigma^\infty)$ suppression of the forcing is sufficient to push through the asymptotic analysis without trouble. This is the easiest section of the paper.
	\item \S\ref{sec:1} involves the analysis at the boundary hypersurface $\mathrm{zf}\subset X^+_{\mathrm{res}}$, that is at \emph{exactly} zero energy. In this section, we solve the inhomogeneous PDE modulo terms which are suppressed by one order of the boundary-defining-function $\rho_{\mathrm{zf}}$ of the boundary hypersurface $\mathrm{zf}$. That is, we produce an $O(\rho_{\mathrm{zf}})$-quasimode, where by ``$O(E)$-quasimode'' we mean an approximate solution to the inhomogeneous PDE which solves the PDE modulo an error which is suppressed by a factor of $E$ \emph{relative} to the forcing we started with (or more precisely, relative to the function spaces in which the forcing was supposed to lie; if the forcing was anomalously small, then we are not demanding that the new error be even smaller). 
	\item Finally, \S\ref{sec:2} involves the analysis at $\mathrm{tf} \subset X^+_{\mathrm{res}}$. In it, we upgrade the previously constructed quasimode to an $O(\sigma^{\varepsilon})$-quasimode for some $\varepsilon>0$, i.e.\ an $O(\sigma^{0+})$-quasimode.  
	Such a quasimode solves the PDE modulo terms which are decaying as $\sigma\to 0^+$ (in the function spaces capturing the expected amount of decay at $\partial X$), \emph{uniformly} in $x\in X$. It is the analysis in this section which forms the heart of the analysis of low energy phenomena. 
	
	A key conceptual point is that even if the original quasimode is $O(\rho_{\mathrm{zf}}^\infty)$-good, the quasimode we end up with will typically only be $O(\rho_{\mathrm{zf}}^\varepsilon)$-good at $\mathrm{zf}$ for some $\varepsilon>0$. It is better at $\mathrm{tf}$, since it is $O(\sigma^{0+})$, but the quasimode may be worse at $\mathrm{zf}$. Worsening the quasimode at $\mathrm{zf}$ is a necessary price to pay for bettering the quasimode at $\mathrm{tf}$. For understanding the resolvent output $R(\sigma^2 \pm i0)f$, the relevant quasimodes must be uniformly good as $\sigma\to 0$. This is therefore a price we are forced to pay. In \cite{HintzPrice}, this step in the algorithm, when studying the $O(\sigma)$ term in the resolvent output, worsens an $O(\rho_{\mathrm{zf}})$-quasimode to an $O(\rho_{\mathrm{zf}}|\!\log \rho_{\mathrm{zf}}|) $-quasimode, because the new error term has a logarithmic singularity at zf. Its contribution to the resolvent ends up being an $O(\sigma^2 |\!\log \sigma|)$ term.
	This new logarithmic singularity is the ultimate source of Price's law.
	
	In the example of Price's law just discussed, we can take any $\varepsilon \in (0,1)$;\footnote{If $\partial X\neq \bbS^{d-1}$, then fractional powers of $\sigma$ can arise, which can necessitate taking some $\varepsilon \ll 1$.} our quasimode is just shy of $O(\sigma)$ good. We have only introduced a $\sigma\log \sigma$ singularity. This sort of singularity is very mild.
	As long as we do not introduce an $O(\rho_{\mathrm{zf}}^0) = O(1)$ error at $\mathrm{zf}$, the quasimode produced by combining \S\ref{sec:1}, \S\ref{sec:2} is better than what we started with, namely $u=0$, which, by definition, is just an $O(1)$-quasimode.  
\end{itemize}
These results are combined in \S\ref{sec:arg} in order to prove \Cref{thm:A}. The basic idea is to alternate applying the lemmas \S\ref{sec:1}, \S\ref{sec:2}. Some early applications of the latter lemmas are vacuous --- but they become essential eventually. After each round, we divide the resultant error by the appropriate power of $\sigma$ (namely just $\sigma$ in the asymptotically Euclidean case) and take the result as our new forcing.
This recursive algorithm, followed by asymptotic summation, produces an $O(\sigma^\infty)$ quasimode $w$. In the $\frakm=0$ case, this means
\begin{equation} 
	P(\sigma)(e^{\pm i \sigma r} w) = f+F
\end{equation} 
for $F = O(\sigma^\infty)$ (in suitable function spaces) as $\sigma\to 0^+$.

We then appeal to  \S\ref{sec:bf} to study $R(\sigma^2 \pm i0) F$; because $F=O(\sigma^\infty)$, we will also have $R(\sigma^2 \pm i0) F=O(\sigma^\infty)$ (in some suitable space), so the low-energy analysis of $R(\sigma^2 \pm i0) F$ is ``trivial,'' as we have already indicated. It can then be shown that 
\begin{equation}
R(\sigma^2 \pm i0) f = e^{\pm i \sigma r} w - R(\sigma^2 \pm i0) F.  
\end{equation}
So, if $w$ is constructed so as to be polyhomogeneous on $X^+_{\mathrm{res}}$, and $e^{\mp i \sigma r } R(\sigma^2 \pm i0)F$ is polyhomogeneous on $[0,\infty)_\sigma \times X$ (in fact, it will be), then the function $u$ defined by \cref{eq:udef} is polyhomogeneous on $X^+_{\mathrm{res}}$. 
In fact, because $e^{\mp i \sigma r } R(\sigma^2 \pm i0)F = O(\sigma^\infty)$ (in appropriate spaces), the low energy expansion of $u\sim w$ is just that of $w$. We know the low energy expansion of $w$ because \emph{that is how we built $w$}.

The $\frakm\neq 0$ case is similar, except a more refined ansatz than $e^{\pm i \sigma r}$ is required.

We have written \S\ref{sec:bf}, \S\ref{sec:1}, \S\ref{sec:2} in a high amount of generality. This is more generality than required in \S\ref{sec:arg}, but there may be some applications. For example, the lemmas in \S\ref{sec:bf} allow one to study forcings with slower decay than in \S\ref{sec:arg}. We expect that, when actually computing the asymptotic expansions whose existence is guaranteed by \Cref{thm:A}, the added generality will be useful for proving results with optimal index sets. 

We have seven appendices. As already mentioned, \S\ref{sec:0} is a self-contained exposition of multipole expansions for perturbations of the Laplacian on exactly conic manifolds, this being the model problem on $\mathrm{zf}\subset X^+_{\mathrm{res}}$.  This appendix was originally intended to be a standalone expository article, so it can actually be read independently of the rest of this paper. It is rather more pedagogical. 
In \S\ref{sec:tf}, we provide a similar treatment for a conjugated Schr\"odinger operator on an exact cone, this being the model problem at the boundary hypersurface $\mathrm{tf}\subset X^+_{\mathrm{res}}$. 
In \S\ref{sec:Mellin}, we review some elementary facts about the Mellin transform, for reference elsewhere in the appendices. This is included just to make the paper a bit more self-contained.
Next, \S\ref{sec:tedium} contains a few straightforward but tedious proofs of propositions stated in \S\ref{sec:op} that were felt to clutter the exposition.
Finally, \S\ref{sec:example} contains an ODE example with $\frakm\neq 0$ (the confluent hypergeometric equation). 

The final appendix, \S\ref{sec:index}, is a non-exhaustive index of notation. The reader should refer to this appendix whenever unfamiliar notation is encountered.

\subsection{A central mechanism}
The content of this paper is quite technical at times, so let us highlight a simple mechanism used in the algorithm below. This mechanism is not essential to the proof of polyhomogeneity, but it does explain which logarithmic terms are essential, and this is why we have made use of it. (As already stated, logarithmic terms are the origin of Price's law, so it is important to understand where they appear.)

Consider a function like $1/\langle r \rangle^{k}$ for $k\gg 1$. As far as the analysis at $\mathrm{zf}$ is concerned, this term, which is decaying at $\mathrm{zf}\cap \mathrm{tf}$, is amenable. 
However, 
\begin{equation} 
\frac{1}{r^{k}}  = \Big(\frac{\sigma}{\hat{r}}\Big)^{k}
\label{eq:misc_015}
\end{equation} 
for $\hat{r}=r\sigma$. When studying the problem at $\mathrm{tf}$, the $\hat{r}\to 0^+$ regime is of interest and corresponds to $\mathrm{zf}\cap\mathrm{tf}$. From this perspective, $1/\hat{r}^{k}$ is horrible, blowing up very badly. The factor of $\sigma$ in \cref{eq:misc_015} does not help us, since it commutes with the PDE. Conversely, at $\mathrm{zf}$,  a term like $r^k$ is horrible, but $r^k = (\hat{r}/\sigma)^k$ is decaying as $\hat{r}\to 0^+$ for fixed $\sigma>0$, so is amenable at $\mathrm{tf}$. The factor of $\sigma^{-k}$ does not hurt us. 

Evidently, the terms which are badly behaved with respect to the analysis within one of $\mathrm{zf},\mathrm{tf}$ are well-behaved with respect to the analysis within the other. This is why, when producing a $O(\sigma^\infty)$-quasimode $u$ to $Pu\approx f$, we need to juggle inverting the model operators $N_{\mathrm{zf}}=P(0)$ and $N_{\mathrm{tf}}$ at $\mathrm{zf}$ and $\mathrm{tf}$, respectively.

There is a certain range of $k\in \bbR$ for which the inversion $\smash{N_{\mathrm{zf}}^{-1}1/r^k}$ makes sense, and there is a certain range of $k$ for which the inversion $\smash{N_{\mathrm{tf}}^{-1} 1/r^k}$ makes sense.
By the preceding logic, the range for $\mathrm{zf}$ is an interval which is infinite to the right, and the range for $\mathrm{tf}$ is an interval which is infinite \emph{to the left}. These could conceivably be disjoint, in which case we would be in mild trouble if ever $k$ in the gap arose. Fortunately, this is not the case --- the ranges end up being 
\begin{itemize}
	\item $(0,\infty)$ for $\mathrm{zf}$ and $(-\infty,d-2)$ for $\mathrm{tf}$ if we are talking about the solution $u$
	\item $(2,\infty)$ for $\mathrm{zf}$ and $(-\infty,d)$ for $\mathrm{tf}$ if we are talking about the forcing $f$.
\end{itemize}
(The reason these are different for $u$ versus $f$ is because $\triangle$ is non-degenerately a weighted b-/regular singular operator at spatial infinity, not an unweighted one. It induces two orders of decay when applied. For example, if $u$ is a polynomial, then $\triangle u$ is a polynomial of degree $2$ lower.)
These intervals for $\mathrm{zf},\mathrm{tf}$ overlap \emph{because we are assuming $d\geq 3$}. 

We quickly explain why the intervals above are what they are.
The Laplacian is a second-order regular singular operator at spatial infinity. Its two indicial roots for s-waves are $0$ and $d-2$. This can be seen from two things:
the fact that constants are in the kernel of the exact Laplacian 
tells us that $0$ is an indicial root, and the fact that solutions of Poisson's equation typically decay like $1/r^{d-2}$ (Coulomb's law) tells us that $d-2$ is an indicial root.
The indicial root relevant to the amenable $k$ is the worse one, namely $0$. 
The point is that $P(0)^{-1} f$ makes sense as long as $f\in \calA^{2+}(X)$. Compare with Coulomb's law.

The model operator $\hat{N}_{\mathrm{tf}}$ capturing $P$ at $\mathrm{tf}$ modulo lower-order terms is regular singular at $\hat{r}=0$.\footnote{The hat is because we are factoring out a $\sigma^2$. Indeed, $P\approx \sigma^2 \hat{N}_{\mathrm{tf}}$ at $\mathrm{tf}$. So, really $N_{\mathrm{tf}}=\sigma^2 \hat{N}_{\mathrm{tf}}$ is the model operator.} This corresponds to the corner $\mathrm{zf}\cap \mathrm{tf}$, so the indicial roots are related to those of $P(0)$ at $r=\infty$. In fact, the former are just the negatives of the latter. For example, they are $0$ and $2-d$ for s-waves. The reason for the flip in sign is that 
\begin{itemize}
	\item the indicial roots $c$ of $P(0)\overset{\text{at zf}}{\approx} P$ are describing which powers $1/r^c$ are special vis-\`a-vis its theory, 
	\item  the indicial roots $c$ of $\hat{N}_{\mathrm{tf}} \overset{\text{at tf}}{\approx} P$ at $\hat{r}=0$ are describing which powers of $\hat{r}^c$ are special vis-\`a-vis its theory.
\end{itemize}
The worse indicial root between $0,2-d$ is $2-d$. So, if $k<d$, then $1/r^k$ can be fed into $\hat{N}_{\mathrm{tf}}^{-1}$. Moreover, we can get logarithmic terms in the small-$\hat{r}$ expansion --- this happens for s-waves if $k\geq 2$. These logarithmic terms have the form $\log(\hat{r})=\log \sigma+\log r$, or a power thereof. This is how logarithmic terms in $\sigma$ arise.

Attempting to produce the low-energy expansion of the resolvent output without using $N^{-1}_{\mathrm{tf}}$, one encounters the problematic situation 
\begin{equation} 
	f\notin \calA^{2+}(X).
\end{equation} 
This necessitates applying the analysis at $\mathrm{tf}$. 
Actually, what we do is preemptively apply $N_{\mathrm{tf}}^{-1}$ \emph{in the previous step} so as to solve away the terms in the forcing 
\begin{equation} 
	f_{\mathrm{prev}}\in \calA^{2+}(X)
\end{equation} 
of that step giving rise to the problematic terms in $f$. 
We extract the terms $1/r^k$ for $k\in (2,d)$ from $f_{\mathrm{prev}}$ and then apply $N_{\mathrm{tf}}^{-1}$ to those. Note that these are in the amenable range $(-\infty,d)$, but we might generate  log terms when we apply $N_{\mathrm{tf}}^{-1}$. 
In the asymptotically Euclidean case, this usually happens. This is the origin of Price's law. 

We focused on s-waves in the discussion above, but the situation regarding other partial waves is even better, because the gap between the indicial roots associated with such waves is bigger, and the amenable intervals increase. For example, on $\bbR^3$, the amenable interval widens by $1$ on each side every time the azimuthal quantum number is incremented.

In situations where waves decay rapidly (e.g. odd-dimensional spacetimes with Schwartz metric perturbations), an algebraic conspiracy happens so as to prevent $f\notin \calA^{2+}(X)$ from ever arising. This means that one never needs to resort to applying $N_{\mathrm{tf}}^{-1}$, and so $\log \sigma$ terms never appear. We will see an example of such an algebraic conspiracy when we discuss the asymptotically $\bbR^3$ case, in \S\ref{sec:Price}. It just has to do with comparing the indicial roots of $P$ at $\mathrm{tf}$ versus at $\mathrm{bf}$, and whether those differ by an integer.

\label{rem:mechanism}

\section{Geometric setup and the admissible operators}
\label{sec:op}

Our setup here is a special case of that considered in \cite{Vasy}. Since we are concerned with asymptotics, we must assume that the coefficients of the PDE under consideration admit at least partial asymptotic expansions themselves, and we assume some additional symmetry of the $O(r^{-1})$ and $O(r^{-2})$ terms, but otherwise the setup is completely analogous. Before proving the main results of this section, we present the details of the setup. 

First, we discuss notation.
Let $X$ denote a connected, smooth manifold-with-boundary equipped with a boundary collar, by which we mean an embedding of the cylinder $\dot{X} = [0,1) \times \partial X$ into $X$
such that $(0,\theta)\mapsto \theta$ for all $\theta\in \partial X$. So, we can identify $\dot{X}$ with a subset of $X$. Moreover, we assume that the coordinate function $(\rho,\theta)\mapsto \rho$ on $\dot{X}$ extends to an element of $C^\infty(X;\bbR^{\geq 0})$, which we can assume vanishes only at $\partial X$. Fix such an extension, which we denote for the rest of the paper as $\rho$. Thus, $\rho$ is a boundary-defining-function for the boundary of $X$. 
In order to facilitate the comparison between the asymptotically conic setting and the exact Euclidean setting, define $r = \rho^{-1}$. 

The exact Euclidean case is when 
\begin{equation}
	X = \overline{\bbR^d} =\bbR^d\cup \infty \bbS^{d-1}
\end{equation}
is the radial compactification of $\bbR^d$ and when the boundary collar is the map $\dot{X}= [0,1)\times \bbS^{d-1}\hookrightarrow X$ given by $(\rho,\theta)\mapsto \rho^{-1} \theta$. Inverting, this just means that 
\begin{equation}
	\rho(\bfx) = \lVert \bfx \rVert^{-1}, \qquad \theta = \bfx/\lVert \bfx \rVert 
\end{equation}
for all $\bfx\in \bbR^d$ with $\lVert \bfx \rVert >1$. In other words, away from the origin $r=\rho^{-1}$ is the usual Euclidean radial coordinate and $\theta = \hat{\bfx}$ is the unit vector in the direction of $\bfx$. We can extend $\rho$ to an element of $C^\infty(X)$. Then, $r=1/\rho\in C^\infty(\bbR^d)$ is only given by the usual formula $r(\bfx) = \lVert \bfx \rVert$ outside of some neighborhood of the origin, but this should not cause any confusion.

\begin{remark*}
	We emphasize the choice of boundary collar (and therefore boundary-defining function) because the forms of the multipole expansions below depend on it. 
	There are two reasons for this: 
	\begin{enumerate}[label=(\Roman*)]
		\item Consider the fact that in the multipole expansion from Euclidean electrostatics the $1/r^2$ term involves only the spherical harmonics $Y^1_m$, not $Y^0_0=1$ (assuming the charge distribution is sufficiently rapidly decaying at infinity). In other words, the subleading term has no s-wave component. If we were to rewrite the expansion in terms of $r+1$ instead of $r$, this would no longer be true. So, the choice of boundary-defining function matters.
		\item More seriously, the presence of logarithmic terms depends on the presence of terms in the PDE that break the conic symmetry of the problem (meaning spherical symmetry in the Euclidean case). This is the reason there are no logarithmic terms in the multipole expansion from Euclidean electrostatics. But to even make sense of what it means to ``break the spherical symmetry,'' we need to fix a boundary collar.
	\end{enumerate}
A more thorough discussion of these points may be found in the introduction to \S\ref{sec:0}.
If the reader is willing to settle for a theorem involving sub-optimal index sets, then the choice of boundary collar may be forgotten.  
\end{remark*}

\subsection{Some standard notation} In this subsection, we review some objects defined by Melrose \cite{MelroseCorners, MelroseAPS, MelroseSC, MelroseGeometric}.

Let $\calV_{\mathrm{b}}(X)$ denote the set of smooth vector fields on $X$ (which are smooth all the way up to and including the boundary and) tangent to the boundary. These are the ``b-vector fields.'' Similar notation and terminology is used for tensors, algebras of differential operators, et cetera. For example, for $m\in \bbR$, $\operatorname{Diff}^m_{\mathrm{b}}(X)$ is the set of linear combinations over $C^\infty(X)$ of compositions of $\leq m$ b-vector fields. Then, \begin{equation} 
\operatorname{Diff}^{m,\ell}_{\mathrm{b}}(X) = \rho^{-\ell} \operatorname{Diff}^m_{\mathrm{b}}(X)
\end{equation} 
denotes the set of weighted b-differential operators. 

The set $\calV_{\mathrm{sc}}(X) = \rho \calV_{\mathrm{b}}(X)$ is the set of ``sc''-vector fields. These are the sections of a smooth vector bundle over $X$, which is called ${}^{\mathrm{sc}}T X$. Dualizing, one gets an associated cotangent bundle ${}^{\mathrm{sc}} T^* X$ which is canonically the usual cotangent bundle over the interior. Again, we can form the algebra $\operatorname{Diff}_{\mathrm{sc}}(X)$ of sc-differential operators, and notation similar to that used in the b- case can be employed.

We use $\calA^{\bullet}(M)$ to denote the sets of partially polyhomogeneous functions on a manifold-with-corners $M$, where the $\bullet$ contains a list of pairs $(\calE,\alpha)$, one for each boundary hypersurface $\mathrm{f}$ of $M$, where $\calE$ is the index set of the polyhomogeneous expansion at $\mathrm{f}$ and $\alpha\in \bbR \cup \{\infty\}$ is the order of the merely conormal error. When $\alpha=\infty$, in which case one has full polyhomogeneity at $\mathrm{f}$, we simply write $\calE$ in place of $(\calE,\infty)$. Likewise, if $\calE=\varnothing$, we simply write $\alpha$. For example, 
\begin{equation}
\calA^{(\calE,\alpha), \beta, \calK }(X^+_{\mathrm{res}}) 
\end{equation}
denotes the set of functions which are partially polyhomogeneous at $\mathrm{bf}$ with index set $\calE$ and conormal error $\alpha$, merely conormal at $\mathrm{tf}$ with order $\beta$, and fully polyhomogeneous at $\mathrm{zf}$ with index set $\calK$. 

Bundles over $\dot{X}$, function spaces on $\dot{X}$, et cetera are denoted similarly to the corresponding objects on $X$. While we will not state the definitions of such objects explicitly, it should be noted that we never require uniformity as $\rho\to 1^-$. 
For example, 
\begin{equation}
	\calA^0(\dot{X}) =\{a\in C^\infty(X^\circ) : Q a \in L^\infty_{\mathrm{loc}}(\dot{X}) \text{ for all }Q \in \operatorname{Diff}_{\mathrm{b}}^{\infty,0}(\dot{X}) \}. 
\end{equation}
Thus, for $a\in \calA^0(\dot{X})$, $\chi Q a\in L^\infty(\dot{X})$ for all $\chi\in C_{\mathrm{c}}^\infty(\dot{X})$, including those that do not vanish on $\partial X$. (So, $L^\infty_{\mathrm{loc}}(\dot{X})$ is smaller than $L^\infty_{\mathrm{loc}}(\dot{X}^\circ)$.) 
We will also add a `c' subscript to denote compact support, e.g.\ $\calA^0_{\mathrm{c}}(\dot{X})$. Note that this allows functions to be nonvanishing near $\rho=0$ (because $X$, with its boundary included, is compact), just not near $\rho=1$. So, in the Euclidean case, $C_{\mathrm{c}}^\infty(\dot{X})$ consists of smooth functions on $\bbR^d$ that extend to the radial compactification and that vanish on a neighborhood of the unit ball centered at the origin.

In this section, we will also work with the sc- and b- $L^2$-based Sobolev spaces $H_{\mathrm{sc}}^{r,\ell}(X),H_{\mathrm{b}}^{r,\ell}(X)$. The parameter $r\in \bbR$ is the differential order and $\ell\in \bbR$ is the amount of decay.
We refer to \cite{VasyLA} for details when $r\notin \bbN$. When $r\in \bbN$, 
\begin{equation}
H_{\mathrm{sc}}^{r,\ell}(X) = \{u\in \calD'(X^\circ) : \operatorname{Diff}_{\mathrm{sc}}^{r,\ell}(X) u \in L^2(X) \}
\end{equation}
and similarly for $H_{\mathrm{b}}^{r,\ell}$. In the Euclidean case, these are just the ordinary Sobolev spaces. 
One convention worth noting is that we follow Vasy in indexing the b-Sobolev spaces so that 
\begin{equation} 
H_{\mathrm{b}}^{0,0}(X)=H_{\mathrm{sc}}^{0,0}(X)=L^2(X,g),
\end{equation} 
where $g$ is an asymptotically conic metric (see below).
This differs from the other commonly used convention where $\smash{H_{\mathrm{b}}^{0,0}} = L^2(X,g_{\mathrm{b}})$ for $g_{\mathrm{b}}$ an asymptotically cylindrical metric (a b-metric). The different conventions just differ by a shift in the decay index.

\begin{remark}
	In \S\ref{sec:Price}, \S\ref{sec:0}, \S\ref{sec:tf} we will work with partially polyhomogeneous spaces $\calA^{\calE,\alpha}$ for $\calE$ a pre-index set. (See \S\ref{sec:index} for the definition of ``pre-index set.'') For example, if $\alpha$ is finite, we say that $f\in  \calA^{\calE,\alpha}(X)$ if it is smooth in the interior and 
	\begin{equation}
	\exists f_{j,k}\in C^\infty(\partial X_\theta)\text{ s.t. }f - \sum_{(j,k)\in \calE, \Re j \leq \alpha } f_{j,k}(\theta) \rho^j (\log \rho)^k \in \calA^\alpha(\dot{X}) .
	\end{equation}
	Usually, one works with $\calA^{\calE,\alpha}$ for $\alpha$ an index set, which has the advantage that it does not depend on the choice of boundary-defining-function or boundary collar. If $\calE$ is only a pre-index set, then $\calA^{\calE,\alpha}$ does depend on these choices. Unfortunately, in order to get sharp statements regarding the index sets involved in multipole expansions, it is necessary to fix a boundary collar and consider pre-index sets. We have avoided using this in the proof of our main theorems.
\end{remark} 

\subsection{Metric}
\label{subsec:metric}

Let $g$ denote a smooth ``sc-metric'' on $X$, meaning a Riemannian metric on $X^\circ$ of the form $g \in C^\infty(X; \operatorname{Sym}^2 {}^{\mathrm{sc}} T^* X)$, i.e.\ a smooth section of the symmetric product of two copies of the sc-cotangent bundle. 
We assume that the boundary collar has been chosen such that $g$ differs from the exactly conic metric $g_0 = \mathrm{d} \rho^2 + \rho^2 h$ on $\dot{X}$ by a well-behaved error:
\begin{equation}
	g -g_0 \in \rho C^\infty(\dot{X}; \operatorname{Sym}^2 {}^{\mathrm{sc}} T^* \dot{X}).
	\label{eq:misc_metrick}
\end{equation}
Thus, $L^2(X,g_0)=L^2(X,g)$ at the level of sets. (One sometimes allows an $\rho^{1+\delta} S^0(\dot{X}; \operatorname{Sym}^2 {}^{\mathrm{sc}} T^* \dot{X})$ term on the right-hand side of \cref{eq:misc_metrick}, for some $\delta>0$, but we will not do so here.)
If $g-g_0$ is decaying faster than \cref{eq:misc_metrick} requires, then it is worth keeping track of. Moreover, the $\mathrm{d} r^2$ terms in $g-g_0$ have a different effect than the other $d^2-1$ terms, so we keep track of these separately. 
Let $\gimel\in \bbN$, $\gimel_0\in \bbN_+$ be such that 
\begin{equation}
g-g_0 \in \rho^{1+\gimel} C^\infty([0,1)_\rho ) \mathrm{d} r^2 + \rho^{1+\gimel_0} C^\infty(\dot{X}; \operatorname{Sym}^2 {}^{\mathrm{sc}} T^* \dot{X})
\label{eq:g}
\end{equation}
We assume without loss of generality that $\gimel\leq \gimel_0$. We may also allow $\gimel,\gimel_0$ to be infinite, in which case the corresponding terms in $g-g_0$ are Schwartz. Notice that $\gimel$ can be $0$, but $\gimel_0\geq 1$; this is because we require the subleading term in the metric to preserve spherical symmetry. Asymptotically Schwarzschild/Kerr spaces have this property.

Let $\triangle_g$ denote the positive semidefinite Laplace--Beltrami operator on $(X,g)$. Then: 

\begin{proposition}
	Given \cref{eq:g}, the Laplacian
	$\triangle_g$ differs from the exactly conic Laplacian 
	\begin{equation}
	\triangle_{g_0} = - \frac{\partial^2}{\partial r^2} - \frac{d-1}{r} \frac{\partial}{\partial r} + \frac{1}{r^2} \triangle_{g_{\partial X}} \in \rho^2\operatorname{Diff}^2_{\mathrm{b}}(X)
	\label{eq:exact_Laplacian_formula}
	\end{equation}
	by an element of $\rho^{3+\gimel} \operatorname{Diff}^2_{\mathrm{b}}([0,1)_\rho) +  \rho^{3+\gimel_0} \operatorname{Diff}^2_{\mathrm{b}}(X)$.
	\label{prop:metric}
\end{proposition}
See \S\ref{sec:tedium} for the proof.

A particularly important term in $\triangle_g$ is the subleading, i.e.\ $O(1/r)$, term in the coefficient of $\partial_r^2$. So, let $\frakm\in \bbR$ be such that 
\begin{equation}
	\triangle_g - \triangle_{g_0} - \frac{\frakm}{r} \frac{\partial^2}{\partial r^2} \in \rho^{\max\{4,3+\gimel\} } \operatorname{Diff}^2_{\mathrm{b}}([0,1)_\rho) +  \rho^{3+\gimel_0} \operatorname{Diff}^2_{\mathrm{b}}(X).
	\label{eq:misc_018}
\end{equation}
The reason for calling the coefficient $\frakm$ is that, in the proof of Price's law in \cite{HintzPrice}, this term is proportional to the black hole's mass. Note that $\gimel\neq 0\Rightarrow \frakm=0$.

\subsection{Operators}
\label{subsec:op}

Let $P = \{P(\sigma)\}_{\sigma\geq 0 }$ denote a family of differential operators $P(\sigma) \in \operatorname{Diff}^2(X^\circ)$ of the form 
\begin{equation}
	P(\sigma) = \triangle_g + 2i \sigma (1-\chi) \frac{\partial}{\partial r} + \frac{i\sigma(d-1)}{r} + L + \sigma Q + \sigma^2 R,
	\label{eq:Pform}
\end{equation}
where 
\begin{itemize}
	\item $\chi\in C_{\mathrm{c}}^\infty(X^\circ)$ satisfies $\operatorname{supp}(1-\chi) \Subset \dot{X}$, so that $(1-\chi)\partial_r$ is a well-defined vector field on $X$,
	\item $\rho^{-3} L, \rho^{-2} Q \in \operatorname{Diff}^1_{\mathrm{b}}(X)$, and
	\item $R \in \rho^2 C^\infty(X)$. 
\end{itemize}

Some aspects of the limiting absorption principle are simplified by assuming that $P$ is a symmetric operator on $C_{\mathrm{c}}^\infty(X^\circ)$ with respect to the $L^2(X,g)$-inner product, so that $P=P^*$.  However, this rules out applications to wave propagation on stationary \emph{but non-static} black hole spacetimes. Indeed, in Price's law, the relevant $P(\sigma)$ is what Hintz calls $\widehat{\square}_g(\sigma)$ in \cite[Eq.\ 2.4]{HintzPrice}. This is defined by throwing out the time derivatives in $e^{i\sigma t_*} \square_g e^{-i \sigma t_*}$, where $t_*$ is a certain coordinate on the spacetime. Viewed as a stationary differential operator on spacetime, it is symmetric with respect to the spacetime metric. However, it is not symmetric when viewed as an operator on the spatial slices (with respect to the natural inner product), unless the angular momentum parameter $\fraka$ of the black hole is zero.
So, we will not require $P=P^*$.

We can find $\beth,\beth_0,\beth_1,\beth_2,\beth_3,\beth_4 \in \bbN\cup\{\infty\}$, which may be zero,  such that the following conditions hold:
\begin{align}
	\rho^{-3} (\triangle_{g_0} - \triangle_g) &\in  \rho^{\beth}\; \operatorname{Diff}^2_{\mathrm{b}}( [0,1)_\rho) + \rho^{\beth_0} \operatorname{Diff}^2_{\mathrm{b}}(\dot{X}),  \label{eq:misc_a} \\
	\rho^{-3} L &\in \rho^{\beth}\, \operatorname{Diff}^1_{\mathrm{b}}( [0,1)_\rho) + \rho^{\beth_0} \operatorname{Diff}^1_{\mathrm{b}}(\dot{X}), \label{eq:misc_b} \\
	\rho^{-2} Q &\in  \rho^{\beth_1}\! \operatorname{Diff}^1_{\mathrm{b}}([0,1)_\rho )+ \rho^{\beth_2} \operatorname{Diff}^1_{\mathrm{b}}(\dot{X}) 
\end{align}
and
\begin{equation}
	\rho^{-2} R \in \rho^{\beth_3} C^\infty([0,1)_\rho) + \rho^{\beth_4} C^\infty(X).  
\end{equation}
By \Cref{prop:metric}, the first of these, \cref{eq:misc_a}, will be satisfied as long as $\beth\leq \gimel$ and $\beth_0 \leq \gimel_0$, where the constants $\gimel\geq0$, $\gimel_0\geq 1$ are as in \S\ref{subsec:metric}. We may assume without loss of generality that $\beth\leq \beth_0$, $\beth_1\leq \beth_2$, and $\beth_3\leq \beth_4$. 

It will be useful to relate the class of operators considered here to the conjugations of spectral families of Schr\"odinger operators. Given any family $\{O(\sigma)\}_{\sigma \geq 0} \subset \operatorname{Diff}^2(X^\circ)$, let 
\begin{equation}
\hat{O}(\sigma) = e^{i\sigma r} r^{i\sigma \frakm/2} \circ O(\sigma)\circ e^{-i\sigma r} r^{-i\sigma \frakm/2}, \qquad \check{O}(\sigma) = e^{-i\sigma r} r^{-i\sigma \frakm/2} \circ O(\sigma)  \circ e^{i\sigma r} r^{i\sigma \frakm/2}, 
\label{eq:misc_025}
\end{equation}
where by $e^{\pm i \sigma r}$ we mean the multiplication operator $\bullet \mapsto e^{\pm i \sigma r} \bullet$, and likewise for $r^{\alpha}$. 
Then, it is a matter of computation to prove:
\begin{proposition}
	Let $\beth=\gimel,\beth_0=\gimel_0$.
	\begin{enumerate}[label=(\Roman*)]
		\item Suppose that $P(\sigma)$ has the form described above. Then, $\check{P}(\sigma)$ has the form 
		\begin{equation}
		\check{P}(\sigma) = P_0 - \sigma^2 + \sigma P_1 + \sigma^2 P_2 
		\label{eq:Pcheck_form}
		\end{equation}
		for $P_0 = \triangle_g + L$ and $P_1,P_2$ such that $P_1 \in \rho^{2+\min\{\beth,\beth_1\} }\operatorname{Diff}^1_{\mathrm{b}}([0,1)_\rho ) + \rho^{2+\min\{\beth_0,\beth_2\}} \operatorname{Diff}^1_{\mathrm{b}}(X)$,
		$P_2 \in \rho^{1+\min\{\beth,\beth_1,1+\beth_3\} }C^\infty([0,1)_\rho ) + \rho^{1+\min\{\beth_0,\beth_2,1+\beth_4\}} C^\infty(X)$.
		\item Conversely, suppose that we are given some Schr\"odinger operator $P_0 = \triangle_g + L$ for  $L$ as in \cref{eq:misc_b}. Then, defining $P_0(\sigma) = P_0-\sigma^2$ and 
		\begin{equation} 
		P(\sigma) = \hat{P}_0(\sigma),
		\end{equation} 
		$P(\sigma)$ has the form described above, with $\beth_1,\beth_3=\max\{0,\gimel-1\}$, $\beth_2=\gimel_0$, and $\beth_4= \gimel_0-1\geq 0$. 
		\label{it:Schrodinger_form_II}
	\end{enumerate}
	\label{prop:Schrodinger_form}
\end{proposition}
The proof is contained in \S\ref{sec:tedium}.

So, the spectral families of Schr\"odinger operators fall into the framework here. 

For applications to wave propagation on non-ultrastatic spacetimes, it is important to allow greater generality. This is why we do not assume that our operator $P(\sigma)$ arises from  \Cref{prop:Schrodinger_form}.\ref{it:Schrodinger_form_II}, instead allowing any family of operators of the form \cref{eq:Pform}.

\begin{lemma} For each $k\in \bbN$, 
	$\operatorname{Diff}_{\mathrm{b}}^k(X)\subseteq \operatorname{Diff}^k_{\mathrm{b}}(X^+_{\mathrm{res}})$.
\end{lemma}
\begin{proof}
	Any b-vector field on $X$ lifts to a smooth vector field on $X^+_{\mathrm{res}}$, and the lift is tangent to the boundary of $X^+_{\mathrm{res}}$. So, $\calV_{\mathrm{b}}([0,\infty)_\sigma \times X) \subseteq \calV_{\mathrm{b}}(X^+_{\mathrm{res}})$. Thus, 
	\begin{equation}
		\operatorname{Diff}^k_{\mathrm{b}}(X) \subseteq \operatorname{Diff}^k_{\mathrm{b}}([0,\infty)_\sigma \times X) \subseteq \operatorname{Diff}^k_{\mathrm{b}}(X^+_{\mathrm{res}}) 
	\end{equation}
	for any $k\in \bbN$. 
\end{proof}

Consequently, $P \in \operatorname{Diff}_{\mathrm{b}}^{2,-1,-2,0}(X^+_{\mathrm{res}})$, where the orders are the differential order, the decay order at $\mathrm{bf}$, the decay order at $\mathrm{tf}$, and the decay order at $\mathrm{zf}$. 

\subsection{The Sommerfeld problem}
\label{subsec:Sommerfeld}

Assuming that $P(\sigma) = \hat{P}_0(\sigma)$ for $P_0(\sigma)=P_0 - \sigma^2$ the spectral family of a Schr\"odinger operator of the form described in \Cref{prop:Schrodinger_form}, let us recall some basic solvability theory for $P(\sigma)$.

First, recall that, since $P_0$ is essentially self-adjoint on $\calS(X)$ with respect to the $L^2(X,g)$ inner product (this follows easily from the theory of deficiency indices, owing to the fact that $\operatorname{ran}(P_0\pm i)^\perp\subset L^2(X,g)$ is a subset of $\ker_{\calS'}(P_0\mp i)$, but $\ker_{\calS'}(P_0\mp i)=\ker_{\calS}(P_0\mp i)$ consists only of Schwartz functions by ellipticity in the sc-calculus, and $\ker_{\calS}(P_0\mp i)$ is empty by the usual symmetry argument), we have, for all $E\in \bbC\backslash \bbR$, a well-defined two-sided inverse 
\begin{equation}
R(E) \in \Psi_{\mathrm{sc}}^{-2,0}(X).
\end{equation}
In other words, for any $f\in \calS'(X)$, the unique solution $u\in \calS'(X)$ to $P_0u-Eu=f$ is given by $u=R(E)f$. Every other solution of the PDE fails to be tempered.
Moreover, $R(E)$ is a pseudodifferential operator in the Parenti--Shubin--Melrose ``sc-calculus'' which regularizes by two orders and does not worsen decay, the set of such operators being $\Psi^{-2,0}_{\mathrm{sc}}$.

If $E\in \bbR$, it may be the case that $E$ lies in the spectrum of $P_0$, in which case the resolvent $R(E)$ will not be well-defined. In fact, if $E\geq 0$, then it must be in the spectrum. For $E>0$, in lieu of the resolvent one has the \emph{limiting absorption principle}.
One particularly weak form of this principle states that, for any $E>0$, each of the two limits
\begin{equation}
R(E\pm i0) = \lim_{\epsilon\to 0^+} R(E\pm i \epsilon) : \calS(X)\to \calS'(X) 
\end{equation}
exists in the strong operator topology. That is, for any $f\in \calS(X)$, the Schwartz functions $R(E\pm i\epsilon)f$ converge as $\epsilon \to 0^+$ to some tempered distributions in the topology of $\calS'(X)$. Moreover, the ``limiting resolvent'' $R(E\pm i0)$ is a right inverse to $P_0 - E = P_0(\sigma)$ for $\sigma=E^{1/2}$, so $P_0(\sigma)R(E\pm i0) f = f$ for all $f\in \calS(X)$; it produces solutions to the PDE. 

The particular solution $u = R(\sigma^2\pm i0) f$ to $P_0(\sigma)u=f$ produced by the limiting resolvent is distinguished in the set of all solutions by its oscillatory behavior $\sim \exp(\pm i \sigma r)$ at spatial infinity (the ``Sommerfeld radiation condition'').  
Thus, the \emph{Sommerfeld problem}, that of solving $P_0u - Eu=f$ for $u$ with ``outgoing'' oscillatory behavior at infinity is well-posed and solved by the limiting resolvent.

A much more precise version of the limiting absorption principle is contained in \cite{VasyLA}. For instance, it, when applied to the case at hand (and stated without assuming that $P$ is the spectral family of a Schr\"odinger operator), contains:
\begin{proposition}
	Suppose that $P$ satisfies the hypotheses in the previous subsection.
	Let $\ell,r\in \bbR$ satisfy the inequalities $\ell<-1/2$ and $r>-1/2-\ell$.
	Then, for each $\sigma>0$, 
	\begin{equation}
	P(\sigma) : \{u\in H_{\mathrm{b}}^{r,\ell}(X): P(\sigma) u \in  H_{\mathrm{b}}^{r,\ell+1}(X) \}\to H_{\mathrm{b}}^{r,\ell+1}(X)
	\label{eq:Sommerfeld}
	\end{equation}
	is Fredholm. If $P=P^*$, then it is invertible. 
	Moreover, if $K\Subset \bbR^+$, the operator norm of $P(\sigma)^{-1} : H_{\mathrm{b}}^{r,\ell+1}(X)\to H_{\mathrm{b}}^{r,\ell}(X)$ is uniformly bounded for $\sigma\in K$.
	\label{prop:Vasy_absorption}
\end{proposition}

The (almost immediate) proof from \cite[Thm.\ 1.1]{VasyLA} is contained in \S\ref{sec:tedium}. Henceforth, we will add to our assumptions that \cref{eq:Sommerfeld} is invertible. As stated in the proposition, this holds whenever $P=P^*$.

\begin{remark} Part of \cite[Thm.\ 1.1]{VasyLA} is that, when $P(\sigma)$ arises from conjugating the spectral family of a Schr\"odinger operator, the inverse of \cref{eq:Sommerfeld} is given by 
	\begin{equation}
	P(\sigma)^{-1} = e^{i\sigma r}r^{i\sigma\frakm/2} R(\sigma^2-i0)e^{-i\sigma r} r^{-i\sigma \frakm/2}
	\label{eq:resolv_def}
	\end{equation}
	on $\calS(X)$.\footnote{Actually, Vasy does not conjugate by $r^{i\sigma \frakm/2}$, but since multiplication by this function is an isomorphism of the Sobolev spaces appearing in Vasy's theorem, this additional conjugation does not change the result.} Thus, \Cref{prop:Vasy_absorption} directly defines the conjugated limiting resolvent via the problem it solves. 
	
	The connection with the Sommerfeld radiation condition is that, if $u_0 \in H_{\mathrm{b}}^{r,\ell}(X)$ for high $r$, then $u_0$ is non-oscillatory at infinity (the b-Sobolev regularity implying that such oscillations are suppressed by at least $\ell$ powers of $\rho$), which means that 
	\begin{equation}
	R(\sigma^2 -i0) f = e^{-i\sigma r}r^{-i\sigma\frakm/2} P(\sigma)^{-1}  (e^{i \sigma r} r^{i\sigma\frakm/2} f)
	\end{equation}
	has the correct oscillatory behavior at infinity.
\end{remark}

Even when $P(\sigma)$ does not arise from conjugating the spectral family of the Schr\"odinger operator, we \emph{define} $R(\sigma^2-i0)$ by \cref{eq:resolv_def}.

Similar statements apply to $R(\sigma^2 +i0)$.

Here, we prefer to work with the conormal function spaces $\calA^\bullet(X)$ instead of the b-Sobolev spaces $H_{\mathrm{b}}^\bullet(X)$. These are related by 
\begin{equation}
H_{\mathrm{b}}^{\infty,\ell}(X) \subseteq \calA^{\ell+d/2}(X) \subseteq H_{\mathrm{b}}^{\infty,\ell-}(X) = \bigcap_{\varepsilon>0} H_{\mathrm{b}}^{\infty,\ell-\varepsilon}(X),\quad \ell\in \bbR, 
\label{eq:Sobolev}
\end{equation}
which is a consequence of the Sobolev embedding theorems. The shift in the decay order comes from the convention that the indexing of $\calA^\bullet$ is based on $L^\infty$, while that of $H_{\mathrm{b}}^\bullet$ is based on $L^2$. Specifically, $H_{\mathrm{b}}^{0,0}(X)=L^2(X,g)$, so 
\begin{equation}
\calA^{d/2}(X)\subset \rho^{d/2} L^\infty(X) \subset \rho^{-\varepsilon} L^2(X,g)= H_{\mathrm{b}}^{0,-\varepsilon}(X).
\end{equation}
This is the upper bound in \cref{eq:Sobolev}, except without talking about derivatives.

It therefore follows that, for all $\sigma>0$, 
\begin{equation} 
P(\sigma)^{-1} : \calA^{\alpha+1}(X) \to \calA^{\alpha-}(X)
\label{eq:misc_030}
\end{equation} 
for all $\alpha< (d-1)/2$. 

We will need to know that $P(\sigma)^{-1}$ varies smoothly in $\sigma$ as a map $\calA^{\alpha+1}(X) \to \calA^{\alpha-}(X)$ in the sense that, whenever $f\in \calA^{\alpha+1}(X)$, then 
\begin{equation}
P(\sigma)^{-1} f \in C^\infty(\bbR^+_\sigma;\calA^{\alpha-}(X)).
\label{eq:misc_Pfs}
\end{equation}
The argument showing this is standard, so we only illustrate the idea without covering all of the details. The starting point is the usual resolvent identity
\begin{multline}
P(\sigma_2)^{-1} - P(\sigma_1)^{-1} = P(\sigma_1)^{-1} P(\sigma_1) P(\sigma_2)^{-1} - P(\sigma_1)^{-1} P(\sigma_2) P(\sigma_2)^{-1} \\ = P(\sigma_1)^{-1}(P(\sigma_1) - P(\sigma_2))P(\sigma_2)^{-1}. 
\label{eq:misc_0ll}
\end{multline}
(The manipulations are justified because the compositions $P(\sigma_\bullet)^{-1}P(\sigma)P(\sigma_\bullet)^{-1}$ in the middle of \cref{eq:misc_0ll} are well-defined, mapping $\calA^{\alpha+1}(X)\to \calA^{\alpha-}(X)$.) We may write $P(\sigma_1) - P(\sigma_2) = (\sigma_1-\sigma_2) T(\sigma_1,\sigma_2) $ for some smooth two-parameter family of $T\in \operatorname{Diff}^{1,-1}_{\mathrm{b}}(X)$. So, \cref{eq:misc_0ll} says
\begin{equation}
P(\sigma_2)^{-1} - P(\sigma_1)^{-1} = (\sigma_1-\sigma_2) P(\sigma_1)^{-1} T(\sigma_1,\sigma_2) P(\sigma_2)^{-1}.
\label{eq:misc_oyy}
\end{equation}
The composition $P(\sigma_1)^{-1} T(\sigma_1,\sigma_2) P(\sigma_2)^{-1}$ maps $\calA^{\alpha+1}(X)\to \calA^{\alpha-}(X)$, and each seminorm of the codomain is, for $\sigma_1$ near $\sigma_2$, uniformly bounded in terms of finitely many seminorms of the domain (by the uniformity clause in \Cref{prop:Vasy_absorption}). Thus, as $\sigma_2 \to \sigma_1$, the right-hand side of \cref{eq:misc_oyy}, when applied to $f\in \calA^{\alpha+1}(X)$, goes to $0$ in $\calA^{\alpha-}(X)$.  This shows that $P(\sigma)^{-1} f \in C^0(\bbR^+_\sigma;\calA^{\alpha-}(X))$. So, \cref{eq:misc_oyy} implies 
\begin{equation}
\frac{\mathrm{d}}{\mathrm{d} \sigma} P(\sigma)^{-1} f =- P(\sigma)^{-1} T(\sigma,\sigma) P(\sigma)^{-1} f= - P(\sigma)^{-1} \frac{\mathrm{d} P}{\mathrm{d} \sigma} P(\sigma)^{-1} f.
\label{eq:misc_034} 
\end{equation}
Applying \cref{eq:misc_034} inductively allows upgrading continuity to smoothness. 

The limiting absorption principle says nothing about the $\sigma\to 0$ limit; for instance, \Cref{prop:Vasy_absorption} requires $0\notin K$.
For the analysis of such low-energy phenomena, the key result upon which we will build is the following corollary of \cite{VasyLagrangian}:
\begin{proposition}
	Suppose that $\ker_{\calA^{d-2}(X)} P(0) = \{0\}$. Let $\ell<-1/2$, $s>-1/2-\ell$, $\nu \in (\ell+2-d/2,\ell+d/2)$. Then, there exists a $\sigma_0>0$ and $C>0$ such that
	\begin{equation}
	\lVert (\rho+\sigma)^\nu u \rVert_{H_{\mathrm{b}}^{s,\ell} } \leq C  \lVert (\rho+\sigma)^{\nu-1} P(\sigma)u \rVert_{H_{\mathrm{b}}^{s,\ell+1} }
	\end{equation}
	holds for all $\sigma \in (0,\sigma_0)$ and $u\in H_{\mathrm{b}}^{s,\ell}(X)$.  
	\label{prop:Vasy_low}
\end{proposition}
This is \cite[Theorem 2.11]{HintzPrice} (for real $\sigma$), but stated in greater generality. Again, see \S\ref{sec:tedium} for the details.

Going forwards, we will always assume $\ker_{\calA^{d-2}(X)} P(0) = \{0\}$. Thus, when we say that $P(\sigma)$ is as in \S\ref{sec:op}, we are including this assumption.

Let us extract a mapping property of $P(\sigma)$ from the foregoing estimates. 

\begin{proposition}
	Let $0<\alpha<(d-1)/2$. 
	For each $f(\sigma) \in \sigma L^\infty([0,1)_\sigma; \calA^{\alpha+1}(X)) $, the function $u=P(\sigma)^{-1} f$, which is well-defined by \cref{eq:misc_030}, satisfies 
	\begin{equation} 
		u\in L^\infty([0,1)_\sigma; H_{\mathrm{b}}^{s,\ell}(X))
	\end{equation} 
	for all $s,\ell\in \bbR$ such that $\ell<\alpha-d/2$.
	\label{prop:nice_low_energy} 
\end{proposition}
\begin{proof}
	We already know, by the discussion in the previous subsection, that $u\in L^\infty_{\mathrm{loc}} ((0,1]_\sigma; \calA^{\alpha-}(X))$. By \cref{eq:Sobolev}, this implies 
	\begin{equation} 
		u\in L^\infty_{\mathrm{loc}}((0,1]_\sigma; H_{\mathrm{b}}^{s,\ell}(X)).
	\end{equation} 
	So, it is only the boundedness in the $\sigma\to 0^+$ limit that needs to be understood.
	By \cref{eq:Sobolev}, we have  
	\begin{equation} 
	f(\sigma) \in \sigma L^\infty([0,1)_\sigma; H_{\mathrm{b}}^{s,\alpha+1-d/2-}(X)),
	\end{equation} 
	for any $s\in \bbR$. Let $\ell<\alpha-d/2$. Then, $f\in \sigma L^\infty([0,1)_\sigma; H_{\mathrm{b}}^{s,\ell+1}(X))$.
	Taking $s>-1/2-\ell$ (it sufficing to consider this case), and choosing some $\nu \in (\ell+2-d/2,\ell+d/2)$, \Cref{prop:Vasy_low} says (note that the hypothesis $\ell<-1/2$ is satisfied, since $\alpha<(d-1)/2$) that
	\begin{equation}
	\lVert (\rho+\sigma)^\nu u \rVert_{H_{\mathrm{b}}^{s,\ell} } \lesssim  \lVert (\rho+\sigma)^{\nu-1} P(\sigma)u \rVert_{H_{\mathrm{b}}^{s,\ell+1} } = \lVert (\rho+\sigma)^{\nu-1} f \rVert_{H_{\mathrm{b}}^{s,\ell+1} },
	\label{eq:misc_k77}
	\end{equation}
	at least if we restrict attention to $\sigma$ is sufficiently small.
	
	Using the lower bound $0<\alpha$ in the proposition statement,  the allowed $\ell$ includes at least the nonempty interval $(-d/2,\alpha-d/2)$. Choosing such $\ell$, the interval of allowed $\nu$ then includes $(\alpha+2-d,\ell+d/2)\supseteq [-(d-3)/2,\ell+d/2)\ni  0$. So, \cref{eq:misc_k77} gives
	\begin{equation}
	\lVert u \rVert_{H_{\mathrm{b}}^{s,\ell} } \lesssim   \lVert (\rho+\sigma)^{-1} f \rVert_{H_{\mathrm{b}}^{s,\ell+1} } = \Big\lVert \frac{\sigma}{\rho+\sigma} \sigma^{-1} f \Big\rVert_{H_{\mathrm{b}}^{s,\ell+1}} \lesssim \rVert \sigma^{-1} f \rVert_{H_{\mathrm{b}}^{s,\ell+1} },
	\end{equation}
	using the lemma that $\sigma/(\rho+\sigma)$ is, on each b-Sobolev space, a uniformly bounded one-parameter family of operators. This latter fact follows from the observation that $|(\rho\partial_\rho)^k (\sigma / (\rho+\sigma))| \leq C_k$ for all $\sigma\in [0,1]$ and $k\in \bbN$; that is, $\sigma/(\rho+\sigma)$ is a symbol on $X$ whose various symbolic seminorms are uniform in $\sigma$. 
	
	Since $f\in \sigma L^\infty([0,1)_\sigma; H_{\mathrm{b}}^{s,\ell+1}(X))$, we conclude that $u\in L^\infty([0,1)_\sigma; H_{\mathrm{b}}^{s,\ell}(X))$.
\end{proof}

We can summarize the conclusion of the proposition as stating that $u\in L^\infty([0,1)_\sigma; \calA^{\alpha-}(X))$. So, 
\begin{equation}
	P^{-1} : \sigma L^\infty([0,1)_\sigma; \calA^{\alpha+1}(X)) \to L^\infty([0,1)_\sigma; \calA^{\alpha-}(X)).
\end{equation}

Using the usual resolvent identity \cref{eq:misc_034} and the fact that $\mathrm{d} P/\mathrm{d} \sigma :  L^\infty([0,1)_\sigma; \calA^{\beta}(X)) \to L^\infty([0,1)_\sigma; \calA^{\beta+1}(X))$ for all $\beta\in \bbR$, it can be concluded that 
\begin{equation}
	\frac{\mathrm{d}^k P^{-1}}{\mathrm{d} \sigma^k} : \sigma^{k+1} L^\infty([0,1)_\sigma; \calA^{\alpha+1}(X)) \to L^\infty([0,1)_\sigma; \calA^{\alpha-}(X)).
	\label{eq:resolvent_mapping_main}
\end{equation}
for all $k\in \bbN$, $\alpha$ as above. For example, one application of \cref{eq:misc_034} yields
\begin{equation}
	\frac{\mathrm{d}}{\mathrm{d} \sigma} P(\sigma)^{-1}  =- P(\sigma)^{-1} \frac{\mathrm{d} P}{\mathrm{d} \sigma} P(\sigma)^{-1} : \sigma^2 \calA^{\alpha+1} \overset{P^{-1}}{\to} \sigma \calA^{\alpha-} \overset{P'}{\to} \sigma \calA^{1+\alpha-} \overset{P^{-1}}{\to} \calA^{\alpha-} ,
\end{equation}
with all the bounds uniform as $\sigma\to 0^+$. 
\section{Analysis at $\mathrm{bf}$}
\label{sec:bf}

We now prove the lemma used in the final step of the proof of the main theorem describing the asymptotic behavior of the solution of the Sommerfeld problem for $\sigma$-dependent forcing $f\in \calS([0,1)_\sigma; \calA^{1+\calE}(X) )$, for any index set $\calE\subseteq \bbC\times \bbN$. Despite this being the \emph{final} step of the proof, we discuss it first, the reason being that it is the easiest, owing to the assumed Schwartz behavior of $f$ as $\sigma\to 0^+$. The rest of this paper consists of the reduction of the full problem to this special case.

As a corollary of \Cref{prop:Vasy_low}, we get an $O(\sigma^\infty)$ estimate on the resolvent with respect to the seminorms of $\calA^{\alpha-}$:
\begin{proposition}
	Let $\alpha \in (0,(d-1)/2)$ and $f\in \calS([0,1)_\sigma; \calA^{1+\alpha}(X))$. 
	Then, 
	\begin{equation} 
		P(\sigma)^{-1} f \in \calS([0,1)_\sigma ; \calA^{\alpha-}(X)), 
	\end{equation}   
	where $P(\sigma)^{-1} f$ is defined by \cref{eq:misc_030}. 
	\label{prop:bf_1}
\end{proposition}
Cf.\ \cite[Theorem 2.14]{HintzPrice}.
\begin{proof}
	We have, for any $k,\kappa\in \bbN$, 
	\begin{equation}
		\sigma^{-\kappa} \frac{\partial^k}{\partial \sigma^k} P(\sigma)^{-1} f(-;\sigma) = \sum_{j=0}^k \binom{k}{j} \frac{\partial^j P(\sigma)^{-1}}{\partial \sigma^j} \Big(\sigma^{j+1} \Big(\underbrace{\sigma^{-\kappa-j-1}\frac{\partial^{k-j} f}{\partial \sigma^{k-j}}}_{\in\calS([0,1)_\sigma; \calA^{1+\alpha}(X)) } \Big)\Big).
	\end{equation}
	Via \cref{eq:resolvent_mapping_main}, each term on the right-hand side is in $L^\infty([0,1)_\sigma; \calA^{\alpha-}(X))$. 
\end{proof}

But, since we want asymptotics, we want estimates of $P(\sigma)^{-1} f$ with respect to the seminorms of the polyhomogeneous function spaces $\smash{\calA^{\calE}}$, not the conormal function spaces $\calA^{\alpha-}$. So, we need to upgrade the previous result.  This can be done using a standard ``normal operator'' argument, varieties of which will be used throughout the other sections of this paper. The version required in this section is the simplest among them, so it also serves us as a warm-up.

The operation $\uplus$ (defined in \S\ref{sec:index}) on (pre)index sets satisfies the following lemma, which we will need shortly:
\begin{lemma}
	For an index set $\calE$ as above, and for any $a\in \bbC$ and $f\in \calA^{((1+a,0)\uplus (1+\calE))\cup \calE}([0,1)_\rho)$, 
	\begin{equation}
	\rho^a \int^{1}_\rho s^{-a-1}f(s) \dd s \in \calA^{(a,0)\uplus  \calE}([0,1)_\rho)  
	\end{equation}
	holds.
	\label{lem:uplus_mapping}
\end{lemma}
\begin{proof}
	Suppose that $f\in \calA^{\calG}([0,1)_\rho)$ for some index set $\calG$. Then, we can write 
	\begin{equation}
		f(\rho) = F+ \sum_{(j,k)\in \calG, \Re j \leq \Re a} f_{j,k} \rho^j (\log \rho)^k 
	\end{equation} 
	for some $f_{j,k}\in \bbC$ and $F\in \calA^{\Re a+} \cap \calA^{\calG}([0,1)_\rho) $. So, 
	\begin{equation}
			\rho^a \int^{1}_\rho s^{-a-1}f(s) \dd s   = \rho^a \int_\rho^1 s^{-a-1} F(s) \dd s + \sum_{(j,k)\in \calG, \Re j \leq \Re a} f_{j,k} \rho^a \int_\rho^1 s^{j-a-1} (\log s)^k \dd s .
	\end{equation}
	The first term on the right-hand side is in $\calA^{(a,0)\cup \calG}([0,1)_\rho)$. On the other hand, 
	\begin{equation}
		\rho^a \int_\rho^1 s^{j-a-1} (\log s)^k \dd s \in 
		\begin{cases}
			\calA^{(a,0)\cup \calG} & (j\neq a) \\ 
			\calA^{(a,k+1) } & (j=a). 
		\end{cases}
	\end{equation}
	So, 
	\begin{equation}
	\rho^a \int^{1}_\rho s^{-a-1}f(s) \dd s \in \calA^{(a,0)\uplus \calG}([0,1)_\rho).
	\end{equation}
	If $\calG = ((1+a,0)\uplus (1+\calE))\cup \calE$, then 
	\begin{align} 
		\begin{split} 
		(a,0)\uplus \calG&=(a,0)\uplus (((1+a,0)\uplus (1+\calE))\cup \calE) \\ 
		&= (a,0) \uplus ( (1+ (a,0)\uplus \calE) \cup \calE ) = (a,0) \uplus \calE,
		\end{split} 
	\end{align} 
	so we are done.
\end{proof}

The \emph{normal operator} of $P$ at $\mathrm{bf}$ is defined by 
\begin{equation}
N_{\mathrm{bf}}(P) = 2i\sigma \frac{\partial}{\partial r} + \frac{i\sigma(d-1)}{r}.
\label{eq:Nbf_form}
\end{equation}
The sense in which this deserves to be called the ``normal operator'' is that 
\begin{equation}
P- N_{\mathrm{bf}}(P) \in \operatorname{Diff}_{\mathrm{b}}^{2,-2,-2,0}(X^+_{\mathrm{res}});
\end{equation}
in $\operatorname{Diff}_{\mathrm{b}}^{m,s,\ell,q}(X^+_{\mathrm{res}})$, $m$ is the differential order, $s$ is the growth order at $\mathrm{bf}$, $\ell$ is the growth order at $\mathrm{tf}$, and $q$ is the growth order at $\mathrm{zf}$.
Thus, as measured in orders of b-decay, $P- N_{\mathrm{bf}}(P)$ is faster decaying than $P\in \operatorname{Diff}_{\mathrm{b}}^{2,-1,-2,0}(X^+_{\mathrm{res}})$ at $\mathrm{bf}$.

We will see that the asymptotic behavior of the resolvent output at $\mathrm{bf}$ is governed by the ``indicial roots'' of $N_{\mathrm{bf}}(P)$. These are defined to be the roots of the polynomial in $a$ formed by replacing $r\partial_r$ in $rN_{\mathrm{bf}}(P)$ by $-a$. (This terminology comes from the theory of regular singular ODE.) Concretely, this polynomial is $-2i \sigma a +i\sigma(d-1) $. So, the only indicial root is 
\begin{equation} 
	a = (d-1)/2,
\end{equation} 
which, fortunately, does not depend on $\sigma$. This is why solutions of the Helmholtz equation decay like $1/r^{(d-1)/2}$ as $r\to\infty$. 

\begin{proposition}
	Let $\alpha >0 $ and $\calE$ denote an index set such that $\calE\subset \bbC_{>0}\times \bbN$, where $\bbC_{>\beta}=\{z\in \bbC:\Re z>\beta\}$.   
	Suppose that 
	$f\in \calS([0,1)_\sigma; \calA^{1+\calE,1+\alpha}(X))$, $u\in \calS([0,1)_\sigma ;\calA^{0+}(X))$, and also that $Pu=f$. Then, 
	\begin{equation} 
		u\in 
		\begin{cases}
		\calS([0,1)_\sigma; \calA^{(a,0)\uplus \calE,\alpha}(X)) & (\alpha\neq a) \\
		\calS([0,1)_\sigma; \calA^{(a,0)\uplus \calE,\alpha-}(X)) & (\alpha=a)
		\end{cases}
	\end{equation}  
	holds, where $a=(d-1)/2$, as above. 
	This applies, in particular, to the solution $u=P^{-1} f$ to the Sommerfeld problem, when $P^{-1}$ exists.
	\label{prop:bf_2}
\end{proposition}
The reason why the index set $(a,0)$ appears in $u$ is because, when integrating, we get a constant term. 
\begin{proof}
	We will discuss the $\alpha\neq a$ case. The $\alpha=a$ case is similar, except one needs to allow an arbitrarily small loss of decay.
	
	It suffices to restrict attention to $\dot{X}^+_{\mathrm{res}} \subset X^+_{\mathrm{res}}$. 
	Let $S$ denote the set of $\beta\in \bbR$ such that 
	\begin{equation} 
	u\in 
	\calS([0,1)_\sigma; \calA^{(a,0)\uplus \calE,\beta \wedge \alpha}(\dot{X})) ,
	\end{equation} 
	where the `$\wedge$' means minimum. 
	Because we are assuming that $u\in \calS([0,1)_\sigma ;\calA^{0+}(X))$, we know that $(-\infty,\epsilon]\subset S$ for some $\epsilon>0$.  Our goal is to show that $S=\bbR$.
	Suppose that $\beta \in S$. 
	
	Writing $P=N_{\mathrm{bf}}(P)  + (P-N_{\mathrm{bf}}(P))$, we have 
	\begin{equation}
		N_{\mathrm{bf}}(P) u = Pu - (P-N_{\mathrm{bf}}(P)) u = f - (P-N_{\mathrm{bf}}(P)) u. 
		\label{eq:misc_004}
	\end{equation}
	Because $P-N_{\mathrm{bf}}(P) \in \rho_{\mathrm{tf}}^2 \rho_{\mathrm{bf}}^2 \operatorname{Diff}^2_{\mathrm{b}}(\dot{X}_{\mathrm{res}}^+)$, we have $(P-N_{\mathrm{bf}}(P)) u \in \calS([0,1)_\sigma; \calA^{(a+2,0)\uplus(2+\calE),2+\beta\wedge \alpha}(\dot{X}))$. So, letting $f_0=f - (P-N_{\mathrm{bf}}(P)) u$, 
	so that \cref{eq:misc_004} reads $N_{\mathrm{bf}}(P) = f_0$, 
	\begin{equation} 
		f_0 \in \calS([0,1)_\sigma; \calA^{((2+a,0)\uplus(2+\calE) )\cup (1+\calE),\min\{1+\alpha,2+\beta\}}(X)).
	\end{equation}
	Using the explicit formula \cref{eq:Nbf_form}, \cref{eq:misc_004} reads 
	\begin{equation}
		(r\partial_r + a)u=f_1 
	\end{equation}
	for $f_1 = -i r f_0  / 2 \sigma \in \calS([0,1)_\sigma; \calA^{((1+a,0)\uplus (1+\calE))\cup \calE,\min\{\alpha,1+\beta\}}(X))$. 
	This first-order ODE can be solved explicitly via integrating in the radial direction. Rewriting it in terms of $\rho=1/r$, it becomes $(\rho\partial_\rho -a) u = -f_1$. 
	That is, $\partial_\rho (\rho^{-a} u) = -\rho^{-1-a} f_1$. 
	So, 
	\begin{equation}
	u(\sigma,\rho,\theta) = \rho^a\Big(  c(\sigma,\theta)  +  \int_\rho^{1/2} \frac{f_1(\sigma,s,\theta)}{s^{1+a}} \dd s \Big)
	\end{equation}
	for some $c(\sigma,\theta) :(0,1)_\sigma\times \partial X_\theta \to \bbC$. This satisfies $c(\sigma,\theta) = 2^{a} u(\sigma,1/2,\theta) \in \calS([0,1)_\sigma; C^\infty(\partial X_\theta))$, so $c$ is Schwartz as $\sigma\to 0^+$. 
	Then, applying \Cref{lem:uplus_mapping} gives 
	\begin{equation}
		u\in
		\calS([0,1)_\sigma; \calA^{(a,0)\uplus \calE,\alpha\wedge(1+\beta-)} (\dot{X})),
	\end{equation}
	i.e.\ that $1+\beta- \in S$.

	Proceeding inductively, we conclude that $S=\bbR$. 
\end{proof}

\section{Lemma producing $O(\rho_{\mathrm{zf}})$-quasimodes}
\label{sec:1}

We discuss in \S\ref{sec:0} the mapping properties of $P(0)^{-1}$ with respect to function spaces
\begin{equation}
\calA^\bullet(X;\ell) + \sum_{j=0}^{\ell-1} \calA_{\mathrm{c}}^\bullet([0,1) ) \calY_j,
\label{eq:b934k}
\end{equation} 
whose definition is in that appendix.
In this section, we apply the fruits of that discussion to the construction of $O(\rho_{\mathrm{zf}})$-quasimodes, the first step in the proof of our main theorem. 
For this section and the ones that follow, let $P$ be as in \S\ref{sec:op}, and fix $\ell\in \bbN$, which counts the number of harmonics (meaning $\calY_j$'s) we keep track of separately from the rest.

What we want to do, given forcing $f$ which is polyhomogeneous on $X^+_{\mathrm{res}}$, is find a polyhomogeneous $u$ such that $Pu-f$ is decaying at $\mathrm{zf}$ relative to $f$, without being much worse elsewhere. This is accomplished by the following lemma. The function spaces used in the statement of the lemma are denoted
\begin{equation}
\calA^\bullet(X^+_{\mathrm{res}};\ell) + \sum_{j=0}^{\ell-1} \calA_{\mathrm{c}}^\bullet([0,1)^+_{\mathrm{res}} ) \calY_j. 
\label{eq:b934z}
\end{equation}
These are defined analogously to the spaces in \cref{eq:b934k}. Namely, the first term contains functions on $X^+_{\mathrm{res}}$ that, near the large-$r$ boundary (i.e.\ on $\dot{X}^+_{\mathrm{res}}=[0,1)^+_{\mathrm{res}}\times \partial X_\theta$), are orthogonal to the harmonics in $\calY_0\cup\cdots\cup \calY_{\ell-1}$. The sum in \cref{eq:b934z} contains functions on $\dot{X}^+_{\mathrm{res}}$ independent of $\theta$ (and localized near the boundary) times the elements of the $\calY_j$'s. 

For simplicity, we work only with index sets in this section (and those that follow, except \S\ref{sec:Price}), not pre-index sets. However, we wrote \S\ref{sec:0} using pre-index sets. So, we employ the following mild change in notation: we use $\calI[\calF_\bullet,\ell]$, $\calI_j[\calF_\bullet]$ to refer to the index sets generated by the pre-index sets in
\Cref{thm:phg0} that we denoted using the same symbols.

\begin{lemma}
	Suppose that  
	\begin{equation} 
	f\in \calA^{1+\calE,2+\calF,(0,0)}(X_{\mathrm{res}}^+;\ell) + \sum_{j=0}^{\ell-1} \calA^{1+\calE_j,2+\calF_j,(0,0)}_{\mathrm{c}}([0,1)^+_{\mathrm{res}}) \calY_j,
	\end{equation} 
	where the index sets are listed at $\mathrm{bf}$, $\mathrm{tf}$, $\mathrm{zf}$, respectively. 
	Suppose moreover that $\calF_\bullet \subset \bbC_{>0}\times \bbN$. Let\footnote{We use ``$\calF_\bullet$'' as shorthand for the tuple of all of $\calF,\calF_0,\cdots$. The index sets $\calI_\bullet$, as defined in \S\ref{sec:0}, depend on the index sets $\calF_\bullet$ (which we called $\calE_\bullet$ in \S\ref{sec:0}) and on the parameters $\beth,\beth_0\geq 0$ appearing in \S\ref{subsec:op}.}
	\begin{align}
		\begin{split} 
			\calI^+[\calF_\bullet,\ell] &= \calF\cup \calI[\calF_\bullet,\ell]\cup (1+\beth_0+\calI_\cup) \cup (1+\beth_2\wedge(1+\beth_4) +\calI_\cup ), \\ 
			\calI^+_j[\calF_\bullet] &= \calF_j\cup \calI_j[\calF_\bullet]\cup (1+\beth_0+\calI_\cup) \cup (1+\beth_2\wedge(1+\beth_4) +\calI_\cup ).
		\end{split}  
		\label{eq:I+def}
	\end{align}
	Then, there exists 
	\begin{equation}
	u\in \calA^{\infty,\calI[\calF_\bullet,\ell],(0,0)}(X_{\mathrm{res}}^+;\ell) + \sum_{j=0}^{\ell-1} \calA^{\infty,\calI_j[\calF_\bullet],(0,0)}_{\mathrm{c}}([0,1)^+_{\mathrm{res}}) \calY_j
	\label{eq:misc_j66}
	\end{equation}
	such that 
	\begin{equation}
	Pu-f \in \calA^{1+\calE,2+\calI^+[\calF_\bullet,\ell],(1,0)}(X_{\mathrm{res}}^+;\ell) + \sum_{j=0}^{\ell-1} \calA^{1+\calE_j,2+\calI_j^+[\calF_\bullet],(1,0)}_{\mathrm{c}}([0,1)^+_{\mathrm{res}}) \calY_j.
	\label{eq:misc_j77}
	\end{equation}
	\label{lem:zf_pre}
\end{lemma}
Thus, $u$ is an $O(\rho_{\mathrm{zf}})$-quasimode (though we may have worsened decay at $\mathrm{tf}$). 

\begin{proof}
	The restriction $f|_{\mathrm{zf}^\circ}$ is well-defined, since $f$ is smooth at $\mathrm{zf}^\circ$. (We will drop the `$\circ$' in $f|_{\mathrm{zf}^\circ}$.) This satisfies
	\begin{equation} 
	f|_{\mathrm{zf}} \in \calA^{2+\calF}(X;\ell) + \sum_{j=0}^{\ell-1} \calA^{2+\calF_j}_{\mathrm{c}}([0,1)_\rho) \calY_j.
	\label{eq:misc_072}
	\end{equation}  
	Via the solvability theory of $N_{\mathrm{zf}}(P) = P(0)= \triangle_g +L$ embedded in \Cref{thm:phg0}, we have a well-defined element 
	\begin{equation}
	N_{\mathrm{zf}}(P)^{-1} f|_{\mathrm{zf}} \in \calA^{\calI[\calF_\bullet,\ell]} (X;\ell) + \sum_{j=0}^{\ell-1} \calA_{\mathrm{c}}^{\calI_j[\calF_\bullet] } ([0,1)_\rho )\calY_j.
	\label{eq:misc_073}
	\end{equation}
	Indeed, the most important fact about $N_{\mathrm{zf}}(P)$ is the fact that it is invertible as a map $\calA^{0+}(X)\to \calA^{2+}(X)$, 
	\begin{equation}
		N_{\mathrm{zf}}(P)^{-1} : \calA^{2+}(X)\to \calA^{0+}(X) ,
	\end{equation}
	owing to the assumption of the  absence of a resonance or bound state at zero energy; see \cite[Eq.\ 2.9]{HintzPrice} for the $d=3$, asymptotically Euclidean case, the general case being analogous. Then, \Cref{thm:phg0} says, for $f|_{\mathrm{zf}}$ as in \cref{eq:misc_072}, that \cref{eq:misc_073} holds. 
	
	Let $u = \chi_0(r\sigma) \smash{N_{\mathrm{zf}}(P)^{-1}} f|_{\mathrm{zf}}$ for $\chi_0 \in C_{\mathrm{c}}^\infty(\bbR)$ identically $1$ near the origin. Then, $\chi_0(r\sigma)$ is identically $1$ near $\mathrm{zf}$ and vanishing near $\mathrm{bf}$. So, if we interpret $N_{\mathrm{zf}}(P)^{-1} f|_{\mathrm{zf}}$ as a function on $X^{+\circ}_{\mathrm{res}}$ that is just independent of $\sigma$, then \cref{eq:misc_j66} holds.
	
	We now check that \cref{eq:misc_j77} holds.
	Decompose 
	\begin{equation} 
	Pu - f = (\chi_0 f|_{\mathrm{zf}} - f) +[N_{\mathrm{zf}}(P),\chi_0 ]N_{\mathrm{zf}}(P)^{-1} f|_{\mathrm{zf}}  + (P-N_{\mathrm{zf}}(P) ) u.
	\end{equation} 
	Let us analyze each piece of this. 
	\begin{itemize}
		\item It immediately follows from the stated regularity of $f|_{\mathrm{zf}}$ that 
		\begin{equation} 
		\chi_0 f|_{\mathrm{zf}} \in \calA^{\infty,2+\calF,(0,0)}(X_{\mathrm{res}}^+;\ell) + \sum_{j=0}^{\ell-1} \calA_{\mathrm{c}}^{\infty,2+\calF_j,(0,0)}([0,1)^+_{\mathrm{res}}) \calY_j.
		\end{equation} 
		So, 
		\begin{equation}
		\chi_0 f|_{\mathrm{zf}} - f \in  \calA^{1+\calE,2+\calF,(1,0)}(X_{\mathrm{res}}^+;\ell) + \sum_{j=0}^{\ell-1} \calA_{\mathrm{c}}^{1+\calE_,2+\calF_j,(1,0)}([0,1)^+_{\mathrm{res}}) \calY_j,
		\end{equation}
		where the improvement at $\mathrm{zf}$ of one order comes from the fact that $f$ agrees with $\chi_0 f|_{\mathrm{zf}}$ at $\mathrm{zf}$. 
		\item Because $[N_{\mathrm{zf}}(P),\chi_0 ]$ is supported away from $\mathrm{zf} \cup \mathrm{bf}$, we have 
		\begin{equation}
		[N_{\mathrm{zf}}(P),\chi_0 ] \in \rho_{\mathrm{zf}}^\infty \rho_{\mathrm{bf}}^\infty( \rho_{\mathrm{tf}}^2 \operatorname{Diff}_{\mathrm{b}}^1 ([0,1)^+_{\mathrm{res}} )  + \rho^{3+\beth_0} \operatorname{Diff}_{\mathrm{b}}^1(X^+_{\mathrm{res}})   ).
		\end{equation}
		(Recall that $\beth_0$ is as in \cref{eq:misc_a}, \cref{eq:misc_b}.)
		Therefore,
		\begin{multline} 
		[N_{\mathrm{zf}}(P),\chi_0 ]N_{\mathrm{zf}}(P)^{-1} f|_{\mathrm{zf}} \in \calA^{\infty,(2+\calI[\calF_\bullet,\ell])\cup(3+\beth_0+\calI_\cup),\infty  } (X^+_{\mathrm{res}};\ell) \\ + \sum_{j=0}^{\ell-1} \calA_{\mathrm{c}}^{\infty,(2+\calI_j[\calF_\bullet])\cup (3+\beth_0+\calI_\cup)),\infty } ([0,1)^+_{\mathrm{res}} )\calY_j,
		\end{multline} 
		where $\calI_\cup = \calI[\calF_\bullet,\ell]\cup\bigcup_{j=0}^{\ell-1} \calI_j[\calF_\bullet]$. 
		\item Finally, note that
		\begin{multline}
		\sigma^{-1}(P(\sigma) - N_{\mathrm{zf}}(P)) = 2 i (1-\chi) \frac{\partial}{\partial r} + \frac{i(d-1)}{r} + Q + \sigma R \\ \in \rho_{\mathrm{bf}}\rho_{\mathrm{tf}} \operatorname{Diff}_{\mathrm{b}}^1([0,1)^+_{\mathrm{res}} ) + \rho_{\mathrm{tf}}^{(2+\beth_2)\wedge (3+\beth_4)}\rho_{\mathrm{bf}}^{2+\beth_2 \wedge \beth_4} \operatorname{Diff}^1_{\mathrm{b}}(X^+_{\mathrm{res}}) .
		\end{multline}
		So,  
		\begin{multline} 
		\sigma^{-1}(P(\sigma) - N_{\mathrm{zf}}(P)) u \in  \calA^{\infty,(1+\calI[\calF_\bullet,\ell])\cup (2+\beth_2\wedge(1+\beth_4)+\calI_\cup )  ,(0,0)} (X^+_{\mathrm{res}};\ell ) \\ 
		+ \sum_{j=0}^{\ell-1} \calA_{\mathrm{c}}^{\infty,(1+\calI_j[\calF_\bullet])\cup(2+\beth_2\wedge(1+\beth_4)+\calI_\cup ),(0,0) } ([0,1)^+_{\mathrm{res}} )\calY_j.
		\end{multline} 
	\end{itemize}
	Combining everything, we get \cref{eq:misc_j77}.
\end{proof}

We actually want to apply this with more general index sets at $\mathrm{zf}$: 
\begin{lemma}
	Let $\kappa\in \bbN$. 
	Suppose that  $\calF_\bullet,\calJ_\bullet\subset \bbC_{>0}\times \bbN$ are some index sets, and that
	\begin{equation} 
		f\in \calA^{1+\calE,2+\calF,(0,\kappa)\cup \calJ}(X_{\mathrm{res}}^+;\ell) + \sum_{j=0}^{\ell-1} \calA^{1+\calE_j,2+\calF_j,(0,\kappa)\cup \calJ_j}_{\mathrm{c}}([0,1)^+_{\mathrm{res}}) \calY_j.
	\end{equation} 
	Then, there exists 
	\begin{equation}
		u \in \calA^{\infty,\calI[\calF_\bullet+(0,\kappa),\ell],(0,\kappa)}(X_{\mathrm{res}}^+;\ell) + \sum_{j=0}^{\ell-1} \calA^{\infty,\calI_j[\calF_\bullet+(0,\kappa)],(0,\kappa)}_{\mathrm{c}}([0,1)^+_{\mathrm{res}}) \calY_j
		\label{eq:misc_po1}
	\end{equation}
	such that 
	\begin{equation}
	Pu-f \in  \calA^{1+\calE,2+\calI[\calF_\bullet + (0,\kappa),\ell]^+,(1,\kappa)\cup \calJ}(X_{\mathrm{res}}^+;\ell) + \sum_{j=0}^{\ell-1} \calA^{1+\calE_j,2+\calI_j[\calF_\bullet + (0,\kappa)]^+,(1,\kappa)\cup \calJ_j}_{\mathrm{c}}([0,1)^+_{\mathrm{res}}) \calY_j
	\end{equation}
	for $\calI_\bullet^+$ as above.
	\label{lem:zf_main}
\end{lemma}  
\begin{proof}
	We can find $\tilde{f}_\varkappa$ in the same spaces as $f$, except with index set $(0,0)$ at $\mathrm{zf}$,  such that $f = F+ \sum_{\varkappa=0}^\kappa (\log \rho_{\mathrm{zf}} )^{\varkappa} \tilde{f}_\varkappa$ for 
	\begin{equation}
		F \in \calA^{1+\calE,2+\calF,\calJ}(X_{\mathrm{res}}^+;\ell) + \sum_{j=0}^{\ell-1} \calA^{1+\calE_j,2+\calF_j,\calJ_j}_{\mathrm{c}}([0,1)^+_{\mathrm{res}}) \calY_j.
	\end{equation} 
	We can choose the boundary-defining-functions $\rho_\bullet$ such that $\rho_{\mathrm{zf}}=\sigma/\rho_{\mathrm{tf}}$, and then 
	\begin{equation}
		f = F + \sum_{\varkappa=0}^\kappa (\log \sigma)^\varkappa \sum_{\varkappa'=\varkappa}^\kappa \underbrace{\binom{\varkappa'}{\varkappa} (-\log \rho_{\mathrm{tf}} )^{\varkappa'-\varkappa} \tilde{f}_{\varkappa'} }_{f_\varkappa}.
	\end{equation}
	Call the $\varkappa$th summand, without the $(\log \sigma)^\varkappa$, ``$f_\varkappa$,'' as indicated. Then,
	\begin{equation}
		f_\varkappa \in \calA^{1+\calE,2+\calF + (0,\kappa-\varkappa),(0,0)}(X_{\mathrm{res}}^+;\ell) + \sum_{j=0}^{\ell-1} \calA^{1+\calE_j,2+\calF_j+ (0,\kappa-\varkappa),(0,0)}_{\mathrm{c}}([0,1)^+_{\mathrm{res}}) \calY_j. 
	\end{equation}
	We now apply \Cref{lem:zf_pre} to get 
	\begin{equation}
		u_\varkappa \in \calA^{\infty,\calI[\calF_\bullet+(0,\kappa-\varkappa),\ell],(0,0)}(X_{\mathrm{res}}^+;\ell) + \sum_{j=0}^{\ell-1} \calA^{\infty,\calI_j[\calF_\bullet+(0,\kappa-\varkappa)],(0,0)}_{\mathrm{c}}([0,1)^+_{\mathrm{res}}) \calY_j
	\end{equation}
	satisfying the conclusion of the lemma, with $f_\varkappa$ in place of $f$ and $\calF+(0,\kappa-\varkappa)$ in place of $\calF$. Then, defining $u=\sum_{\varkappa=0}^\kappa (\log \sigma)^\varkappa u_\varkappa$, we see that $u$ satisfies \cref{eq:misc_po1} and 
	\begin{multline}
		Pu-f = - F + \sum_{\varkappa=0}^\kappa (\log \sigma)^\varkappa(Pu_\varkappa-  f_\varkappa) \\ \in \calA^{1+\calE,2+\calI[\calF_\bullet + (0,\kappa),\ell]^+,(1,\kappa)\cup \calJ}(X_{\mathrm{res}}^+;\ell) + \sum_{j=0}^{\ell-1} \calA^{1+\calE_j,2+\calI_j[\calF_\bullet + (0,\kappa)]^+,(1,\kappa)\cup \calJ_j}_{\mathrm{c}}([0,1)^+_{\mathrm{res}}) \calY_j. 
		\label{eq:misc_06j}
	\end{multline}	
\end{proof}
\begin{remark*}
	The reason for the awkward rewriting of the logarithmic terms $(\log \rho_{\mathrm{zf}})^\varkappa \tilde{f}_\varkappa$ above is that $N_{\mathrm{zf}}$ commutes with $\log \sigma$ but not $\log \rho_{\mathrm{zf}}$.
\end{remark*}

\section{Lemma upgrading $O(\rho_{\mathrm{zf}})$-quasimodes to  $O(\sigma^{0+})$-quasimodes}
\label{sec:2}

Given $f\in \calA^{1+\calE,2+\calF,(0,\kappa)\cup \calJ}(X^+_{\mathrm{res}} )$ and index sets $\calF ,\calJ\subset \bbC_{>0}\times \bbN$, we can now find some $u\in \calA^{\calE, \calI ,(0,\kappa) }(X^+_{\mathrm{res}} )$ such that 
\begin{equation}
Pu - f  \in \calA^{1+\calE,2+\calI^+ ,(1,\kappa)\cup \calJ}(X^+_{\mathrm{res}} ),
\end{equation}
where the index sets $\calI,\calI^+\subset \bbC_{>0}\times \bbN$ are defined in the previous section.
Thus, $Pu-f$ is suppressed by some (typically one in the Euclidean case) order at $\mathrm{zf}$. In this sense, $u$ is an $O(\sigma^{0+})$-quasimode in $\mathrm{zf}^\circ$, where $O(\sigma^{0+})$ means $O(\sigma^\varepsilon)$ for some $\varepsilon>0$. But it is not an $O(\sigma^{0+})$-quasimode at $\mathrm{zf}\cap \mathrm{tf}$, let alone at $\mathrm{tf}$, because the index set $\calI^+$ may be much larger than $\calF$. For example, even if $\calF=\varnothing$, in which case $f|_{\mathrm{zf}}$ is Schwartz, it will typically be the case that the set $\calI$ is nonempty  (for example, Coulomb's law results in solutions of Poisson's equation on $\bbR^3$ decaying like $1/r$ even if the forcing is Schwartz); so, $Pu-f$ will generally not be Schwartz. So, relative to $f$, $Pu-f$ may be large at $\mathrm{tf}$.

Our next task is to solve away the error at $\mathrm{tf}$ to a high degree, though we may find later that it is only useful to solve away a few terms. This process will end up costing us logarithmic losses at $\mathrm{zf}$. Since these logarithms are the ultimate source of Price's law, as we have remarked several times already, this loss is worth noting. The analytic back-end of this section is \S\ref{sec:tf}, in the same sense that the analytic back-end of the previous section was \S\ref{sec:0}.
 
Our goal is, given $g\in \calA^{1+\calE,2+\calI^+ ,(1,\kappa)\cup \calJ}(X^+_{\mathrm{res}} )$,
to find $v\in \calA^{\calE^+,\calI^+,\calK}(X^+_{\mathrm{res}} )$ such that 
\begin{equation}
Pv-g \in \calA^{1+\calE^{++}, K-,\calK^+} (X^+_{\mathrm{res}}), 
\end{equation}
where $\calK,\calK^+$ are some index sets such that $\calK,\calK^+\subset \bbC_{> 0} \times \bbN$ and $K\in \bbR$ is hopefully large (though $\calK,\calK^+$ may depend on $K$), and where $\calE^+,\calE^{++}$ are some not too large index sets. 
Applying this with $g=f-Pu$, the function $w=u+v$ satisfies 
\begin{equation} 
Pw-f = Pv-g \in  \calA^{1+\calE^{++}, K-, \calK^+} (X^+_{\mathrm{res}}).
\end{equation} 
So, $w$, as opposed to $u$, is an honest $O(\sigma^{0+})$-quasimode (at least modulo logarithmic losses at $\mathrm{zf}$ when $\min \Pi \calE^{++}\geq \min \Pi \calE$, which will be the case when we use this below). 

In order to carry this out, we can use the following lemma. It is phrased using the sequences $\{b_j\}_{j=0}^\infty,\{c_j\}_{j=0}^\infty \subset [0,\infty)$, defined in \S\ref{sec:0}, \S\ref{sec:tf} in terms of the eigenvalues $\lambda_j$ of the boundary Laplacian, 
\begin{align*}
b_j &= 2^{-1}(2-d + ((d-2)^2+4\lambda_j)^{1/2}) \\ 
c_j &= 2^{-1}(d-2 + ((d-2)^2+4\lambda_j)^{1/2}),
\end{align*}
as well as the constants $\beth_\bullet\geq 0$ defined in \S\ref{subsec:op}.

We also reference the set $\calB_{\geq \ell}[\calG]$ defined in \S\ref{sec:tf}, for each index set $\calG$, to be the smallest pre-index set containing $(b_j,0)\uplus \calG$ for every $j\geq \ell$. See \cref{eq:uplus} for the definition of the $\uplus$ operation. Actually, here we let $\calB_{\geq \ell}[\calG]$ denote the index set generated by that pre-index set. This is just to simplify our computations a bit.
\begin{lemma} 
	Suppose that $\calE_\bullet,\calF_\bullet,\calJ_\bullet$ are some index sets such that $\calJ_\bullet \subset \bbC_{>0}\times \bbN$. 
	Fix $\ell\in \bbN$ and $K_j\leq c_j+ \min\Pi \calJ_j$ and $K\leq c_\ell+\min \Pi \calJ$. Let 
	\begin{equation}
	f \in \calA^{1+\calE,2+\calF,\calJ}(X^+_{\mathrm{res}};\ell) + \sum_{l=0}^{\ell-1} \calA_{\mathrm{c}}^{1+\calE_l,2+\calF_l,\calJ_l}([0,1)^+_{\mathrm{res}} ) \calY_l, 
	\label{eq:misc_090}
	\end{equation}
	$\calE^+ = (2^{-1}(d-1),0)\uplus \calE$, $\calE_l^+ = (2^{-1}(d-1),0)\uplus \calE_l$, and 
	\begin{align}
	\begin{split} 
	\calK &= \bigcup_{(j,k)\in \calF, \Re j< K} (j+\calB_{\geq \ell}[ \calJ+(-j,k_j)])\\ 
	\calK_l & = \bigcup_{(j,k)\in \calF, \Re j< K}  (b_l+j,0)\uplus ( \calJ_l +(0,k_{l,j}))
	\end{split} 
	\label{eq:misc_094}
	\end{align}
	where $k_{\bullet,j}$ is the largest $k$ such that $(j,k) \in \calF_\bullet$.
	Then, 
	there exists
	\begin{equation}
	u\in \calA^{\calE^+,\calF, \calK}(X^+_{\mathrm{res}};\ell ) + \sum_{l=0}^{\ell-1} \calA_{\mathrm{c}}^{\calE_l^+,\calF_l,\calK_l}([0,1)^+_{\mathrm{res}} ) \calY_l
	\end{equation}
	satisfying
	\begin{equation}
	Pu-f \in \calA^{1+\calE^{++},2+\calF^+, \calK_\cup}(X^+_{\mathrm{res}};\ell ) + \sum_{l=0}^{\ell-1} \calA_{\mathrm{c}}^{1+\calE_l^{++},2+\calF_l^+,  \calK_\cup}([0,1)^+_{\mathrm{res}} ) \calY_l
	\label{eq:misc_095}
	\end{equation}
	for
	\begin{align}
	\calE^{++}_\bullet &=\calE_\bullet\cup ((2+\beth)\wedge(1+ \beth_1\wedge \beth_3) + \calE^+_\bullet)  \cup ((2+\beth_0)\wedge(1+ \beth_2\wedge \beth_4) + \calE^+_{\cup}) ,
	\label{eq:misc_0uu} \\
	\calF^+_\bullet &= \calF_{\bullet,\geq K}\cup (1+\beth\wedge\beth_1 \wedge(1+\beth_3) + \calF_\bullet)\cup (1+\beth_0\wedge\beth_2\wedge(1+\beth_4) + \calF_\cup),
	\label{eq:misc_097}
	\end{align}
	where $\calF_{\bullet,\geq K} = \calF_\bullet \cap (\bbC_{\geq K}\times \bbN)$ consists of the indices $(j,k)$ in $\calF$ whose first component $j$ satisfies $\Re j\geq K$. 
	\label{lem:Osig}
\end{lemma}
At this point, keeping track of subscripts in parallel arguments becomes annoying, so we make liberal use of `$\bullet$'. In any individual expression, $\bullet$ either means $j\in \{0,\dots,\ell-1\}$ in all locations or, in all locations, it means one of $\ell$ or a lack of a subscript (only one of which will make sense), or sometimes `rem.'

\begin{remark}[Decay at $\mathrm{zf}$]
	We have stated the previous lemma with no assumptions whatsoever on $\calE_\bullet,\calF_\bullet$, but the result is not useful for the intended applications without imposing some restrictions. Indeed:
	\begin{itemize}
		\item if $\min \Pi \calF_l \leq - b_l$ or $\min \Pi \calF \leq -b_\ell$, then $Pu-f$ might not even be decaying at $\mathrm{zf}$, as then $\calK_\cup$ can have terms $(j,k)$ with $\Re j\leq 0$. 
		\item If $\calE_\bullet$ is too big, then we may have $\calE_\bullet^{++}\supsetneq \calE_\bullet$, and then $Pu-f$ may be worse at $\mathrm{bf}$ then $f$. 
	\end{itemize} 
	These issues do not occur as long as $\calE_\bullet,\calF_\bullet$ satisfy reasonable assumptions. For example, the easiest way to ensure that  
	$\calE_\bullet^{++}$ is not much larger than $\calE_\bullet$ 
	is to require 
\begin{equation}
\calE_\bullet=(2^{-1}(d+1),0)  \subseteq \bbC_{>2^{-1}(d-1)} \times \bbN.  
\label{eq:misc_109}
\end{equation}
Indeed, if this holds, then $\calE_\bullet^+ = (2^{-1}(d-1),0)$. Then $\calE_\cup^+ =  (2^{-1}(d-1),0)$ as well, so \cref{eq:misc_0uu} gives $\calE_\bullet^{++}=(2^{-1}(d+1),0)$ back. We will eventually assume \cref{eq:misc_109}, but we do not do this just yet, so as to keep the discussion general.

Regarding $\calF_\bullet$, note that $\min \Pi \calK_\cup \geq \min\{\Pi \calJ_\bullet,\Pi \calF_\cup\}$. So, as long as $\calF_\bullet\subset \bbC_{>0}\times \bbN$, then $\min \Pi \calK_\cup >0$. 
So, in the application we have in mind,  the index set $\calK_\cup$ does imply decay at $\mathrm{zf}$. 
\label{rem:index_cancellation}
\end{remark}

\begin{remark}[A finer observation about $\calK$]
	The indices $(j,k)\in \calK_\cup $ such that $j\notin \calJ_\bullet$ necessarily satisfy $j \geq   \min \Pi \calF_\cup$.
	This means that, though the $\mathrm{zf}$ index set $\calK_\cup$ appearing in \cref{eq:misc_095} can be bigger than $\calJ_\bullet$, the difference is only in indices that are subleading and maybe some extra logarithmic terms on top of indices already present. Moreover, if $\calF_\bullet \subset \bbC_{\geq 1}\times \bbN$, then the indices with new $j$ are all subleading by at least one full order. 
	\label{eq:K_reqs}
\end{remark}

\begin{remark*}[On $u$ being a $O(\sigma^{0+})$-quasimode]
	When $K_\bullet>\min \{\Pi \calF_\cup\}$, the index set $\calF_\cup^+$ is better than $\calF_\cup$.
	Combining this observation with the previous remark:  when also the assumptions in \Cref{rem:index_cancellation} regarding $\calE_\bullet,\calF_\bullet$ are satisfied, $Pu-f= O(\sigma^{0+})$ (with respect to the relevant spaces of polyhomogeneous functions), as desired.
\end{remark*}

\begin{remark*}[Keeping $\calF_\bullet^+$ small]
		We will want to keep $\calF_\bullet^+$ small. Specifically, if we want to factor out a $\sigma^\varepsilon$ from $Pu-f$ and repeat the analysis with $\sigma^{-\varepsilon}(Pu-f)$ in place of $f$ (which is what we do in \S\ref{sec:arg}), then 
		\begin{itemize}
			\item the restriction on $\calF_\bullet$ imposed in \Cref{rem:index_cancellation}, when applied to  $\sigma^{-\varepsilon}(Pu-f)$, means that we want $\calF_\bullet^+ \subset \bbC_{>-b_\bullet+\varepsilon}\times \bbN$.  This means we should take $K_\bullet > - b_\bullet+\varepsilon$. 
			\item  Another restriction on $K_\bullet$ comes from wanting to apply \Cref{lem:zf_main} to $\sigma^{-\varepsilon}(Pu-f)$. Then, we want $\calF_{\bullet}^+\subseteq \bbC_{>\varepsilon}\times \bbN$. For this, we want $K_\bullet>\varepsilon$. 
		\end{itemize}
		As long as $\varepsilon \in (0,1]$, neither of these requirements are in tension with our required upper bound $K_\bullet \leq c_\bullet + \min \Pi \calJ_\bullet$, since $b_\bullet\geq 0$ and the interval $(\varepsilon,c_\bullet+\min \Pi \calJ_\bullet]$ leaves ample room for choosing $K_\bullet$. For example, the smallest of these intervals is 
		\begin{equation} 
			(\varepsilon,c_0+\min \Pi \calJ_\bullet]\supseteq (1,d-2+\min \Pi \calJ_\bullet]\neq \varnothing.
		\end{equation} 
	In summary, \emph{as long as our goal is to keep the $\calF_\bullet$ sets describing $\sigma^{-\varepsilon}(Pu-f)$ (that is, $\calF^+_\bullet-(\varepsilon,0)$) within $\bbC_{>0}\times \bbN$, then we can do this.}

	Actually, the situation is even better, because we can keep the $\calF$-index sets describing $\sigma^{-\varepsilon}(Pu-f)$ within $\bbC_{\geq 1}\times \bbN$. For this, we want $\calF_\bullet^+\subset \bbC_{\geq 1+\varepsilon} \times \bbN$, which would mean taking $K_\bullet \geq 1+\varepsilon$. Then, the bound of admissible $K_\bullet$ is 
	\begin{equation}
	[1+\varepsilon,c_\bullet + \min \Pi \calJ_\bullet]\supseteq [1+\varepsilon,1+ \min \Pi \calJ_\bullet], 
	\end{equation} 
	and this is nonempty if $\varepsilon \leq \min \Pi \calJ_\bullet$. We will never take  $\varepsilon > \min \Pi \calJ_\bullet$ (we cannot factor out more $\sigma$ than $Pu-f$ has at $\mathrm{zf}^\circ$, because then $\sigma^{-\varepsilon}(Pu-f)$ will not have the right orders at $\mathrm{zf}$). In fact, in \S\ref{sec:arg}, we will take $\varepsilon = \min \Pi \calJ_\bullet$ if this is less than one. In that case, the interval of allowed $K_0$ is a singleton, $\{1+\min \Pi \calJ_\bullet\}$, and the other $K_\bullet$ have larger admissible intervals. 
	
	For the reader worried about having such a small interval of allowed $K_0$, let us point that actually slightly smaller $K_0\in (1,2)$ will work. This is because if $\calF_0$ does not have any $(j,k)$ with $j=K_0$, and if $K_0 < 1+\min \Pi \calF_\cup$, then $\calF_0^+$ will \emph{automatically} be in $\bbC_{\geq K_0+\delta}\times \bbN$ for some $\delta>0$ having to do with the next index in $\calF_0$ to the right of $j=K_0$. For example, in the asymptotically Euclidean case, the index sets $\calF_\bullet,\calJ_\bullet$ usually have only $(j,k)$ with $j\in \bbN$. Then, once $K_0>1$, we automatically have $\calF_0^+\subseteq \bbC_{\geq 2}\times \bbN$ (assuming that $\calF_\cup \subset \bbC_{\geq 1}\times \bbN$). Factoring out a $\sigma^\varepsilon$ for $\varepsilon \leq 1$ then results in $\sigma^{-\varepsilon}(Pu-f)$ being described by a $\calF$-index set in $\bbC_{\geq 1}\times \bbN$. In this case, $\min \Pi \calJ_\bullet =1$, so the upper bound on allowed $K_0$ is $2$. Thus, the whole interval $(1,2]$ of $K_0$ is allowed and works, not just the endpoint $K_0=2$.
	
	Moreover, 
	demanding that the $\calF$-index sets describing $\sigma^{-\varepsilon}(Pu-f)$ be in $\bbC_{\geq 1}\times \bbN$ is slight overkill for what we need in \S\ref{sec:arg}. In \S\ref{sec:arg}, all we need is that the $\calF$-index sets be in $\bbC_{\geq \delta}\times \bbN$ for some constant $\delta\in(0,1)$. Then, the interval of allowed $K_\bullet$ becomes 
	\begin{equation}
	[\delta+\varepsilon,c_\bullet + \min \Pi \calJ_\bullet]\supseteq [\delta+\varepsilon,1+ \min \Pi \calJ_\bullet], 
	\end{equation} 
	which, for $\varepsilon \leq \min\{1,\Pi \calJ_\bullet\}$, is nonempty and not a singleton. (And the previous paragraph still applies, so actually slightly smaller $K_\bullet$ work if there are no borderline indices in $\calF_\bullet$.) 
	
	For the reader surprised at how much wiggle room we have, we point out that this is related to \S\ref{rem:mechanism}.
\end{remark*}

Before proving the lemma, we explain the idea of the proof in the case $\calJ,\calJ_l = (1,0)$. (The general case just involves more bookkeeping.) 
Let $\tilde{f} = \tilde{\chi}(\rho) f$, where $\tilde{\chi} \in C_{\mathrm{c}}^\infty((-1,1)_\rho )$ is identically $1$ near the origin. 
Then, we can consider $\tilde{f}$ as a function on $\dot{X}^+_{\mathrm{res}}$, which we embed in 
\begin{equation}
\tilde{X}^+_{\mathrm{res}} = [ [0,1)_\sigma \times \partial X_\theta \times [0,\infty]_{\hat{r}} ; \{\sigma=0,\hat{r}=0\} ]
\end{equation}
via the natural embedding, $(\sigma,\theta,r)\mapsto (\sigma,\theta,\hat{r}=r\sigma)$; see \Cref{fig:Xtilde+res}. Note that  closure of $\{\sigma=0,\hat{r} \in (0,\infty)\}\subset [0,1)_\sigma \times \partial X_\theta \times [0,\infty]_{\hat{r}}$ in $\tilde{X}^+_{\mathrm{res}}$ is identifiable with $\mathrm{tf}$.
So, we can write
\begin{equation}
\tilde{f} \in \calA^{1+\calE,2+\calF,(1,0),\infty} (\tilde{X}^+_{\mathrm{res}} ;\ell) + \sum_{l=0}^{\ell-1} \calA_{\mathrm{c}}^{1+\calE_l ,2+\calF_l,(1,0),\infty } (\widetilde{[0,1)}^+_{\mathrm{res}} )\calY_l,
\end{equation} 
where the last `$\infty$' in the superscripts denotes Schwartz behavior at $\{\sigma>0,\hat{r}=0\}\subset \tilde{X}^+_{\mathrm{res}}$. (Really, $\tilde{f}$ is supported away from this face, so ``Schwartz'' is an understatement.)

\begin{figure}[t]
	\begin{center}
		\begin{tikzpicture} 
		\fill[gray!5] (5,1.5) -- (0,1.5) -- (0,0)  arc(90:0:1.5) -- (5,-1.5) -- cycle;
		\draw[dashed] (5,1.5) -- (5,-1.5);
		\node (ff) at (.75,-.75) {$\mathrm{zf}$};
		\node (zfp) at (-.35,.75) {$\mathrm{tf}$};
		\node (pf) at (2.5,1.8) {$\mathrm{bf}$};
		\node (mf) at (3.25,-1.8) {$\operatorname{cl}_{\tilde{X}^+_{\mathrm{res}}}\{\hat{r}=0,\sigma>0\}$};
		\draw[dashed, gray!30] (4.5,-1.5) -- (4.5,1.5);
		\draw[dashed, gray!30] (4,-1.5) to[out=90, in=270] (3.5,1.5);
		\draw[dashed, gray!30] (3.5,-1.5) to[out=90, in=270] (2.5,1.5);
		\draw[dashed, gray!30] (3,-1.5) to[out=90, in=270] (1.5,1.5);
		\draw[dashed, gray!30] (2.5,-1.5) to[out=90, in=270] (.85,1.5);
		\draw[dashed, gray!30] (2,-1.5) to[out=90, in=270] (.25,1.5);
		\draw[->, darkred] (.1,.1) -- (.1,.5) node[right] {$\hat{r}$};
		\draw[->, darkred] (.1,1.4) -- (1.2,1.4) node[below] {$\sigma$};
		\draw[->, darkred] (.1,1.4) -- (.1,1) node[right] {$1/\hat{r}$};
		\draw[->, darkred] (.1,.1) to[out=0, in=150] (1,-.23) node[above] {$\sigma/\hat{r}$};
		\draw[->, darkred] (1.6,-1.4) -- (2.2,-1.4) node[above] {$\sigma$};
		\draw[->, darkred] (1.6,-1.4) to[out=90, in=-60] (1.375,-.65) node[right] {$\hat{r}/\sigma$};
		\draw (5,1.5) -- (0,1.5) -- (0,0)  arc(90:0:1.5)  -- (5,-1.5);
		\node at (2.5,0) {$\tilde{X}^+_{\mathrm{res}} $};
		\end{tikzpicture}
		\qquad 
		\begin{tikzpicture} 
		\fill[gray!5] (5,1.5) -- (0,1.5) -- (0,0)  arc(90:0:1.5) -- (5,-1.5) -- cycle;
		\fill[darkgray!25] (5,1.5) -- (0,1.5) -- (0,0)  arc(90:45:1.5) to[out=35, in=180] (5,.5) -- cycle;
		\draw[dashed, darkblue] (0,0)  arc(90:45:1.5) to[out=35, in=180] (5,.5);
		\draw[dashed] (5,1.5) -- (5,-1.5);
		\node (ff) at (.75,-.75) {$\mathrm{zf}$};
		\node (zfp) at (-.35,.75) {$\mathrm{tf}$};
		\node (pf) at (2.5,1.8) {$\mathrm{bf}$};
		\node (mf) at (3.25,-1.8) {$\operatorname{cl}_{\tilde{X}^+_{\mathrm{res}}}\{\hat{r}=0,\sigma>0\}$};
		\draw (5,1.5) -- (0,1.5) -- (0,0)  arc(90:0:1.5)  -- (5,-1.5);
		\node[darkblue] at (3,0) {$r=1$};
		\node at (1,.75) {$\dot{X}^+_{\mathrm{res}} $};
		\node at (3.5,-.75) {$\tilde{X}^+_{\mathrm{res}}\backslash \dot{X}^+_{\mathrm{res}} $};
		\end{tikzpicture}
	\end{center}
	\caption{The mwc $\tilde{X}^+_{\mathrm{res}}$ (left) and the image of the embedding $\dot{X}^+_{\mathrm{res}}\hookrightarrow \tilde{X}^+_{\mathrm{res}}$ (right). Lines of constant $\sigma$ are depicted as dashed gray lines in the left figure. Cf.\ \Cref{fig}.}
	\label{fig:Xtilde+res}
\end{figure}

We can consider $N_{\mathrm{tf}}$ as a differential operator on $\tilde{X}^{+\circ}_{\mathrm{res}}$. Indeed, since $N_{\mathrm{tf}}$, considered as a differential operator on $X_{\mathrm{res}}^{+\circ}$, commutes with $\sigma$, the formula 
\begin{equation}
N_{\mathrm{tf}} = \sigma^2 \hat{N}_{\mathrm{tf}}, \quad \hat{N}_{\mathrm{tf}}  = - \frac{\partial^2}{\partial \hat{r}^2} - \Big( \frac{d-1}{\hat{r}} + 2i \Big) \frac{\partial}{\partial \hat{r}} + \frac{1}{\hat{r}^2} \triangle_{\partial X} - \frac{i(d-1)}{\hat{r}}
\label{eq:Ntf_form}
\end{equation}
can be read as holding on $(\tilde{X}^+_{\mathrm{res}})^\circ$ where the partial derivatives are taken with respect to the coordinate system $(\sigma,\hat{r},\theta) \in \bbR^+_\sigma \times \bbR^+_{\hat{r}}\times \partial X_\theta$. 
Using the canonical diffeomorphism $\mathrm{tf}^\circ =\bbR^+_{\hat{r}}\times \partial X_\theta$, we are \emph{also} considering $\hat{N}_{\mathrm{tf}}$ as an operator on $\mathrm{tf}^\circ$, given by the same formula \cref{eq:Ntf_form}.

Unfortunately, $N_{\mathrm{tf}}$ does not commute with a boundary-defining-function 
\begin{equation} 
\rho_{\mathrm{tf}} \in C^\infty(\tilde{X}^+_{\mathrm{res}};[0,\infty))
\end{equation}
of $\mathrm{tf}$.
(We previously used ``$\rho_{\mathrm{tf}}$'' to refer to a boundary-defining-function of $\mathrm{tf}$ in $X^+_{\mathrm{res}}$, but since $\tilde{X}^+_{\mathrm{res}}\cong X^+_{\mathrm{res}}$ in a neighborhood of $\mathrm{tf}$, no confusion will arise from changing notation here.) 
This has the annoying consequence that, in the polyhomogeneous expansion 
\begin{equation}
\tilde{f} \sim \sum_{(j,k) \in \calF} \rho_{\mathrm{tf}}^{2+j} (\log\rho_{\mathrm{tf}})^k \tilde{f}_{j,k}, 
\end{equation}
though $\tilde{f}_{j,k} = \tilde{f}_{j,k}(\hat{r},\theta) \in C^\infty(\mathrm{tf}^\circ)$ is a function of $\hat{r}$ and $\theta$ alone, we should not expect
\begin{equation}
\hat{N}_{\mathrm{tf}}^{-1} \tilde{f} \sim \sum _{(j,k) \in \calF} \rho_{\mathrm{tf}}^{2+j} (\log\rho_{\mathrm{tf}})^k \hat{N}_{\mathrm{tf}}^{-1} \tilde{f}_{j,k}.
\end{equation}
Fortunately, we can instead expand in $\sigma$: 
\begin{equation}
\tilde{f} \sim \sum_{(j,k) \in \calF} \sigma^{2+j} (\log\sigma)^k \tilde{f}_{j,k}
\end{equation}
for some other $\tilde{f}_{j,k} = \tilde{f}_{j,k}(\hat{r},\theta) \in C^\infty(\mathrm{tf}^\circ)$. That is, for any $\beta\in \bbR$, 
\begin{equation}
\tilde{f} - \sum_{(j,k) \in \calF, \, \Re j\leq \beta} \sigma^{2+j} (\log\sigma)^k \tilde{f}_{j,k} \in \calA^{1+\calE,2+\beta,(1,0),\infty}(\tilde{X}^+_{\mathrm{res}}).
\label{eq:misc_h1v}
\end{equation}
More specifically, $\tilde{f}_{j,k} \in \calA^{1+\calE,(-1-j,k_j)}(\mathrm{tf})$; see \Cref{lem:joint} (note that what we call $\calF$ in that lemma is $2+\calF$ here). The cost of expanding in $\sigma$ rather than in $\rho_{\mathrm{tf}}$ is that each term in the expansion is worse than the previous at $\mathrm{zf}$. Fortunately, this does not mean that the \emph{error} that results from truncating the expansion (that is, the function in \cref{eq:misc_h1v}) needs to be worse at $\mathrm{zf}$. The point is that the error term is the same as what it was when we were expanding in $\rho_{\mathrm{tf}}$; all we have done is shuffled around some terms in the truncated expansion.  Again, see \Cref{lem:joint} for details.

Having set up these ``joint'' expansions, we can solve away each $\tilde{f}_{j,k}$, this procedure being facilitated by the fact that 
\begin{equation}
[N_{\mathrm{tf}}, \sigma^{2+j} (\log \sigma)^k] = 0.
\end{equation}
The expression ``$\hat{N}_{\mathrm{tf}}^{-1}$'' means the inverse of $\hat{N}_{\mathrm{tf}}$ on some suitable function spaces in which the $\tilde{f}_{j,k}$ lie, as discussed in \S\ref{sec:tf}.\footnote{The reason why we need to restrict $K_\bullet$ to not be too big is that if $j$ is large, $\tilde{f}_{j,k}$ will not lie in the domain of \Cref{prop:Ntf_main}. However, this does not seem to be an essential point, as we will remark below, in \Cref{rem:conf}.}

Let us now carry out the details, keeping track of things harmonic-by-harmonic: 
\begin{proof}[Proof of \Cref{lem:Osig}]
	Define $\tilde{f}$ as above, and let $\tilde{f}_{l}$ denote the component of $\tilde{f}$ in $\calY_l$ and let $\tilde{f}_{\mathrm{rem}}$ denote the component orthogonal to $\calY_0,\ldots,\calY_{\ell-1}$. 
	Letting 
	\begin{equation}
	\tilde{f}_\bullet \sim \sum_{(j,k) \in \calF_\bullet} \sigma^{2+j} (\log\sigma)^k \tilde{f}_{\bullet,j,k}
	\end{equation}
	denote the joint $\sigma\to 0^+$ expansion of $\tilde{f}_\bullet$, we have 
	\begin{equation}
	\tilde{f}_{\mathrm{rem},j,k} \in \calA^{1+\calE,\calJ+(-2-j,k_{j}-k)}(\mathrm{tf};\ell), \quad  
	\tilde{f}_{l,j,k} \in \calA^{1+\calE_l,\calJ+(-2-j,k_{l,j}-k) }([0,\infty]_{\hat{r}} )\calY_l 
	\end{equation}
	according to \Cref{lem:joint}, where $k_{\bullet,j}$ is the largest $k$ such that $(j,k)$ is in $\calF_\bullet$.  In the superscripts, the index sets are listed at bf first and zf second. Now let
	\begin{equation}
	\tilde{v} = \underbrace{\sum_{(j,k) \in \calF,   \Re j< K_{\mathrm{rem}}} \sigma^{j} (\log \sigma )^k \hat{N}_{\mathrm{tf}}^{-1} \tilde{f}_{\mathrm{rem},j,k}}_{\tilde{v}_{\mathrm{rem}}} + \sum_{l=0}^{\ell-1}\underbrace{ \sum_{(j,k) \in \calF_l,   \Re j< K_l} \sigma^{j} (\log \sigma )^k \hat{N}_{\mathrm{tf}}^{-1} \tilde{f}_{l,j,k}}_{\tilde{v}_l}, 
	\end{equation}
	where $K_\bullet$ is to-be-determined and $\tilde{v}_{\bullet,j,k}=\hat{N}_{\mathrm{tf}}^{-1} \tilde{f}_{\bullet,j,k} \in C^\infty(\mathrm{tf}^\circ)$ solves the PDE 
	\begin{equation} 
		\hat{N}_{\mathrm{tf}}\tilde{v}_{\bullet,j,k} = \tilde{f}_{\bullet,j,k}.
	\end{equation} 
	Here, $\hat{N}_{\mathrm{tf}}$ is read as an operator on $\mathrm{tf}$, as discussed above. We are using \Cref{prop:Ntf_main} to define $\hat{N}_{\mathrm{tf}}^{-1} \tilde{f}_{\bullet,j,k}$. Indeed, since we are assuming that $K_l \leq c_l +\Pi \calJ_l$, the functions $\tilde{f}_{\cdots}$ to which $\hat{N}_{\mathrm{tf}}$ is applied satisfy 
	\begin{equation} 
		\tilde{f}_{l,j,k} \in \calA^{1+\calE_l,-2-K_l+}([0,\infty]_{\hat{r}})\calY_l \subset \calA^{1+\calE_j,-2-c_l+}([0,\infty]_{\hat{r}})\calY_l,
	\end{equation} 
	and analogously for $\tilde{f}_{\mathrm{rem},j,k}$. This is what is required for $\hat{N}_{\mathrm{tf}}^{-1} \tilde{f}_{\bullet,j,k}$ to make sense.

	To write this more simply, let 
	\begin{equation}
	f_{\mathrm{approx}} =  \underbrace{\sum_{(j,k) \in \calF,   \Re j< K} \sigma^{2+j} (\log \sigma )^k \tilde{f}_{\mathrm{rem},j,k}}_{f_{\mathrm{rem},\mathrm{approx}}} + \sum_{l=0}^{\ell-1} \underbrace{\sum_{(j,k) \in \calF_l,   \Re j<K_l} \sigma^{2+j} (\log \sigma )^k \tilde{f}_{l,j,k}}_{f_{l,\mathrm{approx}}}.
	\end{equation}
	Then, $\tilde{v} = N_{\mathrm{tf}}^{-1} f_{\mathrm{approx}}$, where $N_{\mathrm{tf}}^{-1} = \sigma^{-2} \hat{N}_{\mathrm{tf}}^{-1}$.  
	By \Cref{prop:Ntf_main}, 
	\begin{align}
	\begin{split} 
	\hat{N}_{\mathrm{tf}}^{-1} \tilde{f}_{\mathrm{rem},j,k} 
	&\in \calA^{(2^{-1}(d-1),0) \uplus  \calE  , \calB_{\geq \ell}[\calJ+(-j,k_j-k)]  }(\mathrm{tf};\ell),  \\
	\hat{N}_{\mathrm{tf}}^{-1} \tilde{f}_{l,j,k} 
	&\in
	\calA^{(2^{-1}(d-1),0) \uplus \calE_l  ,  (b_l,0) \uplus (\calJ+(-j,k_{l,j}-k) )}([0,\infty]_{\hat{r}}) \calY_l.
	\end{split}
	\end{align}
	So, 
	\begin{multline}
	\tilde{v} \in \sum_{(j,k)\in \calF, \Re j< K } \sigma^j (\log \sigma)^k \calA^{(2^{-1}(d-1),0) \uplus \calE  , \calB_{\geq \ell}[ \calJ+(-j,k_j-k) ] }(\mathrm{tf};\ell) \\ 
	+ \sum_{l=0}^{\ell-1}  \sum_{(j,k)\in \calF_l, \Re j< K_l }   \sigma^j (\log \sigma)^k  \calA^{(2^{-1}(d-1),0) \uplus \calE_l  ,  (b_l,0) \uplus (\calJ+(-j,k_{l,j}-k) )}([0,\infty]_{\hat{r}}) \calY_l.
	\end{multline}
	Notice that 
	\begin{multline}
	\sigma^j (\log \sigma)^k \calA^{(2^{-1}(d-1),0) \uplus  \calE  , \calB_{\geq \ell}[ \calJ+(-j,k_j-k) ] }(\mathrm{tf};\ell) \\ \subseteq \calA^{(2^{-1}(d-1),0) \uplus \calE ,(j,k) ,j+ \calB_{\geq \ell}[ \calJ+(-j,k_j) ],-\infty }(\tilde{X}^+_{\mathrm{res}} ;\ell)
	\end{multline}
	and 
	\begin{multline}
	\sigma^j (\log \sigma)^k  \calA^{(2^{-1}(d-1),0) \uplus \calE_l  ,  (b_l,0) \uplus (\calJ+(-j,k_{l,j}-k)) }([0,\infty]_{\hat{r}})  \\ \subseteq   \calA^{(2^{-1}(d-1),0) \uplus  \calE_l  , (j,k) , (b_l+j,0)\uplus (\calJ+(0,k_{l,j})) ,-\infty }(\widetilde{[0,1)}^+_{\mathrm{res}} ).
	\end{multline}
	So, 
	\begin{equation}
	\tilde{v} \in \calA^{\calE^+,\calF,\calK,-\infty}(\tilde{X}^+_{\mathrm{res}};\ell ) + \sum_{l=0}^{\ell-1} \calA^{\calE^+_l,\calF_l, \calK_l,-\infty}  (\widetilde{[0,1)}^+_{\mathrm{res}} )\calY_l
	\end{equation}
	for $\calE_\bullet^+ = (2^{-1}(d-1),0) \uplus \calE_\bullet$, $\calK,\calK_l$ as specified.

	Now let $v = \chi_1(\sigma/\hat{r} )  \tilde{v}$ for $\chi_1 \in C_{\mathrm{c}}^\infty( (-\infty,1))$ such that $\chi_1=1$ identically near the origin. We can interpret $v$ as a function on $\dot{X}^+_{\mathrm{res}}$ and therefore on $X^+_{\mathrm{res}}$. We have 
	\begin{equation}
	v \in \calA^{\calE^+,\calF,\calK}(X^+_{\mathrm{res}};\ell ) + \sum_{l=0}^{\ell-1} \calA^{\calE^+_l,\calF_l, \calK_l}_{\mathrm{c}}  ([0,1)^+_{\mathrm{res}} )\calY_l.
	\end{equation}
	Let $v_{\mathrm{rem}}$ denote the first term on the right-hand side and $v_l$ the $l$th term in the sum. That is, $v_{\bullet} = \chi_1(\sigma/\hat{r}) \tilde{v}_\bullet$. 
	
	Let us now examine 
	\begin{equation}
	Pv_\bullet  - f_\bullet = (\chi_1 f_{\bullet,\mathrm{approx}} -f_\bullet ) + [N_{\mathrm{tf}} ,\chi_1 ] \tilde{v}_\bullet + (P-N_{\mathrm{tf}}) v_\bullet,
	\end{equation}
	piece-by-piece. Here, $f_\bullet$ denote the portions of $f$ in each of the terms on the right-hand side of \cref{eq:misc_090}.
	\begin{itemize}
		\item We split $\chi_1 f_{\bullet,\mathrm{approx}} -f_\bullet  = \chi_1( f_{\bullet,\mathrm{approx}}  - \tilde{f}_\bullet) + (\chi_1 \tilde{f}_\bullet-f_\bullet)$. Beginning with the first term, we have, from \cref{eq:misc_h1v} (applied with $f_\bullet$ in place of $f$), that 
		\begin{equation} 
			\tilde{f}_\bullet-f_{\bullet,\mathrm{approx}} \in \calA^{1+\calE_\bullet,2+K_\bullet-,\calJ_\bullet,-\infty}.
		\end{equation} 
		(Here, we are omitting the ``$(X^+_{\mathrm{res}};\ell)$'' or ``$([0,1)^+_{\mathrm{res}} ) \calY_l$'' from the notation, as it depends on $\bullet$.) Multiplying by $\chi_1$ then yields 
		\begin{align}
		\begin{split} 
		\chi_1(\tilde{f}_{\mathrm{rem}}-f_{\mathrm{rem},\mathrm{approx}}) &\in  \calA^{1+\calE,2+K_{\mathrm{rem}}-,\calJ}(X^+_{\mathrm{res}};\ell), \\ 
		\chi_1(\tilde{f}_l-f_{l,\mathrm{approx}}) &\in  \calA^{1+\calE_l,2+K_l-,\calJ_l}_{\mathrm{c}}([0,1)^+_{\mathrm{res}} ) \calY_l.
		\end{split}
		\end{align}
		The term $\chi_1 \tilde{f}_\bullet-f_\bullet$ vanishes identically near $\mathrm{tf}$, so 
		\begin{equation}
		\chi_1 \tilde{f}_\bullet-f_\bullet \in \calA^{1+\calE_\bullet, \infty, \calJ_\bullet } .
		\end{equation}
		So, overall, $\chi_1 f_{\bullet,\mathrm{approx}} -f_\bullet \in \calA^{1+\calE_\bullet,K_\bullet-,\calJ_\bullet}$. Actually, since the left-hand side is polyhomogeneous, we can improve this to $\chi_1 f_{\bullet,\mathrm{approx}} -f_\bullet \in \calA^{1+\calE_\bullet,2+\calF_{\bullet,\geq K_\bullet},\calJ_\bullet}$.
		\item On the other hand, $[N_{\mathrm{tf}},\chi_1 ] \in \operatorname{Diff}_{\mathrm{b}}^{2,-\infty,-\infty,0}(X^+_{\mathrm{res}})$, where the two $-\infty$ superscripts denote Schwartz coefficients at $\mathrm{tf}\cup \mathrm{bf}$. So, $[N_{\mathrm{tf}},\chi_1 ] \tilde{v}_\bullet \in \calA^{\infty,\infty ,\calK_\bullet}$.
		\item Finally, 
		\begin{multline}
		P - N_{\mathrm{tf}} = (\triangle_g-\triangle_{g_0}) + L + \sigma Q + \sigma^2 R \\ \in
		\rho^{(3+\beth)\wedge(2+ \beth_1\wedge \beth_3)  }_{\mathrm{bf}} \rho_{\mathrm{tf}}^{3+\beth\wedge \beth_1 \wedge (1+\beth_3) } \operatorname{Diff}^2_{\mathrm{b}}([0,1)^+_{\mathrm{res}} ) + \rho_{\mathrm{bf}}^{(3+\beth_0)\wedge(2+ \beth_2\wedge \beth_4)} \rho_{\mathrm{tf}}^{3+\beth_0\wedge\beth_2\wedge(1+\beth_4)} \operatorname{Diff}^2_{\mathrm{b}}(\dot{X})
		\end{multline}
		near $\partial X$,
		so $(P-N_{\mathrm{tf}}) v_{\mathrm{rem}} \in \calA^{1+\calE^{++0},2+\calF^{+0},\calK}(X^+_{\mathrm{res}};\ell) + \sum_{l=0}^{\ell-1}  \calA^{1+\calE^{++1},2+\calF^{+1},\calK}([0,1)^+_{\mathrm{res}}) \calY_l $
		for 
		\begin{align}
		\begin{split} 
		\calE^{++0} &= ((2+\beth)\wedge(1+ \beth_1\wedge \beth_3)  + \calE^+), \\ 	\calE^{++1} &=  ((2+\beth_0)\wedge(1+ \beth_2\wedge \beth_4) + \calE^+ ), 
		\end{split} 
		\intertext{and} 
		\begin{split}  
		\calF^{+0} &= (1+\beth\wedge\beth_1 \wedge(1+\beth_3) + \calF), \\
		\calF^{+1} &=  (1+\beth_0\wedge\beth_2\wedge(1+\beth_4) + \calF) .
		\end{split} 
		\end{align}
		Likewise, 
		\begin{multline} 
			(P-N_{\mathrm{tf}}) v_{l} \in \calA^{1+\calE^{++1}_l,\calF^{+1}_l,\calK_l}(X^+_{\mathrm{res}};\ell)+ \calA^{1+\calE^{++0}_{l},\calF^{+0}_l,\calK_l}_{\mathrm{c}}([0,1)^+_{\mathrm{res}} ) \calY_l 
			\\ + \sum_{j=0,j\neq l}^{\ell-1} \calA^{1+\calE^{++1}_{l},\calF^{+1}_l,\calK_l}_{\mathrm{c}}([0,1)^+_{\mathrm{res}} ) \calY_j
		\end{multline} 
		for 
		\begin{align}
		\begin{split} 
		\calE^{++0}_l &=  ((2+\beth)\wedge(1+ \beth_1\wedge \beth_3)  + \calE^+_l), \\ 
		\calF^{+0}_l &= 1+\beth\wedge\beth_1 \wedge(1+\beth_3) + \calF_l,
		\end{split} 
		\intertext{and }
		\begin{split} 
		\calE^{++1}_l &=  ((2+\beth_0)\wedge(1+ \beth_2\wedge \beth_4) + \calE^+_l ) , \\
		\calF^{+1}_l &=  (1+\beth_0\wedge\beth_2\wedge(1+\beth_4) + \calF_l) .
		\end{split} 
		\end{align}
	\end{itemize}
	So, all in all, 
	\begin{equation}
	Pv- f \in \calA^{1+\calE^{++},2+ \calF^+, \calK_\cup}(X^+_{\mathrm{res}};\ell)+ \sum_{l=0}^{\ell-1} \calA_{\mathrm{c}}^{1+ \calE^{++}_l,2+ \calF_l^+, \calK_\cup }([0,1)^+_{\mathrm{res}})\calY_l
	\end{equation}
	(where we used $\calK_\bullet \supseteq \calJ_\bullet$).
\end{proof}

\begin{lemma}
	Suppose that $\calE,\calF,\calG$ are index sets and $f\in \calA^{\calE,\calF,\calG}(\dot{X}^+_{\mathrm{res}})$. Then, there exist $f_{j,k}(\hat{r},\theta) \in  \calA^{\calE,(-j,k_j-k)+\calG}([0,\infty]_{\hat{r}}\times \partial X_\theta ) = \calA^{\calE,(-j,k_j-k)+\calG}(\mathrm{tf})$ such that
	\begin{equation}
	f \sim \sum_{(j,k)\in \calF} \sigma^j (\log \sigma)^k f_{j,k}
	\end{equation}
	in the sense that, for every $\beta\in \bbR$, 
	\begin{equation}
	f- \sum_{(j,k)\in \calF,\, \Re j\leq \beta} \sigma^j (\log \sigma)^k f_{j,k} \in \calA^{\calE,\beta,\calG}(\dot{X}^+_{\mathrm{res}}).
	\end{equation}
	Here, $k_j$ is the largest $k$ such that $(j,k)\in \calF$. 
	\label{lem:joint}
\end{lemma}
\begin{remark}
	Analogues exist if we restrict attention to functions orthogonal to the first few spherical harmonics near $\partial X$.
\end{remark}
\begin{proof}
	Let $\tilde{f}_{j,k} \in \calA^{\calE,\calG}(\mathrm{tf})$ denote the terms in the polyhomogeneous expansion of $f$ at $\mathrm{tf}$, so 
	\begin{equation}
	f \sim \sum_{(j,k)\in \calF} \rho_{\mathrm{tf}}^j (\log \rho_{\mathrm{tf}} )^k \tilde{f}_{j,k}, 
	\end{equation}
	in the usual sense that, for every $\beta\in \bbR$, 
	\begin{equation}
	f -\sum_{(j,k)\in \calF, \, \Re j\leq \beta} \rho_{\mathrm{tf}}^j (\log \rho_{\mathrm{tf}} )^k \tilde{f}_{j,k} \in \calA^{\calE,\beta,\calG}(\dot{X}^+_{\mathrm{res}})
	\end{equation}
	holds. Note that we can consider $\tilde{f}_{j,k}$ as an element of $\calA^{\calE,(0,0),\calG}(\dot{X}^+_{\mathrm{res}})$ which only depends on $\hat{r},\theta$. 
	
	We have $\sigma \in \rho_{\mathrm{tf}} \rho_{\mathrm{zf}} C^\infty(X^+_{\mathrm{res}};\bbR^+)$. That is, there exists some positive $\varsigma \in C^\infty(X^+_{\mathrm{res}})$ such that $\sigma \varsigma = \rho_{\mathrm{tf}} \rho_{\mathrm{zf}}$.
	For the sake of this computation, it is useful to note that we can take $\rho_{\mathrm{zf}}$ as a function of $\hat{r}$ alone (though $\hat{r},\theta$ would also suffice) and, simultaneously, that we can take $\varsigma= 1$. Indeed, these conditions are met by choosing $\rho_{\mathrm{tf}} = \sigma+\rho$ and $\rho_{\mathrm{zf}} = \sigma / (\sigma+\rho) = \hat{r}/(1+\hat{r})$. 
	We have
	\begin{equation}
	\rho_{\mathrm{tf}}^j (\log \rho_{\mathrm{tf}} )^k \tilde{f}_{j,k}  = \Big(\frac{\sigma}{\rho_{\mathrm{zf}}} \Big)^j \Big(\log\Big(\frac{\sigma }{\rho_{\mathrm{zf}}} \Big)\Big)^k \tilde{f}_{j,k} = \sigma^j \sum_{\kappa=0}^k (\log \sigma)^\kappa \binom{k}{\kappa} (-\log \rho_{\mathrm{zf}})^{k-\kappa} \frac{\tilde{f}_{j,k} }{ \rho_{\mathrm{zf}}^j }. 
	\end{equation}
	Thus, for fixed $j$, 
	\begin{align}
	\begin{split} 
	\sum_{k \text{ s.t. }(j,k)\in \calF} \rho_{\mathrm{tf}}^j (\log \rho_{\mathrm{tf}} )^k \tilde{f}_{j,k} &= 
	\sigma^j \sum_{k \text{ s.t. }(j,k)\in \calF} \sum_{\kappa=0}^k (\log \sigma)^\kappa \binom{k}{\kappa} (-\log \rho_{\mathrm{zf}})^{k-\kappa} \frac{\tilde{f}_{j,k} }{ \rho_{\mathrm{zf}}^j } \\ 
	&= 
	\sigma^j  \sum_{\kappa=0}^{k_j} (\log \sigma)^\kappa \sum_{k=\kappa}^{k_j}  \binom{k}{\kappa} (-\log \rho_{\mathrm{zf}})^{k-\kappa} \frac{\tilde{f}_{j,k} }{ \rho_{\mathrm{zf}}^j } = \sigma^j \sum_{\kappa=0}^{k_j} (\log \sigma)^\kappa f_{j,\kappa}
	\end{split}
	\end{align}
	for 
	\begin{equation}
	f_{j,\kappa} = \sum_{k=\kappa}^{k_j}  \binom{k}{\kappa} (-\log \rho_{\mathrm{zf}})^{k-\kappa} \frac{\tilde{f}_{j,k} }{ \rho_{\mathrm{zf}}^j }  \in \calA^{\calE,(-j,k_j-\kappa)+\calG}(\mathrm{tf}) .
	\end{equation}
	Note that $f_{j,\kappa}$ depends only on the coordinates $\hat{r},\theta$.
	
\end{proof}

\begin{remark}\label{rem:conf}
	If we drop the requirement above that $K_\bullet \leq c_\bullet+\Pi \calJ_\bullet$, then \Cref{lem:Osig} should still hold, except with slightly more logarithmic terms; we just need to use a modified form of \Cref{prop:Ntf_main} in which the first handful of harmonics (the $\calY_j$ for which $c_j< K_j-\min \Pi \calJ_j$) are dealt with by hand. The point is that if one is studying a second-order regular singular ODE $Pu=f$, an existence statement holds for $u$ even if $f$ blows up very badly at the singular point. 
	However, one loses uniqueness in this case. The resultant logarithmic singularities should be fictitious. 
	This is all to say that the restriction on $K_\bullet$ does not appear to be essential, but an optimal algorithm should involve it.
	
	To see how the existence works, consider the (unweighted) regular singular ordinary differential operator $P=(x\partial_x-1)(x\partial _x)$. Its indicial roots are $0,1$. If $f(x)=1/x$ (note that this has worse blowup at $x=0$ than either indicial powers $x^c$, $c\in \{0,1\}$), then $u(x) = 2^{-1}x^{-1} +c_0+c_1 x$ solves $Pu=f$. If $f(x)=1$ instead, then $u(x)= - \log(x)+c_0+c_1 x$ works, so we have picked up a log term.	In each of these cases, we have two real degrees-of-freedom $c_0,c_1$. For our actual $P$, another boundary condition is supplied somewhere. This eliminates one degree-of-freedom. So, one real degree-of-freedom is left undetermined. 
\end{remark}

\section{Main proof}
\label{sec:arg}

This section contains the proof of \Cref{thm:A}.
The first step is the construction of $O(\sigma^{0+})$-quasimodes, which is just a matter of combining the lemmas in the previous two sections. What we want to do is, given $f\in \calA^{1+\calE,2+\calF, \calJ}(X_{\mathrm{res}}^+)$ for appropriate index sets $\calE,\calF$ and $\calJ\subset \bbC_{\geq 0}\times \bbN$, find  
\begin{equation}
u \in \calA^{\calE^+,\calF^+,\calJ^+ }(X_{\mathrm{res}}^+)
\end{equation}
solving the quasimode construction problem 
\begin{equation} 
Pu=f \bmod \sigma^{0+} \calA^{1+\calE^{++},2+\calF^{++},\calJ^{++}}(X_{\mathrm{res}}^+)
\end{equation} 
for some index sets $\calE^+,\calF^+,\calJ^+,\calE^{++},\calF^{++},\calJ^{++}$ which differ from $\calE,\calF,\calJ$ in terms of which logarithmic and subleading terms are present, so that $\min \Pi \calI = \min \Pi \calI^+ = \min \Pi \calI^{++}$ for each $\calI=\calE,\calF,\calJ$, where, as before, $\Pi:\bbC\times \bbN\to \bbC$ denotes the projection onto the first factor. Indeed:

\begin{proposition}
	Suppose that $\calE_\bullet,\calF_\bullet,\calJ_\bullet$ are index sets such that $\min \Pi \calE_\bullet> (d-1)/2$, $\calF_\bullet\subseteq \bbC_{\geq 1}\times \bbN$, and $\calJ_{\bullet} \subset \bbC_{\geq 0}\times \bbN$, and let
	\begin{equation}
		f\in \calA^{1+\calE,2+\calF, \calJ}(X^+_{\mathrm{res}};\ell) + \sum_{l=0}^{\ell-1}\calA_{\mathrm{c}}^{1+\calE_l,2+\calF_l,\calJ_l}([0,1)^+_{\mathrm{res}})\calY_l.
	\end{equation}
	Let $\varepsilon  = \min \{1, \Pi \calJ_{\bullet} \cap (\bbC_{>0}\times \bbN) \}$. There exist index sets $\calE^+_\bullet,\calF^+_\bullet,\calJ^+_\bullet$, $\calE^{\mathrm{new}}_\bullet,\calF^{\mathrm{new}}_\bullet,\calJ^{\mathrm{new}}_\bullet$ 
	\begin{itemize}
		\item 
		satisfying the same hypotheses as their predecessors, with
		\item $\calI_\bullet\subseteq \calI_\bullet^+ \subseteq \calI_\bullet^{\mathrm{new}}$ for each $\calI\in \{\calE,\calF\}$,
	\end{itemize} 
	such that there exist 
	\begin{align}
	\begin{split} 
		u&\in \calA^{\calE^+,\calF^+,\calJ^+}(X^+_{\mathrm{res}};\ell) + \sum_{l=0}^{\ell-1}\calA_{\mathrm{c}}^{\calE_l^+,\calF_l^+,\calJ_l^+}([0,1)^+_{\mathrm{res}})\calY_l \\
		g &\in \calA^{1+\calE^{\mathrm{new}},2+\calF^{\mathrm{new}},\calJ^{\mathrm{new}}}(X^+_{\mathrm{res}};\ell) + \sum_{l=0}^{\ell-1}\calA_{\mathrm{c}}^{1+\calE_l^{\mathrm{new}},2+\calF_l^{\mathrm{new}},\calJ_l^{\mathrm{new}}}([0,1)^+_{\mathrm{res}})\calY_l
		\end{split}
		\label{eq:misc_136}
	\end{align}
	such that $Pu=f+\sigma^\varepsilon g$. 
	\label{prop:O_sig_eps}
\end{proposition}  
\begin{proof}
	Without loss of generality, we can assume that $(0,0)\in \calJ_\bullet$. Let $\kappa\in \bbN$ be the largest nonnegative integer such that $(0,\kappa)\in \calJ_\bullet$ for some $\bullet$. 
	Now, \Cref{lem:zf_main} gives $w \in \calA^{\infty,\calI[\calF_\bullet+(0,\kappa)],(0,\kappa)}(X_{\mathrm{res}}^+;\ell) + \sum_{l=0}^{\ell-1} \calA^{\infty,\calI_l[\calF_\bullet+(0,\kappa)],(0,\kappa)}_{\mathrm{c}}([0,1)^+_{\mathrm{res}}) \calY_l$
	such that $Pw=f+h$ for some
	\begin{equation}
		h\in \calA^{1+\calE,2+\calI[\calF_\bullet + (0,\kappa)]^+,(1,\kappa)\cup \calJ_{>0}}(X_{\mathrm{res}}^+;\ell) + \sum_{l=0}^{\ell-1} \calA^{1+\calE_l,2+\calI_l[\calF_\bullet + (0,\kappa)]^+,(1,\kappa)\cup \calJ_{l,>0}}_{\mathrm{c}}([0,1)^+_{\mathrm{res}}) \calY_l, 
	\end{equation}
	where $\calI,\calI_l,\calI^+,\calI_l^+$ are as in \Cref{lem:zf_main}, and where by $\calJ_{\bullet,>0}$ we mean to throw out the terms $(0,\kappa)$. 
	Referring to the definitions of $\calI_\bullet$ in \Cref{lem:zf_pre}, we have $\min \Pi \calI_\bullet\geq \min\{\Pi \calF_\bullet,1+\Pi \calF_\cup,c_\bullet \}$ (see \cref{eq:index_set_defs}). Since all of the $c_j$'s are $\geq d-2\geq 1$, this means that $\calI_\bullet\subset \bbC_{\geq 1 }\times \bbN$, and it then follows from the definition of $\calI^+_\bullet$ that the same holds for it:
	\begin{equation}
	\calI_\bullet[\cdots]^+\subset \bbC_{\geq 1}\times \bbN\label{eq:blah}
	\end{equation}

	Now we apply \Cref{lem:Osig} with $h$ in place of $f$. In this way, we get 
	\begin{equation}
		v \in \calA^{\calE^+,\calI[\calF_\bullet + (0,\kappa)]^+,\calK}(X_{\mathrm{res}}^+;\ell) + \sum_{l=0}^{\ell-1} \calA^{\calE_l^+,\calI_l[\calF_\bullet + (0,\kappa)]^+, \calK_l}_{\mathrm{c}}([0,1)^+_{\mathrm{res}}) \calY_l, 
	\end{equation}
	where $\calK_\bullet \subset \bbC_{>0}\times \bbN$  and $\calE_\bullet^+$ are as in that lemma, such that 
	\begin{equation}
	Pv - h \in \calA^{1+\calE^{++},2+ \calI[\calF_\bullet+(0,\kappa)]^{++},\calK_\cup  }(X^+_{\mathrm{res}};\ell )  + \sum_{l=1}^{\ell-1} \calA_{\mathrm{c}}^{1+\calE^{++}_l,2+ \calI_l[\calF_\bullet+(0,\kappa)]^{++}, \calK_\cup  }([0,1)_{\mathrm{res}}^+ ) \calY_l, 
	\end{equation}
	where $\calE^{++}_\bullet$ and $\calI_\bullet[\cdots]^{++} = (\calI_\bullet[\cdots]^{+})^+$ are defined  as in \Cref{lem:Osig}. There are parameters $K_\bullet$ going into that lemma, which we will say more about shortly. If $K_\bullet$ are sufficiently small/negative, then the hypotheses of that lemma are satisfied, so we may apply it (though it might be trivial if $K_\bullet$ is too small/negative).

	Now let $u=w-v$. This lies in the space in \cref{eq:misc_136}, for $\calE^+$ as above, 
	\begin{equation} 
	\calF_\bullet^+ = \calF_\bullet\cup \calI[\calF_\bullet + (0,\kappa)] \cup \calI[\calF_\bullet + (0,\kappa)]^+, 
	\end{equation} 
	and $\calJ_\bullet^+ = \calJ_\bullet \cup \calK_\bullet$. Note that each of these index sets contains their predecessors, by construction.

	We can write $Pu=Pw-Pv = f+\sigma^\varepsilon g$, for $g=-\sigma^{-\varepsilon}(Pv-h)$. This lies in the space in \cref{eq:misc_136} for
	\begin{align}
	\begin{split} 
	\calE^{\mathrm{new}}_\bullet &= \calE_\bullet\cup \calE^{++}_\bullet=\calE_\bullet^{++} \\
	\calF^{\mathrm{new}}_\bullet &= \calF_\bullet\cup( \calI[\calF_\bullet + (0,\kappa)]^{++}-(\varepsilon,0)) \\ 
	\calJ_\bullet^{\mathrm{new}} &=  (\calK_\cup-(\varepsilon,0)).
	\end{split} 
	\end{align}
	We have defined $\calE^{\mathrm{new}}_\bullet,\calF^{\mathrm{new}}_\bullet$ so that they contain their predecessors. 
	
	The fact that $\calE_\bullet^{\mathrm{new}}$ satisfies $\min \Pi \calE_\bullet^{\mathrm{new}} > (d-1)/2$ follows immediately from its definition; see \cref{eq:misc_0uu}.
	
	Similarly, $\calF_\bullet^{\mathrm{new}} \subset \bbC_{\geq 1} \times \bbN$ if $K_\bullet \geq 1+\varepsilon$. The required upper bound on $K_\bullet$ (coming from the hypotheses of \Cref{lem:Osig}) is 
	\begin{equation*} 
	K_\bullet \leq c_\bullet + \min \Pi \calJ_{\bullet,>0}.
	\end{equation*} 
	Since $c_\bullet \geq 1$, we can take $K_\bullet = 1+\varepsilon$ without violating the required upper bound on $K_\bullet$. (Note that we are applying $\Cref{lem:Osig}$ with $\calJ_{\bullet,>0}$ in place of $\calJ$.)
	
	Finally, $\calJ_{\bullet}^{\mathrm{new}}$ is a subset of $\bbC_{\geq 0}\times \bbN$ if $\calK_\cup$ does not contain any $(j,k)$ with $j<\varepsilon$. For this, we can cite \Cref{rem:index_cancellation}, which says 
	\begin{equation} 
		\min \Pi \calK_\cup  \geq \min\{ \Pi \calJ_{\cup,>0}, \Pi \calI_\cup[\calF_\cup+(0,\kappa)]^+ \}\geq \min\{\varepsilon,\Pi \calI_\cup[\calF_\cup+(0,\kappa)]^+ \}\geq  \min\{\varepsilon,1\} =\varepsilon ,
	\end{equation}
	where the last `$\geq$' used \cref{eq:blah}. 
\end{proof}

We can now iterate this construction to get $O(\sigma^\infty)$-quasimodes. To simplify the induction as much as possible (and to ensure that our index sets at $\mathrm{bf}$ do not become progressively worse), we assume that 
\begin{equation} 
	\calE_\bullet = (2^{-1}(d+1),0)
\end{equation} 
for each choice of $\bullet$. This means that $f\in \calA^{1+\calE_\cup,\dots}$ decays like $\rho_{\mathrm{bf}}^{2^{-1}(d+3)}$ at $\mathrm{bf}$. Then, \cref{eq:misc_0uu} gives $\calE^{\mathrm{new}}_\bullet = \calE_\bullet$.

\begin{theorem}[Construction of $O(\sigma^\infty)$-quasimodes]
	Suppose that 
	\begin{equation}
	f\in \calA^{(2^{-1}(d+3),0),2+\calF,(0,0) }(X^+_{\mathrm{res}} ;\ell) + \sum_{l=1}^{\ell-1} \calA_{\mathrm{c}}^{(2^{-1}(d+3),0),2+\calF_l,(0,0) }([0,1)^+_{\mathrm{res}}) \calY_l, 
	\end{equation}
	where $\calF_\bullet\subset \bbC_{\geq 1}\times \bbN$. Then, there exist index sets $\calI_\bullet,\calK_\bullet\subset \bbC_{>0}\times \bbN $ and an element 
	\begin{equation}
	u\in \calA^{(2^{-1}(d-1),0),\calI, (0,0)\cup \calK}(X^+_{\mathrm{res}} ;\ell) + \sum_{l=1}^{\ell-1} \calA_{\mathrm{c}}^{(2^{-1}(d-1),0),\calI_l,(0,0)\cup \calK_l }([0,1)^+_{\mathrm{res}}) \calY_l
	\label{eq:misc_g66}
	\end{equation}
	such that 
	\begin{equation} 
	Pu-f \in  \calA^{(2^{-1}(d+3),0),\infty,\infty }(X^+_{\mathrm{res}}) = \calS([0,1)_\sigma; \calA^{(2^{-1}(d+3),0 )}(X) ).
	\label{eq:misc_n55}
	\end{equation} 
	\label{thm:C}
\end{theorem}

\begin{proof}
	We will recursively define a strictly increasing sequence $\{\alpha_n\}_{n=0}^\infty$ with $\alpha_0=0$ and $\lim_{n\to\infty} \alpha_n=\infty$ and a sequence of functions $\{u_n\}_{n=0}^\infty$ (depending on $\sigma,x$) such that, in a sense to be made precise, 
	\begin{equation}
	u \sim \sum_{n=0}^\infty \sigma^{\alpha_n} u_n.
	\label{eq:misc_j41}
	\end{equation} 
	In the process, we will also define a sequence $\{f_n\}_{n=0}^\infty$ of functions. The first of these is just $f_0 = f$. The recursive step of the algorithm involves defining $\alpha_{n+1},u_n,f_{n+1}$ in terms of $\alpha_n,f_n$.  At each step in the algorithm, it will be the case that 
	\begin{equation}
	f_n \in \calA^{(2^{-1}(d+3),0),2+\calF_{n},\calJ_n }(X^+_{\mathrm{res}} ;\ell) + \sum_{l=1}^{\ell-1} \calA_{\mathrm{c}}^{(2^{-1}(d+3),0),2+\calF_{l,n},\calJ_{l,n} }([0,1)^+_{\mathrm{res}}) \calY_l
	\label{eq:misc_5hh}
	\end{equation}
	for index sets $\calF_{\bullet,n} \subset \bbC_{\geq 1} \times \bbN$ and $\calJ_{\bullet,n} \subseteq \bbC_{\geq 0} \times \bbN$.

	Suppose that we have defined $\alpha_n,f_n$ for some $n\in \bbN$. Note that $f_n$, by \cref{eq:misc_5hh}, satisfies the hypotheses of \Cref{prop:O_sig_eps}. That proposition then states that there exists some 
	\begin{equation}
	u_{n} \in \calA^{(2^{-1}(d-1),0),\calI_{n},\calK_n  }(X^+_{\mathrm{res}} ;\ell) \\ + \sum_{l=1}^{\ell-1} \calA_{\mathrm{c}}^{(2^{-1}(d-1),0), \calI_{l,n}, \calK_{l,n} }([0,1)^+_{\mathrm{res}}) \calY_l,
	\end{equation}
	where $\calI_{\bullet,n}\subset \bbC_{\geq 1}\times \bbN$ and $\calK_{\bullet,n}\subset \bbC_{\geq 0}\times \bbN$,
	such that, for some $\varepsilon>0$ (specified in the proposition),  the function $f_{n+1}$ defined by $f_{n+1} = - \sigma^{-\varepsilon} ( Pu_{n} - f_{n})$ satisfies 
	\begin{equation}
	 f_{n+1}\in \calA^{(2^{-1}(d+3),0),2+\calF^+_{n},\calJ_{n}^+ }(X^+_{\mathrm{res}} ;\ell)  + \sum_{l=1}^{\ell-1} \calA_{\mathrm{c}}^{(2^{-1}(d+3),0),2+\calF_{l,n}^+, \calJ_{l,n}^+ }([0,1)^+_{\mathrm{res}}) \calY_l,
	\end{equation}
	where the $\calF^+_\bullet,\calJ^+_\bullet$ satisfy the same hypotheses as their predecessors. So, \cref{eq:misc_5hh} holds with $n+1$ in place of $n$, with $\calF_{\bullet,n+1}=\calF^+_{\bullet,n}$ and $\calJ_{\bullet,n+1}=\calJ^+_{\bullet,n}$. 
	
	Let $\alpha_{n+1}=\alpha_n+\varepsilon$. 
	By construction, $\alpha_{n+1}>\alpha_n$. 
	\textit{Claim: $\alpha_n\to \infty$ as $n\to\infty$.} To prove this, use \Cref{eq:K_reqs}. Indeed, it must be the case that $\calJ_{\bullet,n}^++(\varepsilon,0)$ differs from $\calJ_{\bullet,n}$ only in that the former can have more elements in $\bbC_{\geq 1}\times \bbN$ and more logarithmic terms  above indices already present. It follows that, for some $N=N(n)\in \bbN$, specifically 
	\begin{equation}
		N(n) = \# \{j \in [0,1) : j\in \Pi \calJ_{\cup n}\},
	\end{equation}
	$\alpha_{n+N}\geq \alpha_n+1$. So, the set of $\alpha_n$ is unbounded, proving the claim.\footnote{In the asymptotically Euclidean case, the $\alpha_n$'s are increasing integers, so the fact that they go to infinity is obvious. The argument above is to handle the general asymptotically conic case, where the $\alpha_n$'s may be quite arbitrary.}
	
	An asymptotic summation argument therefore shows that there exists a polyhomogeneous $u$ whose $\sigma\to 0^+$ asymptotic expansion is given by \cref{eq:misc_j41}. 
	Specifically, $u$ satisfies \cref{eq:misc_g66} for 
	\begin{equation}
		\calI_\bullet = \bigcup_{n=0}^\infty \Big[ (\alpha_n,0) + \calI_{\bullet,n} \Big], \qquad \calK_\bullet = \bigcup_{n=0}^\infty \Big[ (\alpha_n,0) + \calK_{\bullet,n} \Big], 
	\end{equation}
	which are well-defined index sets because the real parts of the first components of the elements of the $\calI_{\bullet,n},\calK_{\bullet,n}$ are bounded below.

	Then, 
	\begin{align}
		\begin{split} 
		Pu-f \sim  - f +\sum_{n=0}^\infty \sigma^{\alpha_n} Pu_n  &=  \sum_{n=0}^\infty \sigma^{\alpha_n} (P u_n-f_n+ \sigma^{\alpha_{n+1}-\alpha_n} f_{n+1}) \\
		&= \sum_{n=0}^\infty \sigma^{\alpha_n} (P u_n-f_n-(Pu_n-f_n)) = 0 , 
		\end{split} 
	\end{align}
	where $f_{-1}=0$. This means that \cref{eq:misc_n55} holds. 
\end{proof}
\begin{remark}
	It may seem surprising that $Pu-f \in \calA^{2^{-1}(d+3)}$ but $u_\bullet \in \calA^{2^{-1}(d-1)}$, since $P$ has order $-1$ at $\mathrm{bf}$ and therefore ought only induce one order of decay --- you would naturally expect that $Pu-f \in \calA^{2^{-1}(d+1)}$ corresponds to $u_\bullet \in \calA^{2^{-1}(d-1)}$. The reason why $Pu-f$ is one $\rho_{\mathrm{bf}}$ better than expected is because $N_{\mathrm{tf}}$ already contains $N_{\mathrm{bf}}$ in it, so $P-N_{\mathrm{tf}}$ is order $-2$ at $\mathrm{bf}$. When we are constructing $u$ term-by-term, the error after each step arises from applying $P-N_{\mathrm{tf}}$ to the portion of $u$ constructed so far. This gives two extra orders of decay.
	This improvement can be tracked back to \cref{eq:misc_0uu}.
\end{remark}

We can now complete the proof of the main theorem. 

\begin{theorem}
	Let $P$ be as in \S\ref{sec:op}. Then, letting $P(\sigma)^{-1}$ for $\sigma>0$ denote the inverse in \Cref{prop:Vasy_absorption} (which we are just assuming exists in the case $P\neq P^*$ where we cannot prove it), and supposing that $f$ is as in \Cref{thm:C}, we have 
	\begin{equation}
		P(\sigma)^{-1} f \in \calA^{(2^{-1}(d-1),0),\calI, (0,0)\cup \calK}(X^+_{\mathrm{res}} ;\ell) + \sum_{l=1}^{\ell-1} \calA_{\mathrm{c}}^{(2^{-1}(d-1),0),\calI_l,(0,0)\cup \calK_l }([0,1)^+_{\mathrm{res}}) \calY_l 
	\end{equation}
	where $\calI_\bullet,\calK_\bullet$ are as in \Cref{thm:C}. 
	\label{thm:main}
\end{theorem}
\begin{proof}
	Applying \Cref{thm:C}, we get $v$ in the set on the right-hand side of \cref{eq:misc_g66} such that $g=Pv-f$ is in the set on the right-hand side of \cref{eq:misc_n55}. 
	
	We apply \Cref{prop:bf_2} with $-g$ in place of $f$ to get $w \in \calS([0,1)_\sigma; \calA^{(2^{-1}(d-1),0)}(X) )$ such that $Pw=-g$. 
	Then, $u=v+w$ satisfies $Pu=f$, and $u$ is also in the set on the right-hand side of \cref{eq:misc_g66}. In particular, 
	\begin{equation} 
		u(-;\sigma) \in \calA^{2^{-1}(d-1)}(X)
	\end{equation} 
	for each $\sigma>0$. By the injectivity of $P$ between the spaces in \Cref{prop:Vasy_absorption} (recall that this is an assumption that we are making, though one that is automatically satisfied when $P(\sigma)$ is the conjugated spectral family of a Schr\"odinger operator), this implies $u = P(\sigma)^{-1} f$.
\end{proof}

\Cref{thm:A} is an immediate corollary. 

\section{A partially worked-through example (Price's law)}
\label{sec:Price}

It is instructive to work through the first several steps of the algorithm in the Euclidean case. We will not keep track of multipole expansions or other detailed pieces of information like precise index sets, except as necessary for getting Price's law (really the potential scattering analogue). This simplifies the discussion considerably.
To simplify our discussion further, we use $r$ to denote the usual Euclidean radial coordinate, and we allow a singularity at the origin, $r=0$. This will not change things in a serious way, and if the reader prefers they may just replace $r$ with $\langle r \rangle=(1+r^2)^{1/2}$ where necessary.

Let $d=3$, and consider  
\begin{equation}
P(\sigma) = \triangle+2i\sigma \partial_r + \frac{2i\sigma }{r} +V  
\end{equation}
for $V\in \langle r \rangle^{-3} C^\infty(\overline{\bbR^3})$. This is just the Euclidean Laplacian with a potential term, except conjugated by $e^{i\sigma r}$. (And studying the conjugated PDE is equivalent to studying the unconjugated PDE, for conjugated forcing.)

Let $f\in C_{\mathrm{c}}^\infty(\bbR^3)$. We will produce a non-oscillatory $u$ (a function of both $\sigma>0$ and $x\in \bbR^3$) such that \begin{equation} 
	Pu-f = O(\sigma^{2})
\end{equation} 
in a suitable function space. Then, 
\begin{equation}
P(\sigma)^{-1} f = u - \underbrace{P(\sigma)^{-1}( Pu-f)}_{\text{Controlled by \Cref{prop:nice_low_energy}.}}
\end{equation}
Some of this discussion will recapitulate the argument in \cite{HintzPrice}, but our algorithm allows one to keep going beyond the point where Hintz stops. Our goal here is just to illustrate how our algorithm works in a familiar setting. We stop after a couple of steps -- hence why we produce only an $O(\sigma^2)$ quasimode, not an $O(\sigma^\infty)$ quasimode. Subsequent rounds look similar.

The idea to construct $u$ is to build it as a finite sum
\begin{equation}
u= \sum_{j} \sigma^j \underbrace{ \sum_k  u_{j,k} (\log \sigma)^k}_{u_j}
\label{eq:misc_174}
\end{equation}
of some explicit terms $u_{j,k} \in \calA^{-\infty,-\infty,(0,0)} (X^+_{\mathrm{res}})$, each of which is constructed as in \Cref{prop:O_sig_eps}. (Note that $u_{j,k}$ is allowed to depend on $\sigma$, but is required to be a smooth function of $\rho_{\mathrm{zf}}$ near $\mathrm{zf}$.)
The $u_{j,k}$ are chosen so that $P(\sigma)u_{j,k}$ cancels to leading order with the low energy terms in 
\begin{equation}
-f+P \sum_{j'<j} u_{j',k}  \sigma^{j'} (\log \sigma)^k \overset{\mathrm{def}}{=}f_j \sigma^j . 
\label{eq:misc_0174}
\end{equation}
Finding $u_{j,k}$ will require applying $P(0)^{-1}$ to $f_j|_{\mathrm{zf}}$. (This is slightly incorrect, because $f_j$ may have $\log \sigma$ terms, so we may need to split \begin{equation}
	f_j=\sum_k f_{j,k}(\log \sigma)^k
\end{equation} 
into bits with different powers of $\log \sigma$ and then look at $f_{j,k}|_{\mathrm{zf}}$. We will see an example of this at the end of this section.)

The key piece of numerology to keep in mind throughout is that $P(0)^{-1}g$ makes sense whenever $g$ has at least \emph{two} orders of decay as $r\to\infty$: 
\begin{equation}
g\in \calA^{2+}(X), \text{ e.g. if }g\in \calA^{\calF}(X) , \calF\subset \bbC_{>0}\times \bbN.
\end{equation}
(This is consistent with Coulomb's law, the integral in which stops making sense if the charge density is decaying only like $1/\langle r \rangle^2$.)
So, the game is to choose $u_{j,k}$ so that the left-hand side  of \cref{eq:misc_0174} is decaying at $\mathrm{zf}$ \emph{and} $\mathrm{tf}$ to the requisite orders. In order to do this, we juggle solving away error terms at $\mathrm{zf},\mathrm{tf}$. We keep solving away the error terms at $\mathrm{zf}$ until we encounter terms that are too slowly decaying. Then we solve away terms at $\mathrm{tf}$ (producing logs in the process).

Let $\calY_j$ denote the span of the spherical harmonics $Y^m_j$ for $m\in \{-j,\dots,j\}$. Thus, $\calY_0$ consists of the s-orbitals, $\calY_1$ consists of the p-orbitals, $\calY_2$ consists of the d-orbitals, etc.
The $c_j,b_j$ used throughout this paper (see \cref{eq:misc_168z}, \cref{eq:misc_168}) are then $c_j=1+j$ and $b_j=j$, as discussed in \Cref{ex:Euclidean_example}.
\begin{itemize}
	\item (Step 0.) Let $\tilde{v}$ denote the decaying solution to $P(0) \tilde{v}=f$. Such a $\tilde{v}$ exists and is unique (because we are assuming the absence of a bound state/resonance at zero energy). \Cref{thm:phg0}, (applied with $\calE_\bullet=\varnothing$, $\beth,\beth_0=0$) describes this in detail.\footnote{\Cref{lem:zf_pre} includes the relevant information, except in that the lemma is phrased in terms of index sets, instead of pre-index sets.}
	Here is what we can make use of: 
	\begin{itemize}
		\item $\tilde{v}=O(1/r)$ as $r\to\infty$ (this being consistent with Coulomb's law). Indeed, all of the index sets defined in that theorem are contained in $\bbC_{\geq d-2}\times \bbN = \bbC_{\geq 1}\times \bbN$. 
		\item  The $O(1/r)$ term is spherically symmetric. So, it has the form $q/r$ for some $q\in \bbC$ (possibly zero). Indeed, applying \Cref{thm:phg0} with $\ell=1$, the index set $\calI$ describing the other harmonics (p-waves, d-waves, f-waves, etc.) is contained in $\bbC_{\geq c_1}\times \bbN = \bbC_{\geq 2}\times \bbN$. 
		\item \emph{The $O(1/r^{3-})$ term has no s-wave, i.e.\ $\calY_0$, component if the potential $V$ has no $1/r^3$ term in its large-$r$ expansion.} In other words, the $O(1/r^2)$ term consists of just p-waves and higher.\footnote{If the PDE is spherically symmetric modulo lower-order terms, that is if $\beth_0\gg 1$, then the $O(1/r^2)$ term consists of just a p-wave. Cf.\ \Cref{ex:sym}.}
		
		Indeed, in this case, we can take $\beth,\beth_0=1$. (Recall that $\beth$ is describing the orders below $1/r^3$ at which $V$ has a nonzero term, and $\beth_0\geq \beth$ is the same but only for terms which fail to be spherically symmetric.) If we dig into the definition of $\calI_\bullet$, we see that the pre-index set describing the s-wave component has the form 
		\begin{equation}
			\calI_0 =  \{(1,0)\} \uplus (2+\dot{\calI}_\cup), 
		\end{equation}
		where $\calI_\cup \subseteq \bbC_{\geq 1}\times \bbN$ is given by \cref{eq:misc_238} and $\dot{\calI}_\cup$ is the index set it generates. This is \Cref{ex:spherical_breaking}. So, $(2,0)\notin \calI_0$. This means that the s-wave term has no $O(1/r^2)$ component: it has the form $q/r+O(1/r^{3-})$.

		The reason why we have worked with pre-index sets in \S\ref{sec:0} is exactly so as to make this sort of statement. If $\calI_0$ were an index set, then $(1,0)\in \calI_0$ would imply $(2,0)\in \calI_0$. 
		\item \emph{The $O(1/r^2)$ term in $\tilde{v}$ can be computed explicitly.} 
		This involves actually going through the argument used to prove \Cref{prop:zf_computation}. This proposition states the existence of an asymptotic expansion, but the proof is constructive. 
		Here is how that proof went: we can rearrange $P(0)\tilde{v}=f$ as $\triangle \tilde{v}=g$ for $g=f-V\tilde{v}$.
		So, $\tilde{v}$ is given by applying (the exact) Coulomb's law to $g$. The $f$ term in $g$ is compactly supported, so we can apply the vanilla multipole expansion from electrostatics to conclude that its contribution to $\tilde{v}$ is such that the s-wave has no $O(1/r^k)$ term unless $k=1$, the p-wave has no $O(1/r^k)$ term unless $k=2$, and so on; see \Cref{ex:Coulomb}.

		So, it suffices to work out the decaying solution $\bar{v}$ to 
		\begin{equation} 
				\triangle \bar{v} = -V\tilde{v},
		\end{equation} 
		which is given by Coulomb's law with charge distribution $\propto V\tilde{v}$. 
		The leading order term in the large-$r$ expansion of $V\tilde{v}$ is $V_0(\theta)q/r^4$, where $V=r^{-3}V_0+O(r^{-4})$ as $r\to\infty$. To understand its contribution to $\bar{v}$, we can separate variables and study the resultant ODEs (or equivalently use the Mellin transform, which is what we do in \S\ref{sec:0}). This gives, for each harmonic $\calY_j$, a $O(1/r^2)$ term proportional to $q$ and the $j$-wave component of $V_0$. 
		\item The previous item shows also that, \emph{if $V_0$ has no non-s-wave term}, or if $q=0$, then the $O(1/r^{3-})$ part of $\tilde{v}$ consists only of the s-wave term above and an additional p-wave term at $O(1/r^2)$. Thus, 
		\begin{equation}
			\tilde{v} = \frac{q}{r} + \frac{q\frakm+\bfa\cdot \hat{\bfr}}{r^2} + O\Big(\frac{1}{r^{3-}}\Big)
		\end{equation}
		for some $\bfa\in \bbR^3$, where $\frakm$ is the constant of proportionality from the previous item.\footnote{Note that $\frakm$ is not what we called `$\frakm$' elsewhere in this paper, where $\frakm$ is a term in the metric. However, these terms enter the low-energy algorithm in the same way (up to proportionality factors), which is why we are using the same symbol.}
		
		Indeed, in this case, we can apply \Cref{prop:zf_computation} with $\beth=0$ and $\beth_0=1$ (reflecting the fact that terms breaking the spherical symmetry only enter at sub-subleading order, whereas the $V_0$ term in $V$ enters at subleading order). This suffices to tell us that all of the terms we have not computed explicitly lie in $\bbC_{\geq c_2}\times \bbN=\bbC_{\geq 3}\times \bbN$.
	\end{itemize}
	If this seems like a lot of information to keep track of, keep in mind that $\tilde{v}$ will be input to all later steps in the algorithm. 
	
	Normally, we would cut off $\tilde{v}$ away from $\mathrm{zf}$ to form a new function $v$, and this would be the final output of this round of \S\ref{sec:1}. This cutoff turns out not to matter here, for reasons which we will explain in the next paragraph. So, we skip this step. (Nothing would go wrong were we to include the cutoff. We are skipping it merely for simplicity.)
	
	This completes our first go-through of \S\ref{sec:1}, \S\ref{sec:0}. Next, we are supposed to go through \S\ref{sec:2} to improve our ``$O(\rho_{\mathrm{zf}})$-quasimode'' $\tilde{v}$ at $\mathrm{tf}$. The reason this is typically necessary is that we formed $\tilde{v}$ by inverting $P(0)$, but the terms $2i\sigma(\partial_r+1/r)$ in $P$ are as important as the Laplacian at $\mathrm{tf}$ (in the sense of b-decay), so there is no reason to ``trust'' $\tilde{v}$ away from $\mathrm{zf}$. However, going through \S\ref{sec:2} turns out to be unnecessary at this time --- that is, it suffices to apply \Cref{lem:Osig} with $K$ so negative that the $f_{\mathrm{approx}}$ in the proof of that proposition is zero, making the proposition trivial. The reason is that if we define $u_0=\tilde{v}$ and 
	\begin{equation}
		f_1 = \sigma^{-1}(P\tilde{v}- f), 
	\end{equation}
	then $f_1$ has enough decay to feed into $P(0)^{-1}$. Recall that this means $O(1/r^{2+})$ as $r\to\infty$. That this holds is a bit lucky\footnote{By ``lucky,'' we mean that it is not a consequence of \Cref{lem:zf_pre}. The index set $\calI^+_0$ coming from this lemma has $(1,0)$ in it. The lemma says that $2+\calI^+$ describes $P\tilde{v}-f$ at $\mathrm{tf}$ (here $\tilde{v}$ is what we called $u$ in that lemma). So, $(3,0)$ is a possible index in $P\tilde{v}-f$ at $\mathrm{tf}$, according to that lemma. This implies that $(2,0)$ is an index in $f_1=\sigma^{-1}(\tilde{v}-f)$ at $\mathrm{tf}$. We saw above that this term is actually absent, but the reason is an algebraic coincidence. Note that we do not make use of this algebraic coincidence in the proof of \Cref{thm:A}, \Cref{thm:C}.}:
	\begin{equation}
		\sigma^{-1}(P\tilde{v} -f) = 2i \Big(\partial_r+\frac{1}{r} \Big) \tilde{v} = 2i \Big(\partial_r+\frac{1}{r} \Big) \frac{q}{r} + O \Big(\frac{1}{r^3}\Big) = O \Big(\frac{1}{r^3}\Big),
		\label{eq:misc_1ju}
	\end{equation}
	because $\partial_r+1/r$ kills $1/r$. Here, the big-$O$ term results from applying $(\partial_r+1/r) \in \operatorname{Diff}^{1,-1}_{\mathrm{b}}(X)$ to $\tilde{v}-q/r \in \calA^2(X)$. So, $f_1|_{\mathrm{zf}}=f_1$ has \emph{three} orders of decay, more than the required $2+$. 
	
	For later reference, note that the big-$O$ term stands for 
	\begin{align}
		\begin{split} 
		O\Big(\frac{1}{r^3}\Big) &= 2i \Big(\partial_r+\frac{1}{r} \Big) \frac{q\frakm+\bfa\cdot \hat{\bfr}}{r^2} + O\Big(\frac{1}{r^{4-}}\Big)\\
		&= -2i\cdot  \frac{q\frakm+\bfa\cdot \hat{\bfr}}{r^3} +O\Big(\frac{1}{r^{4-}}\Big).
		\end{split} 
		\label{eq:misc_1jj}
	\end{align}
	
	By the way, the requirement that $f_1|_{\mathrm{zf}}\in \calA^{2+}(X)$ is encoded in \S\ref{sec:arg} in the requirement that the index sets $\calF$ all lie in $\bbC_{>0}\times \bbN$.

	So, we can continue on to the next step. The final output of this step is $u_0=u_{0,0}=\tilde{v}$. 
	\item (Step 1.)  Next, let $\tilde{v}$ denote the decaying solution to $P(0)\tilde{v} = -f_1|_{\mathrm{zf}}=-f_1$. 
	Again, \Cref{thm:phg0} tells us all we need to know about $\tilde{v}$. Note: to match our notation with that in \S\ref{sec:1}, we call $\calF$ what we called $\calE$ in \S\ref{sec:0}. In Step 0, we took these index sets to be empty, but now they are nonempty, because $f_1$ is not Schwartz at $\mathrm{bf}$. 
	\begin{itemize}
		\item 
		\emph{Most basically, $\tilde{v}= O(1/\langle r \rangle^{1-})$.} The reason why we have $\langle r \rangle^{1-}$ here and not $\langle r \rangle$ has to do with the index set $2+\calF$ describing the decay of $f_1|_{\mathrm{zf}}$. Note that 
		\begin{equation} (1,0)\in \calF.
		\end{equation}  
		Indeed, \cref{eq:misc_1ju}, \cref{eq:misc_1jj} tell us that (unless $q,\frakm,\bfa$ are all zero), $f_1$ has a $1/r^3$ term in it. 
		
		This means that, when we form $\calI_0$ (as before, see \Cref{ex:spherical_breaking} for the explicit formula, to all orders), we take $\{(c_0,0)\}\uplus \calF=\{(1,0)\}\uplus \calF$, which contains 
		\begin{equation}
			\{(1,0)\}\uplus (1,0) = \{ (1+\delta,0),(1+\delta,1) : \delta\in \bbN \}\ni (1,1). 
		\end{equation}
		The appearance of the index $(1,1)$ means we can have a $r^{-1} \log r$ term in the s-wave part of $P(0)^{-1} f_1$. 
		The proof shows that \emph{this term is proportional to the $O(1/r^3)$ s-wave term in $f_1$, which is proportional to $q\frakm$.}
		\item The index set $\calI_1$ (if we take $\beth,\beth_0=0$ in \Cref{ex:spherical_breaking}) has the form 
		\begin{equation}
			\calI_1 \subseteq (1,0) \cup \bbC_{\geq 2} \times \bbN. 
		\end{equation}
		So, the p-wave part of $\tilde{v}$ has possible logarithmic terms only starting with $r^{-2} \log r$. 
		\item If $\beth_0\geq 1$, then the other index sets $\calI_\bullet$ besides $\calI_0,\calI_1$ are all in $\bbC_{\geq 2} \times \bbN$.
	\end{itemize}
	To summarize, $\tilde{v}$ has the form 
	\begin{equation}
		\tilde{v} = \frac{Cq\frakm \log r}{r} + \frac{q_1}{r} + \frac{\calY_1}{r} + \calA^{2-}(X)
	\end{equation}
	for some $q_1\in \bbC$,
	where $C\in \bbC$ is a (nonzero) constant of proportionality (that can be worked out explicitly).
	
	At this point, it is worth explaining why we cannot just take $u_1=\tilde{v}$, like we took $u_0=\tilde{v}$ in the previous step. Suppose we tried defining 
	\begin{equation}
		f_2= \sigma^{-1}(P\tilde{v}- f_1) = 2i \Big(\partial_r+\frac{1}{r} \Big) \tilde{v} = \frac{Cq\frakm}{r^2} + \calA^{3-}(X),
	\end{equation}
	where $C$ is some (new) constant.
	The problem is that $\partial_r+1/r$ did not kill off $r^{-1} \log r$, as it kills of $1/r$. Thus, $f_2$ has a $1/r^2$ term in it, which means that $f_2|_{\mathrm{zf}}$ falls \emph{outside} the domain of $P(0)^{-1}$. 
	Cutting off $\tilde{v}$ near $\mathrm{bf}$ (which we usually do) does not help, since the issue is at $\mathrm{zf}$. So, we need to use the tools in \S\ref{sec:2}, specifically \Cref{lem:Osig}. 
	
	At this point we would usually cut off $\tilde{v}$ away from $\mathrm{zf}$. As in Step 0, we skip this. Unlike in Step 0, this will produce a fictitious singularity at $\mathrm{bf}$, but since our focus is on low-energy, we will allow this small fictitious singularity for the sake of simplifying the exposition. (The fictitious singularity is only a logarithm.) 
	
	We now use \Cref{lem:Osig}, where what we called $f$ in that lemma is now $g=P\tilde{v}-f_1$. The three index sets in that lemma, describing 
	\begin{equation}
		g = \frac{\sigma C q \frakm}{r^2} + \calA^{3-,4-,(1,0)}(X^+_{\mathrm{res}})
		\label{eq:misc_187}
	\end{equation}
	at $\mathrm{bf},\mathrm{tf},\mathrm{zf}$ respectively, are $\calE_\bullet,\calF_\bullet$, and $\calJ_\bullet=(1,0)$. We can say the following about $\calE_\bullet,\calF_\bullet$:
	\begin{itemize}
		\item The index set $1+\calE_\bullet$ describes the decay of $g$ at $\mathrm{bf}$. According to to \cref{eq:misc_187}, this will satisfy $\calE_\bullet \subset (1,0) \cup \bbC_{\geq 2}\times \bbN$. That is, $g$ has $O(1/r^2)$ decay at $\mathrm{bf}$ (without log terms) and then an $O(1/r^{3-})$ remainder.

		\textit{Note:} if we had used the cutoff version $v$ of $\tilde{v}$, then  $g=Pv-f_1$ would have the same index set as $f_1$ at $\mathrm{bf}$, and this would be in $\bbC_{\geq 3}\times \bbN$, as can be seen from \cref{eq:misc_1ju}. So, $\calE$ would be in $\bbC_{\geq 2}\times \bbN$. In particular, $(1,0)\notin \calE$. This shows why we cut-off $\tilde{v}$ in general. Failing to do so results in fictitiously large index sets at $\mathrm{bf}$. In fact, because $(1,0)=(2^{-1}(d-1),0)$, failing to cut-off gives an extra logarithmic term at $\mathrm{bf}$ in this step, when we apply \Cref{lem:Osig}. This should be contrasted with the $\calE_\bullet$ in \S\ref{sec:arg}, which we take to be $(2^{-1}(d+1),0)=(2,0)$ by the end of the section.
		\item The index set $2+\calF_\bullet$ describes the decay of $g$
		at $\mathrm{tf}$. So, $\calF_0\subseteq (1,0)\cup \bbC_{\geq 2}\times \bbN$; indeed, the $1/r^2$ in \cref{eq:misc_187} gives the leading term two orders of decay at $\mathrm{tf}$, and $\sigma$ gives another order of decay at $\mathrm{tf}$; so $2+\calF_0$ should have $(3,0)$ in it, and all the other terms are higher-order.  The other $\calF_\bullet$ are better.
	\end{itemize}
	 Recall that $K_\bullet$'s are parameters in \Cref{lem:Osig}. 
	 Take $K_0$, the s-wave $K$, to be $K_0\in (1,2]$; the point is that $K_0$ is just big enough to capture the problematic $(1,0)$ term in $\calF_0$, without being too big. Here, too big would mean $K_0>c_0+1=2$. 
	 
	 We can take the other $K_\bullet$ to be $0$ or $1$, or the same as $K_0$. The simplest choice is $K_\bullet\ll 1$, so that what we called $f_{\mathrm{approx}}$ in the proof of \Cref{lem:Osig} (we should really call it $g_{\mathrm{approx}}$ now) consists only of the leading s-wave term $C\sigma q\frakm/r^2$ in $g=P\tilde{v}-f_1$:
	\begin{equation}
		f_{\mathrm{approx}}  = \frac{C \sigma^3 q\frakm}{\hat{r}^2} = \sigma^3 \tilde{f}_{1,0}, \quad \tilde{f}_{1,0} = \frac{C q\frakm}{\hat{r}^2}
	\end{equation}
	near $\mathrm{zf}$. (This is modified away from $\mathrm{zf}$ by terms involving the cutoff used to form $v$ from $\tilde{v}$.) As expected (see the index sets in \S\ref{sec:2}), $\tilde{f}_{1,0}=1/\hat{r}^2$ lies in $\calA^{*,-1-K-}$ near $\hat{r}=0$. 
	
	The $r^{-2}$-decay of $g$ corresponds to a $\hat{r}^{-2}$-\emph{blowup} of $\tilde{f}_{1,0}$ at $\hat{r}=0$. This might look bad, but it is not so bad. Indeed, the s-wave indicial roots of the operator $\smash{\hat{N}_{\mathrm{tf}}}$ at $\hat{r}=0$ are $-1$ and $0$, which correspond to $-3$ and $-2$ decay rates in the forcing, i.e.\ $\hat{r}^{-3}$ and 
	\begin{equation} 
		\hat{r}^{-2}\sim \tilde{f}_{1,0}
	\end{equation} 
	blowups. The best forcing would be blowing up only like $\hat{r}^{-1}$ at $\hat{r}=0$, or not blowing up at all, but $\hat{r}^{-2}$ is the next best thing. The only consequence of hitting the indicial root is the occurrence of logarithmic terms.
	
	Now one applies $N_{\mathrm{tf}}^{-1} = \sigma^{-2} \hat{N}_{\mathrm{tf}}^{-1}$ to $f_{\mathrm{approx}}$ and cuts off away from $\mathrm{tf}$. Concretely, this amounts to solving the ODE 
	\begin{equation}
		\Big( - \frac{\mathrm{d}^2}{\mathrm{d} \hat{r}^2} - \Big(\frac{2}{\hat{r}} + 2i \Big) \frac{\mathrm{d}}{\mathrm{d} \hat{r}}-\frac{2i}{\hat{r}} \Big) \hat{N}_{\mathrm{tf}}^{-1} \tilde{f}_{1,0} = \frac{Cq\frakm}{\hat{r}^2 }
	\end{equation}
	with the recessive boundary condition at $\hat{r}=0$ and the non-oscillatory boundary condition at $\hat{r}=\infty$. As discussed in the appendix, this can be done by solving the homogeneous ODE (a form of Bessel's equation) and then using Duhamel's principle.
	Alternatively, we can write down $\tilde{w}=\hat{N}_{\mathrm{tf}}^{-1} \tilde{f}_{1,0} $ term-by-term. Specifically, 
	\begin{equation}
		\tilde{w} \sim \sum_{j,k} \tilde{w}_{1,0,j,k} \hat{r}^j (\log \hat{r})^k  ,\quad \tilde{w}_{1,0,j,k}\in \bbC, 
	\end{equation}
	where the indices $(j,k)$ are to-be-determined. The leading coefficient is then easily seen to be such that 
	\begin{equation}\tilde{w}_{1,0,j,k}
		 \Big(- \frac{\mathrm{d}^2}{\mathrm{d} \hat{r}^2} - \frac{2}{\hat{r}} \frac{\mathrm{d}}{\mathrm{d} r}\Big)  \hat{r}^j (\log \hat{r})^k = \frac{Cq\frakm}{\hat{r}^2}. 
	\end{equation}
	So, $j=0$, $k=1$, and 
	\begin{equation} 
		\tilde{w}_{1,0,j,k}=-Cq\frakm. 
	\end{equation} 
	So, $\tilde{w}$ is $-Cq\frakm \log \hat{r}$ plus better terms.
	Note the appearance of the logarithm, signaled by $k=1$. \emph{This is the origin of Price's law}. 
	The actual contribution of $\tilde{w}$ to $u_1$ will be $\sigma \tilde{w}$.
	Quickly, we explain the factor of $\sigma$:
	\begin{itemize}
		\item  $N_{\mathrm{tf}}=\sigma^2 \hat{N}_{\mathrm{tf}}$, capturing the fact that $P$ has order $-2$ at $\mathrm{tf}$, 
		\item $f_{\mathrm{approx}} = \sigma^3 \tilde{f}_{1,0}$, so we have one extra $\sigma$ on top of those that $N_{\mathrm{tf}}^{-1}$ absorbs.
	\end{itemize}
	So, the contribution to $u_1$ has the form
	\begin{equation}
		-C q \frakm  \sigma \log \hat{r} \propto \sigma \log (r\sigma) = \sigma (\log r+\log \sigma), 
	\end{equation}
	plus terms with better decay at $\mathrm{zf}$. 	Note that this is order $1-$ at $\mathrm{zf}$ (more specifically, $O(\rho_{\mathrm{zf}}\log \rho_{\mathrm{zf}}$). So, when we add $\tilde{w}$ to $\tilde{v}$, we will have worsened our $O(\rho_{\mathrm{zf}})$-quasimode $\tilde{v}$ to an $O(\rho_{\mathrm{zf}}^{1-})$-quasimode. But this is a tolerable loss. Any $O(\rho_{\mathrm{zf}}^{0+})$-quasimode would do. The point is that we have bettered our quasimode at $\mathrm{tf}$.
	
	Now let $u_1=-\tilde{v} + \sigma \tilde{w}$. \emph{This is the final product of step one}. Taking $u=u_0+\sigma u_1$ (we will modify this momentarily), 
	\begin{equation}
		Pu-f=(Pu_0 -f)-\sigma P\tilde{v} + \sigma^2 P\tilde{w} = \sigma f_1 - \sigma P\tilde{v}+\sigma^2 P\tilde{w}=\sigma(-g+\sigma P\tilde{w}).
	\end{equation}
	The point of $f_{\mathrm{approx}}$ was that $g\approx f_{\mathrm{approx}}$, where the $\approx$ is up to negligible terms. Negligible means that we can factor out a $\sigma$ from them without getting something with $1/r^2$ decay (or less decay).
	So, 
	\begin{align}
		\begin{split} 
		Pu-f &= \sigma^2 \Big(\frac{f_{\mathrm{approx}}-g}{\sigma}\Big) + \sigma(-f_{\mathrm{approx}}+ \sigma P\tilde{w} ) \\ 
		&= \sigma^2 \Big(\frac{f_{\mathrm{approx}}-g}{\sigma}\Big) + \sigma^2( (P-N_{\mathrm{tf}})\tilde{w} ) = \sigma^2 \Big(\frac{f_{\mathrm{approx}}-g}{\sigma}\Big) + \sigma^2 V\tilde{w} 
		\end{split} 
	\end{align}
	is the sum of a $O(\sigma^2 \rho_{\mathrm{tf}}^{3-})$  term and $-C q\frakm  V \sigma^2 \log (r\sigma) $. Note that this has the requisite amount of decay for the next step to proceed. Indeed, we need strictly more than $2$ orders of $\rho_{\mathrm{tf}}$, after factoring out the $\sigma^2$. We have $3-$ orders. This is what using \Cref{lem:Osig} accomplished.
	\item (Step $2-$.) Next, we modify $u$ by adding to it a term $u_{2,1}$ to cancel out with the $-C q\frakm V \sigma^2 \log \sigma$ term in $Pu-f$. Easily enough, this is 
	\begin{equation}
		u_{2,1} = C q \frakm P(0)^{-1}  V 
	\end{equation}
	near $\mathrm{zf}$. As we were just stating, $P(0)^{-1}  V $ makes sense.

	Conceivably, this might cancel with the $O(\sigma^2 \log \sigma)$ contribution to $u$ found in the previous part, but since $u_{2,1}$ depends on all of $V$ and not just the leading term in its large-$r$ expansion, such cancellations generically do not occur. Thus, we have reached the key observation giving Price's law: \emph{if the subleading term in $P(0)$ is nonzero, then the resolvent output generically has a $\sigma^2 \log \sigma$ singularity.}
	
	So, set $u=u_0+\sigma u_1 + \sigma^2 \log \sigma u_{2,1}$. 
	\item (Step 2.) We now keep going. We just take $f_2 = \sigma^{-2}(Pu-f)$ in place of $f_1$ and proceed as in the previous step.
	
\end{itemize}

Let us end this discussion by remarking on cutoffs. We have already indicated why cutting off $P(0)^{-1}\bullet$ away from $\mathrm{zf}$ is necessary to not produce fictitious singularities at $\mathrm{bf}$. We tolerated these fictitious singularities in the exposition above. However, when proving a full polyhomogeneity statement, we need to asymptotically sum the right-hand side of \cref{eq:misc_174} (which now has infinitely many terms). This cannot be done if the $u_{j,k}$ have progressively worse index sets at $\mathrm{bf}$ as $j\to \infty$. 
This is one reason why cutting off in \S\ref{sec:1} is necessary.

In \S\ref{sec:2}, we make use of another cutoff, localizing near $\mathrm{tf}$. If we were to use this above, it would get rid of the singularity at $r=0$.

If actually $V\in \langle r \rangle^{-4}C^\infty(\overline{\bbR^3})$, i.e.\ if $\beth\geq 1$, then the log terms above are zero. We can then repeat the argument above. For every extra order of decay of $V$, it takes one more step to get a $\log r/r$ term in the forcing, which means the $\log \sigma$ singularity is accompanied by an extra $\sigma$. So, the resolvent is one order more regular at $\sigma=0$. Taking the Fourier transform, this results in solutions of the wave equation having one more order of temporal decay.
This is the connection between the tails of the potential (or the metric, which enters the algorithm above in the analogous way) and the temporal decay rate of waves. This was proven by Morgan \cite{MorganThesis, Morgan}.

In odd $d\geq 5$ dimensions, the algorithm works similarly to that above. In even $d\geq 4$ dimensions, one encounters logarithmic singularities even if $V=0$; see the exact formula for the free resolvent. Recall the coincidence 
\begin{equation}
	(P-N_{\mathrm{zf}}) \frac{1}{r} = \Big(\partial_r + \frac{1}{r}\Big)\frac{1}{r}= 0 
	\label{eq:misc_196}
\end{equation}
we made use of above, in Step 0.\footnote{When the metric is variable, then instead equality in \cref{eq:misc_196} holds only up to sufficiently fast decaying terms.} What is going on here is that, when $d=3$, the indicial root $c_0=d-2$ describing Coulomb's law in $d$-dimensions coincides with the indicial root $2^{-1}(d-1)$ describing decay at $\mathrm{bf}$. If it were not for this coincidence, then we would have needed to make use of \S\ref{sec:2} in Step 0. This would then introduce logarithmic terms at $\mathrm{zf}$. When $d\geq 5$ is odd, the indicial root $d-2$ is larger than $2^{-1}(d-1)$ by an integral amount, so the same phenomenon is encountered in some Step $k$, $k\geq 1$. In each subsequent step, after Step 0, we decrease the decay rate of the dangerous term by $1$. Right when that decay rate becomes borderline, $P-N_{\mathrm{zf}}$ kills it, averting the appeal to \S\ref{sec:2}. However, if $d\geq 4$ is \emph{even}, then the indicial root $2^{-1}(d-1)$ of $P-N_{\mathrm{zf}}$ is a half-integer, whereas $d-2$ is an integer, so the problematic term survives. It must then be solved away by \Cref{lem:Osig}, resulting in logarithmic terms.

\appendix

\section{Analysis at exactly zero energy}
\label{sec:0}

A familiar fact from the theory of electrostatics is that given $\rho\in C_{\mathrm{c}}^\infty(\bbR^3)$, say representing the distribution of charge-density in some region of space, the unique decaying solution  
\begin{equation}
\phi(x) = \int_{\bbR^3} \frac{\rho(y)}{4\pi |x-y|} \dd^3 y
\label{eq:coulomb}
\end{equation}
to Poisson's equation $\triangle \phi = \rho$, the voltage generated by that charge-density,
satisfies $\phi \in \smash{\langle r \rangle^{-1}}C^\infty(\overline{\bbR^3})$,
where $\smash{\overline{\bbR^3}} = \bbR^3 \cup \infty \bbS^2$ is the radial compactification of $\bbR^3$, $r$ is the spherical radial coordinate, and $\langle r \rangle = (1+r^2)^{1/2}$ is the Japanese bracket. Prosaically, this means that, in the $r\to\infty$ limit, $\phi$ admits a full asymptotic expansion in terms of negative powers $r$: there exist $\phi_1(\theta),\phi_2(\theta),\cdots \in C^\infty(\bbS^2_\theta)$ such that 
\begin{equation}
\phi(x)  - \sum_{k=1}^K \frac{\phi_k(\theta)}{r^k} \in \langle r \rangle^{-K-1} C^\infty(\overline{\bbR^3}\backslash \{0\}) \subset  \langle r \rangle^{-K-1} L^\infty_{\mathrm{loc}}(\overline{\bbR^3}\backslash \{0\})
\end{equation}
for all $K\in \bbN$. It turns out that each coefficient $\phi_k$ satisfies $\phi_k\in \operatorname{span}_\bbC\{Y^{-k}_k,\cdots,Y^0_k,\cdots,Y^k_k\}$, where, for each $\ell\in \bbN$ and $m\in \{-\ell,\cdots,\ell\}$, $Y^m_\ell(\theta)\in C^\infty(\bbS^2_\theta)$ denotes the usual spherical harmonic. This is the \emph{multipole expansion}. Explicit formulae can be written down for all the $\phi_k$'s. For instance, $\phi_0$ is proportional to the ``monopole moment'' of $\rho$, $\phi_1$ to the ``dipole moment,'' and so on; see \cite{Jackson}.

The following two generalizations are less well-known:
\begin{itemize}
	\item For any $\alpha \in \bbN_{\geq 3}$ and $\rho \in \langle r \rangle^{-\alpha} C^\infty(\overline{\bbR^3})$, the voltage $\phi$, defined by the same formula \cref{eq:coulomb}, also admits an asymptotic expansion as $r\to\infty$, but (unless $\rho$ is Schwartz), it is necessary to have 
	logarithmic terms in the expansion, 
	\begin{equation} 
	\phi_{j,1}(\theta) r^{-j} \log(r)
	\end{equation} 
	for certain $j\in \bbN_+$. 
	Moreover, rather than each spherical harmonic $\smash{Y_k^m}$ showing up in the expansion only proportional to $1/r^{k+1}$, they can also show up in later terms (but not in earlier ones).   
	\item  If one solves instead an equation such as $\triangle \phi + V \phi = \rho$ for 
	\begin{equation} 
	V \in \langle r \rangle^{-3} C^\infty_{\mathrm{c}}[0,1)_{1/r} + \calS(\bbR^3),
	\label{eq:perf_V}
	\end{equation} 
	i.e.\ a short-range potential which is spherically symmetric modulo Schwartz terms, then the situation is similar. 
	
	For $V\in \langle r \rangle^{-3} C^\infty(\overline{\bbR^3})$ not of the form \cref{eq:perf_V}, two new phenomena can occur:
	\begin{enumerate}[label=(\roman*)]
		\item it may now be possible to find logarithmic terms of the form $\phi_{j,\kappa}(\theta) r^{-j} \log(r)^\kappa$ for $\kappa \geq 2$, not just $\kappa=1$, and
		\item more seriously, it may be the case that the spherical harmonics $Y_k^m$ are no longer only found proportional to $r^{-j} (\log r)^\kappa$ for $j\geq k+1$; they may be present in less-decaying terms as well. In fact, the $\phi_{j,\kappa}$ may no longer be linear combinations of \emph{finitely} many spherical harmonics at all. 
	\end{enumerate}
	However, these problems only kick in late in the $r\to\infty$ expansion. Since they are tied to symmetry-breaking terms in $V$, the more rapidly decaying these terms, the later in the expansion problems arise. 
	It is the essence of the multipole expansion that only finitely many spherical harmonics are present at each individual order. So, in the present situation, one should really speak only of ``partial'' multipole expansions. One still has an asymptotic expansion in the $r\to\infty$ limit, it just fails to be a multipole expansion past some particular order. 
\end{itemize}
As the presence of logarithmic terms indicates, these results require a different proof than that presented in physics textbooks for the multipole expansion for Laplace's equation, which is based on manipulations of \cref{eq:coulomb}. For instance, if the source $\rho$ is not Schwartz, then its high-order multipole moments will not be defined, as the integrals which define them in the Schwartz case fail to be convergent. Moreover, if $V\neq 0$, then \cref{eq:coulomb} no longer applies.

In each of the situations described above, $\phi$ is polyhomogeneous on $\overline{\bbR^3}$,  with some index set $\calE\subset \bbC\times \bbN$ keeping track of precisely which pairs $(j,\kappa)$ are present in the asymptotics. Moreover, for each spherical harmonic $Y$, one can compute an index set $\calE_Y\subsetneq \calE$ keeping track in which terms that spherical harmonic shows up. This subset depends, in a rather complicated way, on the decay rate of the source $\rho$, the decay rate of the potential, and the decay rate of symmetry breaking-terms. Regardless, it can be computed. 

Our goal of this section is to carry out this computation in the natural level of generality, that of Schr\"odinger operators on (compact) manifolds-with-boundary equipped with an exactly-conic structure at their boundary, which is precisely the setup introduced in \S\ref{sec:op}.  To keep track of the relevant data, we use the function spaces
\begin{equation}
\calA^{\calE,\alpha}(X;\ell) + \sum_{l=0}^{\ell-1} \calA_{\mathrm{c}}^{\calE_l,\alpha_l}([0,1)_\rho) \calY_l, \label{eq:func_spaces}
\end{equation}
where $\ell\in \bbN$, $\calE,\calE_l$ are pre-index sets, and $\alpha,\alpha_l \in \bbR$. These are defined as follows. For each $l\in \bbN$, each $\calY_l$ is a (positive dimensional) subspace of some eigenspace of the boundary Laplacian\footnote{As elsewhere in this paper, Laplacians are positive semidefinite.} $\triangle_{\partial X} \in \operatorname{Diff}^2(\partial X)$, and $\calA^{\calE,\alpha}(X;\ell)$ consists of those elements $v$ of $\calA^{\calE,\alpha}(X)$ such that 
\begin{equation} 
\langle v(\rho,\theta), Y(\theta) \rangle_{L^2(\partial X_\theta)} =0
\end{equation} 
for all $\rho \in (0,1/2)$ and $Y\in \calY_l$ for $l \in \{0,\cdots,\ell-1\}$. In particular, $\calA^{\calE,\alpha}(X;0) = \calA^{\calE,\alpha}(X)$. Note that these function spaces depend on the precise choice of exactly-conic structure imposed on $X$. An $f$ lies in the space \cref{eq:func_spaces} if and only if we can write
\begin{equation}
f = F + \sum_{l=0}^{\ell-1} \sum_{Y\in \operatorname{onb}(\calY_l)} f_Y(\rho) Y(\theta)  
\end{equation}
for some $F\in \calA^{\calE,\alpha}(X;\ell)$ and $f_Y\in \calA_{\mathrm{c}}^{\calE_l,\alpha_l}([0,1)_\rho)$, where $\operatorname{onb}(\calY_l)$ is, for each $l$, some orthonormal basis of $\calY_l$. 
We require that $\calY_j\perp \calY_k$ as subspaces of $L^2(\partial X)$ if $j\neq k$.
For each $l\in \bbN$, let $\lambda_l$ be the eigenvalue under $\triangle_{\partial X}$ of the elements of $\calY_l$. 
Note that some $\lambda_l$'s may be repeated.
We require that the sequence $\lambda_0,\lambda_1,\lambda_2,\cdots$ be non-decreasing, and we require that each eigenspace of $\triangle_{\partial X}$ be the direct sum of some of the $\calY_l$'s. 
The simplest two possibilities are when
\begin{itemize}
	\item $\calY_0,\calY_1,\calY_2,\cdots$ are the eigenspaces of the boundary Laplacian $\triangle_{g_{\partial X}}$, listed in order of increasing eigenvalue, or
	\item or each $\calY_l$ is given by $\calY_l = \bbC Y_l$, where $Y_l$ is the $l$th eigenfunction of the boundary Laplacian, listed in order of increasing eigenvalue (where, if two eigenfunctions share an eigenvalue, their ordering is arbitrary).
\end{itemize}
All other cases are intermediate to these two.

\begin{theorem}
	Let $P\in \operatorname{Diff}(X^\circ)$ have the form described in \S\ref{subsec:0main}, with $\aleph,\aleph_0=\infty$. Suppose that $f\in \calA^{\calE}(X;\ell) + \sum_{l=0}^{\ell-1} \calA_{\mathrm{c}}^{\calE_l}([0,1)_\rho) \calY_l$. Then, if $u\in \calA^{0+}(X)$ satisfies $Pu=f$,
	\begin{equation}
	u\in \calA^{\calI}(X;\ell) + \sum_{l=0}^{\ell-1} \calA^{\calI_l}_{\mathrm{c}}[0,1) \calY_l
	\label{eq:main}
	\end{equation}
	holds, 
	where $\calI,\calI_l$ are defined by \cref{eq:index_set_defs}.
	\label{thm:phg0}
\end{theorem}
\begin{remark*}
	Note that when $P$ differs from $\triangle_{g_0}$ by Schwartz terms, and when $f$ is Schwartz, then this reproduces the usual multipole expansion from electrostatics. So, the index sets are sharp, at least in this case. Though we will not do so, it should be possible to show that $\calI,\calI_l$ are sharp in general, in the sense that, for each choice of parameters $\calE,\alpha,\calE_l,\beth,\beth_0$ involved in the assumptions on $P,f$, there exist some $P,f$ of the hypothesized forms and some $u$ solving $Pu=f$ such that the index sets in \cref{eq:main} are optimal. 
\end{remark*}
\begin{remark*}
	In the asymptotically Euclidean case, one can shrink $\calI,\calI_l$ if the terms which break the spherical symmetry only involve finitely many spherical harmonics at each order. That is, if the coefficients of the PDE admit a multipole expansion, then so do the solutions. This improvement is a consequence of the Clebsch--Gordan decomposition and rather special to the asymptotically Euclidean case. 
\end{remark*}

The proof of \Cref{thm:phg0} consists of two basic steps:
\begin{enumerate}[label=(\Roman*)] 
	\item First, the result is proven for $P=\triangle_{g_0}$, solving the inhomogeneous problem $Pu=f$ explicitly via the Mellin transform. In this case, the result follows from the mapping properties of the Mellin transform. In preparation for the next step, it will be necessary to to handle the case when $f$ is only partially polyhomogeneous. This is not contained in \Cref{thm:phg0}, but it is a straightforward extension.
	\item  For general $P$, we write $P=\triangle_{g_0} + L$ for some $L$ which decays faster than $\triangle_{g_0}$ at $\partial X$ (in the sense of regular singular differential operators). Then, when solving $Pu=f$, we rearrange to get 
	\begin{equation} 
	\triangle_{g_0} u = f - L u.
	\end{equation} 
	Suppose that $u$ is known to be symbolic or partially polyhomogeneous, with the symbolic error kicking in at order $\alpha\in \bbR$. (We are assuming in \Cref{thm:phg0} that $u\in \calA^{0+}(X)$, so this assumption holds for some $\alpha>0$.) Then, $Lu$ is partially polyhomogeneous as well. Applying $\triangle_{g_0}^{-1}$ (making sense of this on appropriate function spaces using the solvability theory of the Laplacian), we write 
	\begin{equation}
	u = \triangle_{g_0}^{-1}(f-Lu) = \triangle_{g_0}^{-1} f  - \triangle_{g_0}^{-1} L u.
	\label{eq:induc}
	\end{equation}
	The term $\triangle_{g_0}^{-1} f$ is understood from step (I) and fully polyhomogeneous. 
	Likewise, from step (I), 
	\begin{equation} 
		w=\triangle_{g_0}^{-1} Lu
	\end{equation} 
	is partially polyhomogeneous, but because $L$ is faster decaying than $\triangle_{g_0}$, the composition $w$ will be faster decaying than $u$. This applies also to the symbolic errors involved, so, in $w$, the symbolic error kicks in at some order $\geq \alpha+1$. But then \cref{eq:induc} implies that the symbolic error in $u$ only kicks in at order $\geq \alpha+1$ as well, at least one order higher than we started out assuming. A straightforward inductive argument then allows us to take $\alpha\to\infty$ and therefore conclude full polyhomogeneity. 
\end{enumerate}

Part of Price's law states that gravitational radiation from spherically symmetric black holes radiates away at a rate which depends on the angular pattern according to which the radiation is distributed. Monopole (i.e.\ spherically symmetric) radiation radiates away more slowly than dipole radiation, dipole radiation radiates away more slowly than quadropole radiation, and so on. 
As a leading order asymptotic, it is known, from Hintz \cite{HintzPrice}, that Price's law applies to radiation on any asymptotically sub-extremal Kerr background. However, it is to be expected in this context that e.g.\ dipole radiation includes a quadropole component which decays more slowly than pure quadropole radiation, and moreover that the index sets describing the long-time asymptotic profile are more complicated than in the spherically symmetric case. These subleading effects are not studied in \cite{HintzPrice}. We would like to investigate these delicate effects, which depend on the fact that Kerr is only axially symmetric. This requires the multipole expansions whose systematic discussion we have included here.

\subsection{Some mapping properties of the exactly conic Laplacian}
We next discuss the mapping properties of $\triangle_{g_0}$ vis-a-vis the function spaces in \cref{eq:func_spaces}. This is step (I) in the outline above. The exactly-conic Laplacian $\triangle_{g_0}$ is given by
\begin{equation}
\triangle_{g_0} = - \frac{\partial^2}{\partial r^2} - \frac{d-1}{r} \frac{\partial}{\partial r} + \frac{1}{r^2} \triangle_{g_{\partial X}}
\end{equation}
in $\dot{X} = [0,1)_\rho\times \partial X_\theta$. 

This operator is an elliptic element of $\rho^2 \operatorname{Diff}^2_{\mathrm{b}}(X)$, i.e.\ is regular singular at $\partial X$, so its solvability theory is governed by its indicial roots, which are the roots of the polynomial $-c^2 + (d-2) c + \lambda \in \bbR[c]$ for $\lambda\geq 0$ any eigenvalue of the boundary Laplacian $\triangle_{\partial X}=\triangle_{g_{\partial X}}$. If $c$ is a root of this polynomial, then this corresponds (as we will see below) to $O(\rho^{c-})$ bounds for elements of $\ker \triangle_{g_0}$ or $O(\rho^{2+c-})$ bounds for  $\operatorname{coker} \triangle_{g_0}$. For each value of $\lambda$, the two roots of the given quadratic polynomial are $2^{-1}(d-2\pm ((d-2)^2+4\lambda)^{1/2}$. 
Each of the $\calY_l$ in \cref{eq:func_spaces} is a set of eigenfunctions of $\triangle_{\partial X}$ with shared eigenvalue $\lambda_l$. 
For each $l\in \bbN$, let 
\begin{align}
	c_l&=\frac{1}{2}(d-2+\sqrt{(d-2)^2+4\lambda_l})
	\label{eq:misc_168z}\\ 
	\intertext{denote the positive of these two solutions  for $\lambda=\lambda_l$.  Also, let } 
	-b_l &= \frac{1}{2}(d-2-\sqrt{(d-2)^2+4\lambda_l})
	\label{eq:misc_168} 
\end{align}
denote the corresponding negative solution. (The sign on the left-hand side in \cref{eq:misc_168} is chosen for later convenience; $c_l,b_l\geq 0$.)
In particular, $c_0 = d-2$ and $b_0 = 0$. 

Note that $-b_k \leq -b_j < c_j \leq c_k$ whenever $j\leq k$. Thus, the interval $(-b_j,c_j)$ widens as $j$ grows. This is important in low energy scattering theory. Also, via Weyl's law, $\lim_{l\to\infty} b_l, \lim_{l\to\infty} c_l = \infty$. So, the set $\cup_{l\in \bbN}\{-b_l,c_l\}$ is discrete. 

\begin{example}[Euclidean case]
	In the Euclidean case, $\partial X= \bbS^{d-1}$, and $g_{\partial X}$ is the standard metric on the $(d-1)$-sphere. So, assuming that we take $\calY_l$ to denote the $l$th eigenspace of $\triangle_{\partial X}$, we have $\lambda_l = l(l+d-2)$. Then, one computes that $c_l = d-2+l$ and $b_l = l$. 
	\label{ex:Euclidean_example}
\end{example}

\begin{lemma}
	Let $\ell\in \bbN$. 
	For $\alpha \in ( - b_\ell,c_\ell)$ and $\beta \in (0,d-2)$, we have a well-defined inverse 
	\begin{equation} 
	\triangle_{g_0}^{-1} : \calA^{2+\alpha}(X;\ell) + \calA^{2+\beta}(X) \to \calA^{\alpha-}(X;\ell) + \calA^{\beta-}(X). 
	\end{equation} 
	That is, for each $f\in \calA^{2+\alpha}(X;\ell) + \calA^{2+\beta}(X) $, there exists a unique $u\in \calA^{\alpha-}(X;\ell) + \calA^{\beta-}(X)$ such that $\triangle_{g_0} u=f$. 
\end{lemma}
The key ingredient, beyond elementary aspects of the analysis of the Laplacian, is the Mellin transform $\calM$, which we define with the conventions in \cref{eq:Mellin} and some key properties of which we summarize in \S\ref{sec:Mellin}. 
\begin{proof}
	First, one shows that $\triangle_{g_0}$ has trivial kernel acting on $\calA^{\alpha-}(X;\ell) + \calA^{\beta-}(X)$.
	Since the operator has real coefficients, its kernel in this space has a basis consisting of real-valued functions. Suppose that $w\in \calA^{\alpha-}(X;\ell) + \calA^{\beta-}(X)$ is a real-valued function satisfying $\triangle_{g_0} w=0$. Fix $\chi \in C_{\mathrm{c}}^\infty [0,1/2)$ that is identically $1$ on a neighborhood of the origin. Let $f= \rho^{-2} \triangle_{g_0} (\chi w)$. Integrating-by-parts, $\calM f  (c,\theta) = (-c^2 + (d-2) c + \triangle_{\partial X}) \calM(\chi w)(c,\theta)$
	as long as $c$ has sufficiently negative real part. Moreover, as long as $c$ does not lie in the discrete set of points $\cup_{l\in \bbN} \{-b_l,c_l\}$, the resolvent $(-c^2 + (d-2) c + \triangle_{\partial X})^{-1} : H^m(\partial X_\theta) \to H^{m}(\partial X_\theta)$ is well-defined, giving 
	\begin{equation}
	\calM(\chi w)(c,\theta)  = (-c^2 + (d-2) c + \triangle_{\partial X})^{-1} \calM f (c,\theta).
	\label{eq:id_099}
	\end{equation}
	Since $f = \rho^{-2} \triangle_{g_0} (\chi w) = \rho^{-2} [\triangle_{g_0},\chi] w \in C_{\mathrm{c}}^\infty(\dot{X}^\circ)$, the Mellin transform $\calM f (c,\theta)$ is defined for all $c$. 
	The right-hand side of \cref{eq:id_099} therefore defines a meromorphic function of $c$, defined on the entire complex plane $\bbC_c$, save for the poles. We will call this extension 
	\begin{equation} 
	\calM(\chi w)_{\mathrm{ext}} :  \bbC_c \backslash \cup_{l\in \bbN} \{-b_l,c_l\}\to \bbC.
	\end{equation} 
	The possible poles, all of which are simple (as follows from the functional calculus applied to $\triangle_{\partial X}$),  all lie in $\cup_{l\in \bbN} \{-b_l,c_l\}$. 
	Any apparent singularities of $\calM(\chi w)_{\mathrm{ext}}$ at the $b_0,b_1,b_2,\cdots$ are necessarily removable:
	\begin{itemize}
		\item 
		That this holds for $b_\ell,b_{\ell+1},b_{\ell+2},\cdots$ follows from the fact that $w \in \calA^{-b_\ell+\varepsilon}(X)$ for some $\varepsilon>0$
		(which follows from the assumptions on $\alpha,\beta$) and therefore that the Mellin transform  $\calM(\chi w)$ is already well-defined and analytic when $\Re c <- b_\ell + \varepsilon$. 
		\item 
		That this holds for the remaining $b_l$'s, $b_0,\ldots,b_{\ell-1}$, follows from a stronger argument. By assumption, $w= u+v$ for $u\in \calA^{-b_\ell+}(X;\ell)$ and $v\in \calA^{0+}(X)$. The identity
		\begin{equation}
		\calM(\chi w)(c,\theta) = \calM (\chi u)(c,w) + \calM(\chi v)(c,\theta) 
		\label{eq:j67}
		\end{equation}
		holds for all $c$ with sufficiently negative real part. We know, a priori, that $\calM(\chi v)(c,\theta)$ is analytic in a neighborhood of the left half of the complex plane. On the other hand, 
		\begin{equation}
		\calM (\chi u)(c,w)  =  (-c^2 + (d-2) c + \triangle_{\partial X})^{-1} \calM (\rho^{-2} \triangle_{g_0}(\chi u) ) (c,\theta)
		\end{equation}
		if $c$ has sufficiently negative real part. Because $\triangle_{g_0} w=0$ and therefore $\triangle_{g_0} u = - \triangle_{g_0} v \in \calA^{2+}(X) $, we have 
		\begin{equation} 
		\rho^{-2} \triangle_{g_0}(\chi u) = \rho^{-2} [\triangle_{g_0},\chi] u - \rho^{-2} \chi \triangle_{g_0} v \in \calA^{0+}(X).
		\end{equation} 
		So, $\calM (\rho^{-2} \triangle_{g_0}(\chi u) ) (c,\theta)$ is also analytic in some small neighborhood of the left half of the complex plane. It follows that $(-c^2 + (d-2) c + \triangle_{\partial X})^{-1} \calM (\rho^{-2} \triangle_{g_0}(\chi u) ) (c,\theta)$ is meromorphic on some neighborhood of the left half of the complex plane, with possible simple poles at $c\in \cup_{l\in \bbN} \{-b_l,c_l\}$.

		But, since $u$ is orthogonal to all of the $\calY_0,\cdots,\calY_{\ell-1}$ in $\{\rho<1/2\}$, the function  $\rho^{-2} \triangle_{g_0} (\chi u)$ is as well, for all $\rho$, since $\operatorname{supp} \chi \subset \{\rho<1/2\}$. Therefore 
		\begin{equation} 
		\calM (\rho^{-2} \triangle_{g_0}(\chi u) ) (c) \perp \calY_0,\cdots,\calY_{\ell-1}
		\end{equation} 
		for all $c$. Consequently, the residues of $(-c^2 + (d-2) c + \triangle_{\partial X})^{-1} \calM (\rho^{-2} \triangle_{g_0}(\chi u) ) (c,\theta)$ at the various points $b_0,\cdots,b_{\ell-1}$ are all zero. (That is, unless some of these are equal to $b_\ell$, which can happen if $\calY_{\ell-1}$ is a proper subspace of some eigenspace of the boundary Laplacian. To avoid having to repeat this caveat, when we say that a pole at $b_l$ or $c_l$ is absent, we mean that it is absent or that $b_{l+1}=b_l$ or $c_{l+1}=c_l$, in which case we can think of the pole as being associated with the index $l+1$ instead of $l$.)
		
		So, both terms on the right-hand side of \cref{eq:j67} extend analytically to a neighborhood of the left half of the complex plane, and thus the same applies to $\calM(\chi w)(c,\theta)$. 
	\end{itemize}

	If, for $m\in \bbN$, we define the norm on $H^{m}(\partial X)$ using $\triangle_{\partial X}$, then we have the operator-norm bound 
	\begin{equation} 
	\lVert (-c^2 + (d-2) c + \triangle_{\partial X})^{-1} \rVert_{H^m(\partial X) \to H^m(\partial X)} \leq d(-c^2 + (d-2) c, \sigma(\triangle_{\partial X}) )^{-1} 
	\end{equation} 
	for functional analytic reasons. Consequently, if $\gamma_0,\gamma_1$ are real numbers such that $\gamma_0<\gamma_1$ and $[\gamma_0,\gamma_1]$ is disjoint from $\cup_{l\in \bbN}\{-b_l,c_l\}$, we have, for all $c\in \bbC$ satisfying $\Re c \in [\gamma_0,\gamma_1]$, the bound 
	\begin{equation}
	\lVert \calM(\chi w )_{\mathrm{ext}} (c,\theta) \rVert_{H^m(\partial X_\theta) } \leq C(\gamma_0,\gamma_1) \langle \Im c \rangle^{-2} \lVert \calM f (c,\theta) \rVert_{H^m(\partial X_\theta) } 
	\label{eq:mkk}
	\end{equation} 
	for some $C(\gamma_0,\gamma_1)>0$. Since $\calM(\chi w )_{\mathrm{ext}} (c,\theta)$ is analytic at each $b_l$, the same estimate, possibly with a larger value of $C(\gamma_0,\gamma_1)$, holds for any $\gamma_0,\gamma_1 \in \bbR $ satisfying $\gamma_0<\gamma_1$ and $\gamma_1<c_0=d-2$.

	We apply the Mellin inversion formula \cref{eq:Mellin_inversion} to $\chi w$, noting that \cref{eq:mkk} justifies shifting the contour to the right, as long as we replace $\calM (\chi w) (c,\theta)$ by the analytic extension $\calM (\chi w)_{\mathrm{ext}} (c,\theta)$. So, 
	\begin{equation}
	\chi w(\rho,\theta)  = \frac{1}{2\pi i} \int_{\gamma-i \infty}^{\gamma+i\infty}  \rho^{c} \calM (\chi w)_{\mathrm{ext}} (c,\theta)  \dd c
	\end{equation}
	for any $\gamma<d-2$. It follows that $\chi w \in \calA^{c_0-}(X) \subseteq \calA^{0+}(X)$, as discussed in \S\ref{sec:Mellin}. Therefore, $w\in C^0(X)$. 
	
	Since $X$ is compact, continuity implies that, unless $w=0$, it must be that $w$ (which recall is real-valued) attains a global extremum somewhere in the interior of $X$. But then the maximum principle for Laplace--Beltrami operators would state that $w$ is constant. Since $w\in \calA^{0+}(X)$, $w$ being constant implies that $w=0$ identically. This completes the proof that the kernel of $\triangle_{g_0}$ acting on $\calA^{\alpha-}(X;\ell) + \calA^{\beta-}(X)$ is trivial. 
	
	In order to show that $\triangle_{g_0}:\calA^{\alpha-}(X;\ell) + \calA^{\beta-}(X)\twoheadrightarrow \calA^{2+\alpha}(X;\ell) + \calA^{2+\beta}(X)$, one also uses the Mellin transform. It suffices to consider the cases (i) $f\in \calA^{2+\alpha}(X;\ell)$, (ii) $f\in \calA^{2+\beta}(X)$ individually.
	\begin{enumerate}[label=(\roman*)]
		\item If $f\in \calA^{2+\beta}(X)$, then the existence of $u\in \calA^{\beta-}(X)$ satisfying $\triangle_{g_0} u=f$ follows from the standard solvability theory of the Laplacian, e.g.\ \cite[Eq.\ 2.9]{HintzPrice}  for the $d=3$, asymptotically Euclidean case, the general case being analogous. 
		\item If $f\in \calA^{2+\alpha}(X;\ell)$, then we define a meromorphic function $ 
		M(c,\theta)$ by 
		\begin{equation}
		M(c,\theta) = (-c^2 + (d-2) c + \triangle_{\partial X})^{-1} \calM(\rho^{-2} \chi f)(c,\theta),
		\end{equation} 
		where $\chi$ is as above. 
		As $ \calM(\rho^{-2} \chi f)(c,\theta)$ is well-defined and analytic in $c$ for $\Re c < \alpha$, the function $M(c,\theta)$ has at worst poles there. But actually, because $f$ is orthogonal to each $\calY_0,\ldots,\calY_{\ell-1}$, it follows that the possible poles at $ b_0,\ldots,b_{\ell-1}$ (some of which may be $\geq \alpha$) are all absent, as well as those at $c_0,\ldots,c_{\ell-1}$. So, $M(c,\theta)$ is well-defined and analytic for $\Re c<\alpha$ as well.
		(However, we cannot rule out poles at $b_\ell,b_{\ell+1},\cdots$.)

		Analogous estimates  to those above show that 
		\begin{equation}
		u_0 = \frac{1}{2\pi i } \int_{\gamma-i\infty}^{\gamma+i\infty} \rho^c M(c,\theta) \dd c 
		\label{eq:id_jh6}
		\end{equation}
		is, for each $\gamma \in (-b_\ell,\alpha)$, a well-defined element of $\calA^{\gamma}(\dot{X};\ell)$ that does not depend on $\gamma$. Thus, $u_0 \in \calA^{\alpha-}(\dot{X};\ell)$. 
		
		Let $w= \chi u_0 \in \calA^{\alpha-}(X;\ell)$. This satisfies $\triangle_{g_0}w  = [\triangle_{g_0},\chi] u + \chi \triangle_{g_0} u_0$. We compute that
		\begin{align}
		\begin{split} 
		\triangle_{g_0} u_0 &= \frac{\rho^2}{2\pi i } \int_{\gamma-i\infty}^{\gamma+i\infty} \rho^c (-c^2 + (d-2) c + \triangle_{\partial X}) M(c,\theta) \dd c  \\ &= \frac{\rho^2}{2\pi i } \int_{\gamma-i\infty}^{\gamma+i\infty} \rho^c \calM(\rho^{-2} \chi f)(c,\theta) \dd c	 = \chi f. 
		\end{split} 
		\end{align}
		So, $\triangle_{g_0} w = f + f_0 $ for $f_0 = (\chi^2-1) f + [\triangle_{g_0},\chi] u \in C_{\mathrm{c}}^\infty(\dot{X}^\circ)$. 
		
		By the case (i) already discussed, there exists a $v \in \calA^{c_0-}(X)\subset \calA^{\beta-}(X) $ such that $\triangle_{g_0} v = -f_0$, so the function $u=w+v$, which lies in the desired function space, satisfies $\triangle_{g_0} u=f$. 
	\end{enumerate} 
\end{proof}

In order to get more detailed asymptotics, it is necessary to utilize the mapping properties of the Mellin transform on (partially) polyhomogeneous functions, as described in \S\ref{sec:Mellin}. 
Suppose that $f$ is a meromorphic function on a strip $\{z\in \bbC: \Re z \in (\gamma_0,\gamma_1)\}$ such that, for each $c$ in the strip, the pole at $f$ at $c$ is at worst of order $k+1$, where $k$ is the largest nonnegative integer such that $(c,k)\in \calE$. Then, we say that the poles of $f$ are ``given'' by the pre-index set $\calE$.

\begin{lemma}
	Suppose that $u  \in \calA^{0+}(X)$ satisfies $\triangle_{g_0} u =f$ for some function $f\in \calA^{2+\calE,2+\alpha}(X;\ell) + \sum_{l=0}^{\ell-1} \calA^{2+\calE_l,2+\alpha_l}_{\mathrm{c}}([0,1)_\rho) \calY_l$, where $\calE,\calE_l$ are pre-index sets and $\alpha,\alpha_l\in \bbR$.  
	Then, 
	\begin{equation} 
	u \in \calA^{\calC_\ell\uplus  \calE,\alpha-}(X;\ell) + \sum_{l=0}^{\ell-1} \calA^{\{(c_l,0)\} \uplus \calE_l, \alpha_l- }([0,1)_\rho) \calY_l,
	\end{equation}
	where $\calC_\ell\uplus \calE = \bigcup_{l=\ell}^\infty (\{(c_l,0)\}\uplus \calE)$. 
	\label{lem:Lap_basic}
\end{lemma}
Here, we use $\{(c_l,0)\}$ to denote the pre-index set consisting of \emph{just} the single point $(c_l,0)$, whereas elsewhere we have used the standard abbreviation in which $(c_l,0)$ stands for the index set generated by this point. 
The $\uplus$ operation was defined in \S\ref{sec:bf}.
\begin{proof}
	Let $\chi$ be as in the previous lemma, and let $w= \chi u \in C^\infty(\dot{X^\circ})$. It suffices to show that $w\in \calA^{\calC_\ell \uplus \calE,\alpha-}(\dot{X};\ell) + \sum_{l=0}^{\ell-1} \calA^{(c_l,0) \uplus \calE_l, \alpha_l- }([0,1)_\rho) \calY_l$. 
	
	Taking the Mellin transform, $\calM w(c,\theta)$ is defined and analytic in $\Re c<\varepsilon$ for some $\varepsilon>0$, and the Mellin inversion formula reads
	\begin{equation}
	w(\rho,\theta) = \frac{1}{2\pi i} \int_{\gamma-i\infty}^{\gamma+i\infty} \rho^c \calM w(c,\theta) \dd c = \calM^{-1}( \calM w(-,\theta))(\rho), 
	\label{eq:Mel_inv_ex}
	\end{equation}
	for any $\gamma<\varepsilon$. 
	
	Now let 
	$g=\triangle_{g_0} w = [\triangle_{g_0},\chi] u + \chi f$. The Mellin transform $\calM_{\rho\to c} (\rho^{-2} g(\rho,\theta))(c)$ is defined and analytic if $c$ has sufficiently negative real-part, and, for $c$ with sufficiently negative real part, taking the Mellin transform of the PDE $g=\triangle_{g_0} w$ yields $(-c^2+ (d-2)c + \triangle_{\partial X})\calM w(c,\theta) = \calM_{\rho\to c}(\rho^{-2} g(\rho,\theta))(c)$. Therefore we can write
	\begin{equation} 
	\calM w (c,\theta) = \calR(c) \calM_{\rho\to c}(\rho^{-2} g(\rho,\theta))(c,\theta),
	\label{eq:id_022}
	\end{equation} 
	where $\calR(c)=(-c^2 + (d-2) c + \triangle_{\partial X})^{-1}$, as long as $c$ has sufficiently negative real part and $c \notin \cup_{l\in \bbN} \{-b_l,c_l\}$.

	We can write $g=G + \sum_{l=0}^{\ell-1} g_l$, where $G\in \calA_{\mathrm{c}}^{2+\calE,2+\alpha}(\dot{X};\ell)$, and $g_{l} \in \calA_{\mathrm{c}}^{2+\calE_l,2+\alpha_l}[0,1) \calY_l$. Then, \cref{eq:id_022} becomes
	\begin{equation}
	\calM w (c,\theta) =  \calR(c) \calM_{\rho\to c}(\rho^{-2} G(\rho,\theta))(c)  + \sum_{l=0}^{\ell-1} \calR(c) \calM_{\rho\to c}(\rho^{-2} g_l (\rho,\theta))(c) .
	\end{equation}
	\begin{itemize}
		\item First consider the term $\calR(c) \calM_{\rho\to c}(\rho^{-2} G(\rho,\theta))(c)$. Since $\rho^{-2} G\in \calA_{\mathrm{c}}^{\calE,\alpha}(\dot{X};\ell)$,  the mapping properties of the Mellin transform show that $\calM_{\rho\to c}(\rho^{-2} G(\rho,\theta))(c)$ is meromorphic in $\{z\in \bbC: \Re z<\alpha\}$ with poles given by $\calE$. Consequently, $\calR(c) \calM_{\rho\to c}(\rho^{-2} G(\rho,\theta))$ is meromorphic on the same domain, with poles given by $\calE \uplus \calC$, where $\calC$ is the pre-index set consisting of all of the $(-b_l,0)$ and $(c_l,0)$ for all $l\in \bbN$. 
		
		However, because $\calM_{\rho\to c}(\rho^{-2} G(\rho,\theta))(c)$ is orthogonal to $\calY_0,\ldots,\calY_{\ell-1}$ in $L^2(\partial X_\theta)$, for each $c$, the poles at $-b_{\ell-1},\ldots,-b_0$ and $c_0,\ldots,c_{\ell-1}$ are absent. Thus, in the strip $\{z\in \bbC: \Re z \in (0,\alpha)\}$, the poles are given by $\calE\uplus \calC_\ell$.  
		\item Secondly, consider $\calR(c) \calM_{\rho\to c}(\rho^{-2} g_l (\rho,\theta))(c)$. 
		Arguing analogously, this is meromorphic in $\{z\in \bbC: \Re z \in (0,\alpha_l)\}$ with poles given by $\calE_l \uplus \{(c_l,0)\}$, where the point is because $g_l(\rho,-)\in \calY_l$ for each $l$, we can apply $\calR(c)$ to $\calM_{\rho\to c}(\rho^{-2} g_l (\rho,\theta))(c)$ unless $c\in \{-b_l,c_l\}$.
	\end{itemize}
	Using \cref{eq:Mel_inv_ex}, we write $w=W + \sum_{l=0}^{\ell-1} w_l$, where
	\begin{align}
	W(\rho,\theta) &= \calM^{-1}_{c\to \rho} (\calR(c) \calM_{\rho\to c}(\rho^{-2} G(\rho,\theta))(c)) \\
	w_l(\rho,\theta)  &= \calM^{-1}_{c\to \rho} (\calR(c) \calM_{\rho\to c}(\rho^{-2} g_l(\rho,\theta))(c))
	\end{align}
	for each $l\in \{0,\ldots,\ell-1\}$. Each of these inverse Mellin transforms is initially taken along a vertical contour slightly right of the imaginary axis. The mapping properties of the Mellin transform then say that $
	W \in \calA^{\calC_\ell\uplus \calE , \alpha-}(\dot{X};\ell)$, $w_l \in \calA^{(c_l,0)\uplus \calE_l,\alpha_l-}([0,1)_\rho) \calY_l$. 
\end{proof}

\subsection{More general operators}
\label{subsec:0main}

Now suppose that $P\in \operatorname{Diff}_{\mathrm{b}}^2(X)$ has the form $P = \triangle_{g_0}   +  Q  + E$, for some $Q,E\in \operatorname{Diff}_{\mathrm{b}}^2(X)$, which, on $\dot{X} = \smash{[0,1)_{1/r}} \times \smash{\partial X_\theta}$, have the form 
\begin{itemize}
	\item $Q \in r^{-3-\beth} \operatorname{Diff}_{\mathrm{b}}^2[0,1)_{1/r} + r^{-3-\aleph}S  \operatorname{Diff}_{\mathrm{b}}^2[0,1)_{1/r}$, so $Q$ involves only partial derivatives in the radial direction $r$,  
	and 
	\item $E\in r^{-3-\beth_0} \operatorname{Diff}_{\mathrm{b}}^2(X) + S^{-3-\aleph_0 } \operatorname{Diff}_{\mathrm{b}}^2(X)$
\end{itemize}
for some $\aleph,\aleph_0\in \bbR^{\geq 0} \cup \{\infty\}$ and nonnegative integers $\beth,\beth_0\in \bbN \cup \{\infty\}$. 
Without loss of generality, we may assume that $\aleph,\beth_0 \in [\beth,\aleph_0]$. Thus, we keep track of four orders, the order $\beth$ below subleading at which $P$ differs from $\triangle_{g_0}$, the order $\beth_0$ below subleading where $\theta$-dependent terms appear, and the analogues $\aleph,\aleph_0$ regarding merely symbolic coefficients. 

In the following proposition, we use $\dot{\calI}$ to denote the index set generated by a pre-index set $\calI$. 
\begin{proposition}
	Suppose that $u\in \calA^{0+}(X)$ satisfies $Pu=f$ for some $f\in \calA^{2+\calE,2+\alpha}(X;\ell) + \sum_{l=0}^{\ell-1} \calA^{2+\calE_l,2+\alpha_l}_{\mathrm{c}}[0,1) \calY_l$. 
	Suppose that $\calI,\calI_0,\ldots,\calI_{\ell-1}$ are pre-index sets such that 
	\begin{itemize}
		\item $\calC_\ell \uplus (\calE \cup (1+\beth_0 + \dot{\calI}_\cup) \cup (1+\beth+\dot{\calI})) \subseteq \calI$,  where $\calI_\cup = \calI\cup \calI_0\cup \cdots \cup \calI_{\ell-1}$, and 
		\item $\{(c_l,0)\}\uplus (\calE_l\cup (1+\beth_0+\dot{\calI}_\cup) \cup (1+\beth+\dot{\calI}_l))$ for every $l=0,\ldots,\ell-1$.
	\end{itemize}
	Then, 
	\begin{equation} 
	u\in \calA^{\calI,\gamma-}(X;\ell) + \sum_{l=0}^{\ell-1} \calA^{\calI_l,\gamma_l-}_{\mathrm{c}}[0,1) \calY_l
	\end{equation} 
	for any $\gamma \leq \min\{\alpha,1+\aleph_0+\Pi \calI_\cup,1+\aleph+\Pi\calI\}$ and $\gamma_l \leq \{\alpha_l, 1+\aleph_0 + \Pi \calI_{\cup -}, 1+\aleph + \Pi \calI_l\}$ satisfying $\max\{\gamma,\gamma_0,\ldots,\gamma_{\ell-1}\} \leq 1+\beth_0 + \gamma_\wedge$ for $\gamma_\wedge = \min\{\gamma,\gamma_0,\ldots,\gamma_{\ell-1}\}$. 
	\label{prop:zf_computation}
\end{proposition}
Note that such $\gamma,\gamma_l$ exist, as one solution is 
\begin{equation}
\gamma,\gamma_l = \min\{\alpha,\alpha_l,1+\aleph_0+\Pi \calI_\cup,1+\aleph+\Pi\calI,1+\aleph + \Pi \calI_l : l=0,\ldots,\ell-1\}.
\end{equation}
\begin{proof}
	Suppose that 
	\begin{equation} 
	u\in \calA^{\calI,\gamma\wedge \beta-}(X;\ell) + \sum_{l=0}^{\ell-1} \calA^{\calI_l,\gamma_l\wedge \beta-}_{\mathrm{c}}([0,1)) \calY_l, 
	\label{eq:ind}
	\end{equation} 
	where $\gamma\wedge \beta = \min\{\gamma,\beta\}$. 
	Because $u\in \calA^{0+}(X)$, we know this holds for all nonpositive $\beta$, as well as some positive $\beta$. Our goal is to show that it holds for all $\beta$.
	
	Given \cref{eq:ind}, $u\in \calA^{\calI_\cup, \gamma_\wedge\wedge  \beta-}(X)$. Since we have the mapping properties 
	\begin{align}
	\operatorname{Diff}_{\mathrm{b}}^2(X) &: \calA^{\calI_\cup, \gamma_\wedge \wedge \beta-}(X) \to \calA^{\dot{\calI}_\cup, \gamma_\wedge \wedge \beta-}(X) \\ 
	S \operatorname{Diff}_{\mathrm{b}}^2(X) &: \calA^{\calI_\cup, \gamma_\wedge \wedge \beta-}(X) \to \calA^{\min\{\gamma_\wedge,\beta,\Pi \calI_\cup\}-}(X),
	\end{align}
	where $\Pi \calI = \{\Re j:(j,k)\in \calI\}$, 
	we have 
	\begin{multline}
	E u \in \calA^{3+\beth_0+\dot{\calI}_\cup, 3+\beth_0+\gamma_\wedge \wedge \beta-}(X) + \calA^{3+\aleph_0+\min\{\gamma_\wedge,\beta,\Pi \calI_\cup\}- }(X)  \\ 
	= \calA^{3+\beth_0 + \dot{\calI}_\cup,3+ \min\{\beth_0+\gamma_\wedge\wedge \beta,\aleph_0+\Pi \calI_\cup \}-}(X),
	\end{multline}
	using that $\beth_0 \leq \aleph_0$.

	Similarly, we can write $Qu = q + \sum_{l=0}^{\ell-1} q_l$, where 
	\begin{align}
	\begin{split} 
	q&\in \calA^{3+\beth+\dot{\calI},3+\beth+\gamma\wedge \beta-}(X;\ell) + \calA^{3+\aleph+\min\{\gamma_\wedge \beta,\Pi \calI\}-}(X;\ell) \\ &= \calA^{3+\beth+\dot{\calI},3+\min\{\beth+\gamma\wedge \beta, \aleph + \Pi \calI\}-}(X;\ell), 
	\end{split} \\
	\begin{split} 
	q_l &\in \calA_{\mathrm{c}}^{3+\beth+\dot{\calI}_l,3+\beth+\gamma_l\wedge \beta-} [0,1)\calY_l + \calA_{\mathrm{c}}^{3+\aleph+\min\{\gamma_l\wedge \beta,\Pi\calI_l\}-} ([0,1))\calY_l \\
	&= \calA^{3+\beth+\dot{\calI}_l,3+\min\{\beth+\gamma_l\wedge \beta, \aleph + \Pi \calI_l\}-}_{\mathrm{c}}([0,1))\calY_l
	\end{split} 
	\end{align}
	for  each $l\in \{0,\ldots,\ell-1\}$.

	We can rewrite the PDE $Pu=f$ in the form $\triangle_{g_0} u = g$ for $g = f - Qu - Eu$. The computations above show that
	\begin{multline}
	g \in \calA^{(2+\calE)\cup (3+\beth_0+\dot{\calI}_\cup) \cup (3+\beth+\dot{\calI}), \min\{2+\alpha, 3+\beth_0 + \gamma_\wedge\wedge \beta-,3+\aleph_0+\Pi \calI_\cup-,3+\beth+\gamma \wedge \beta-,3+\aleph+\Pi\calI-   \} } (X;\ell) \\ 
	+ \textstyle \sum_{l=0}^{\ell-1} \calA^{(2+\calE_l)\cup (3+\beth_0+\dot{\calI}_\cup) \cup (3+\beth+\dot{\calI}_l), \min\{2+\alpha_l, 3+\beth_0 + \gamma_\wedge \wedge \beta-,3+\aleph_0+\Pi \calI_\cup-,3+ \beth+\gamma_l\wedge\beta-,3+\aleph_l+\Pi\calI_l-   \} }_{\mathrm{c}}([0,1))\calY_l. 
	\end{multline}
	Applying \Cref{lem:Lap_basic}, we can conclude that 
	\begin{equation}
	u \in \calA^{\calC_\ell\uplus(\calE\cup (1+\beth_0+\dot{\calI}_\cup) \cup (1+\beth+\dot{\calI})), \gamma^+- } (X;\ell) 
	+ \textstyle \sum_{l=0}^{\ell-1} \calA^{ \{(c_l,0)\} \uplus (\calE_l\cup (1+\beth_0+\dot{\calI}_\cup) \cup (1+\beth+\dot{\calI}_l)), \gamma^+_l-}_{\mathrm{c}}([0,1))\calY_l, 
	\label{eq:comp_ind_a}
	\end{equation}	
	where 
	\begin{align}
	\begin{split} 
	\gamma^+  &= \min\{\alpha, 1+\beth_0+\gamma_\wedge \wedge \beta, 1+\beth + \gamma\wedge \beta,1+\aleph_0+\Pi \calI_\cup,1+\aleph+\Pi\calI\} \\ & \geq \min\{\gamma,1+\beth+\beta,1+\beth_0+\gamma_\wedge \} 
	\end{split} \\
	\begin{split} 
	\gamma^+_l  &=  \{\alpha_l,1+\beth_0+\gamma_\wedge \wedge \beta, 1+\aleph_0 + \Pi \calI_{\cup -}, 1+\beth + \gamma_l\wedge \beta, 1+\aleph_l + \Pi \calI_l\} \\ &\geq \min\{\gamma_l, 1+\beth + \beta ,1+\beth_0 + \gamma_\wedge\}.
	\end{split} 
	\end{align}
	Since we are assuming that $\gamma,\gamma_l \leq 1+\beth_0 + \gamma_\wedge$, we have $\gamma^+ \geq \gamma\wedge \beta^+$ and $\gamma^+_l \geq \gamma_l \wedge \beta^+$ for $\beta^+ = 1+\beth + \beta$. 
	Using our assumptions regarding $\calI,\calI_l$, the pre-index sets in \cref{eq:comp_ind_a} are subsets of the respective of $\calI,\calI_0,\ldots,\calI_{\ell-1}$. 
	So, \cref{eq:comp_ind_a} gives 
	\begin{equation}
	u\in  \calA^{\calI,\gamma\wedge \beta^+-}(X;\ell) + \sum_{l=0}^{\ell-1} \calA^{\calI_l,\gamma_l \wedge \beta^+-}_{\mathrm{c}}([0,1)) \calY_l.
	\end{equation}
	Applying this argument inductively, we have $u\in  \calA^{\calI,\gamma\wedge \beta^{+\cdots +}-}(X;\ell) + \sum_{l=0}^{\ell-1} \calA^{\calI_l, \gamma_l \wedge \beta^{+\cdots +}-}_{\mathrm{c}}([0,1)) \calY_l$, where $\beta^{+\cdots+} = ((\beta^+)^+)\cdots$. Done enough times, $\beta^{+\cdots+} \geq \max\{\gamma,\gamma_0,\ldots,\gamma_{\ell-1}\}$, so we conclude the result.
\end{proof}

Given $\calE,\calE_l$, it is straightforward to construct minimal pre-index sets $\calI,\calI_l$ with the closure properties required in the proposition, as we now demonstrate.

Let
\begin{equation}
\calI^{(0)} = \calE, \qquad \calI_l^{(0)} = \calE_l, 
\end{equation}
and recursively define pre-index sets
\begin{align}
\begin{split} 
\calI^{(k+1)} &= \calI^{(k)} \cup( \calC_\ell\uplus (\calE\cup (1+\beth_0+\dot{\calI}_\cup^{(k)})\cup (1+\beth+\dot{\calI}^{(k)}))) , \\
\calI_l^{(k+1)} &= \calI_l^{(k)} \cup( \{(c_l,0)\}\uplus( \calE_l \cup (1+\beth_0+\dot{\calI}_\cup^{(k)} ) \cup (1+\beth + \dot{\calI}_l^{(k)}) ))
\label{eq:indexsets}
\end{split} 
\end{align}
for all $k\in \bbN$, where $\calI_\cup^{(k)}= \calI^{(k)}\cup \calI^{(k)}_0\cup\cdots \calI_{\ell-1}^{(k)}$. Then, one takes 
\begin{equation}
\calI = \cup_{k\in \bbN} \calI^{(k)}, \qquad \calI_l = \cup_{k\in \bbN} \calI^{(k)}_l. \label{eq:index_set_defs}
\end{equation}

In order to show that these are pre-index sets (which is not completely obvious), what one wants to show is that, for any $\alpha\in \bbR$, the sets $\{(\gamma,k)\in \calI : \Re \gamma < \alpha\}$, $\{(\gamma,k)\in \calI_l : \Re \gamma < \alpha\}$ are finite. 
This follows from the following claim:  letting 
\begin{equation}
\calI_{<\alpha}^{(k)} = \{(\gamma,k)\in \calI^{(k)} : \Re \gamma < \alpha\} ,\qquad \calI_{l,<\alpha}^{(k)} = \{(\gamma,k)\in \calI_l^{(k)} : \Re \gamma < \alpha\}, 
\label{eq:truncated_ind}
\end{equation}
it is the case that, for each $\alpha$, there exists some $k(\alpha)$ such that these sets are constant in $k$ for $k\geq k(\alpha)$. 
In other words, for $k$ large the definitions \cref{eq:indexsets} are only adding to $\calI_\bullet$ elements $(\gamma,k)$ with $\Re \gamma$ large.

Let $S\subset \bbR$ be the set of $\alpha$ for which this holds. If $\alpha$ is sufficiently negative, then the sets in \cref{eq:truncated_ind} are just empty. So, $S$ is nonempty. 
Indeed,  
\begin{align}
\begin{split} 
\calI^{(k+1)}_{<\alpha} &= \calI^{(k)}_{ < \alpha} \cup ((\calC_\ell)_{<\alpha}\uplus (\calE_{<\alpha}\cup (1+\beth_0+\dot{\calI}_\cup^{(k)})_{<\alpha}\cup (1+\beth+\dot{\calI}^{(k)})_{<\alpha})  \\
&= \calI^{(k)}_{<\alpha} \cup ((\calC_\ell)_{<\alpha}\uplus (\calE_{<\alpha}\cup (1+\beth_0+(\dot{\calI}_\cup^{(k)})_{<\alpha-1-\beth_0})\cup (1+\beth+\dot{\calI}^{(k)}_{<\alpha-1-\beth} )),  
\end{split} 
\intertext{and, likewise,}
\calI^{(k+1)}_{l,<\alpha} &=  \calI^{(k)}_{l,<\alpha} \cup (\{(c_l,0)\}_{<\alpha}\uplus (\calE_{l,<\alpha}\cup (1+\beth_0+(\dot{\calI}_\cup^{(k)})_{<\alpha-1-\beth_0})\cup (1+\beth+\dot{\calI}^{(k)}_{l,<\alpha-1-\beth} )).
\end{align}
So, if $\alpha <d-2$, in which case $(\calC_\ell)_{<\alpha} = \varnothing$ and $(c_l,0)_{<\alpha} = \varnothing$, and if $\alpha< \min\{\Pi \calE,\Pi \calE_l:l=0,\ldots,\ell-1\} = \min \Pi \calE_\cup$, in which case $\calE_{<\alpha},\calE_{l,<\alpha} = \varnothing$, then $\calI_{<\alpha},\calI_{l,<\alpha}=\varnothing$. 
So, $(-\infty,\min\{d-2,\Pi \calE_\cup \}]  \subseteq S$,
and we can take $k(\alpha) = 0$ for $\alpha \leq \min\{d-2,\Pi \calE_\cup \}$.

More generally, if $\alpha-1-\min\{\beth,\beth_0\} = \alpha-1-\beth \in S$, and if $k \geq k(\alpha-1-\beth)+1$, then we can conclude that 
\begin{equation}
\calI_{<\alpha}^{(k+1)} = \calI^{(k)}_{<\alpha} ,\qquad \calI^{(k+1)}_{l,<\alpha} = \calI^{(k)}_{l,<\alpha}, 
\end{equation}
for all $l\in \{0,\ldots,\ell-1\}$. Thus, $\alpha\in S$, and we can conclude that we may take $k(\alpha) = k(\alpha-1-\beth)+1$. 
Proceeding inductively, we conclude that $S=\bbR$, and we may take 
\begin{equation}
k(\alpha) = \Big\lceil \frac{\alpha - \min\{d-2,\Pi \calE_\cup\} }{1+\beth}  \Big\rceil .
\end{equation}

So, the sets $\calI ,\calI_l$ defined above are well-defined pre-index sets, and they satisfy the hypotheses of the proposition. 
Thus, we get:
\begin{propositionp}
	Suppose that $u\in \calA^{0+}(X)$ satisfies $Pu\in \calA^{2+\calE,2+\alpha}(X;\ell) + \sum_{l=0}^{\ell-1} \calA^{2+\calE_l,2+\alpha_l}_{\mathrm{c}}[0,1) \calY_l$.
	Then, 
	\begin{equation} 
	u\in \calA^{\calI,\gamma-}(X;\ell) + \sum_{l=0}^{\ell-1} \calA^{\calI_l,\gamma-}_{\mathrm{c}}[0,1) \calY_l,
	\end{equation}
	where $\gamma,\gamma_l$ are above and $\calI,\calI_l$ are the previously defined pre-index sets.
\end{propositionp}

\begin{example}
	Consider the case where $\beth_0 = \beth$ and $\calE_0,\ldots,\calE_{\ell-1} = \calE$. In other words, the first subleading terms of $P$ are not assumed to preserve the conic symmetry.

	Then, \cref{eq:indexsets} yields $\calI_\cup^{(k+1)} = \calI_\cup^{(k)} \cup (\calC_0 \uplus \calE) \cup (\calC_0 \uplus(1+\beth+\dot{\calI}_\cup^{(k)}) )$. This single recurrence relation can be solved to yield
	\begin{equation}
	\calI_\cup = (\calC_0\uplus \calE) \cup (\calC_0 \uplus  (1+\beth+\calC_0\dot{\uplus} \calE)  ) \cup (\calC_0 \dot{\uplus} (1+\beth+(\calC_0 \uplus  (1+\beth+\calC_0\dot{\uplus} \calE)  ) ) \cup \cdots .
	\label{eq:misc_238}
	\end{equation}
	(Here, $\dot{\uplus}$ means apply $\uplus$ and then take the index set generated by the result.)
	\Cref{eq:indexsets} also yields $\calI^{(k+1)} = \calI^{(k)} \cup (\calC_\ell \uplus \calE) \cup (\calC_\ell \uplus(1+\beth + \dot{\calI}_\cup^{(k)}))$, which in turn yields 
	\begin{equation} 
		\calI = (\calC_\ell \uplus \calE) \cup ( \calC_\ell\uplus (1+\beth+\dot{\calI}_\cup)).
	\end{equation}
	Similarly, $\calI_l = (\{(c_l,0)\} \uplus \calE) \cup (\{ (c_l,0) \}\uplus (1+\beth+\dot{\calI}_\cup)) $.
	So, this case is understood explicitly. 
	\label{ex:spherical_breaking}
\end{example}
On the other extreme:
\begin{example}
	Consider now the case where $\beth_0=\infty$, which means that all symmetry-breaking terms in $P$ are Schwartz. 
	
	Then, the $1+\beth_0 + \dot{\calI}_\cup^{(k)}$ terms in \cref{eq:indexsets} are always $\varnothing$, so the recursion relations decouple:
	\begin{align}
	\calI^{(k+1)} &= \calI^{(k)} \cup (\calC_\ell \uplus \calE ) \cup (\calC_\ell \uplus (1+\beth + \dot{\calI}^{(k)}) ) \\
	\calI^{(k+1)}_l &= \calI^{(k)}_l \cup (\{(c_l,0)\}  \uplus \calE_l ) \cup (\{(c_l,0)\} \uplus (1+\beth + \dot{\calI}^{(k)}_l) ) .
	\end{align}
	These decoupled recurrence relations are readily solved to yield $\calI = (\calC_\ell \uplus \calE) \cup ( \calC_\ell \uplus (1+\beth + \calC_\ell \dot{\uplus} \calE))  \cup ( \calC_\ell \uplus (1+\beth +  \calC_\ell \dot{\uplus} (1+\beth + \calC_\ell \dot{\uplus} \calE) ))  \cup \cdots$, 
	and similarly for $\calI_l$.  
	\label{ex:sym}
\end{example}

\begin{example}
	 If, in addition, $\beth=\infty$ and $\calE_\bullet = \varnothing$  (meaning that $P$ differs from $\triangle_{g_0}$ by Schwartz terms, and that $f$ is Schwartz), then $\calI = \calC_\ell$ and $\calI_l = \{(c_l,0)\}$.
	 So, in this case we recover the multipole expansion from classical electrostatics. Note the advantage to working with pre-index sets: we have proven that s-waves are \emph{exactly} proportional to $1/r$, p-waves are exactly proportional to $1/r^2$, etc., up to Schwartz terms.
	 \label{ex:Coulomb}
\end{example}

\section{Solvability theory of the model problem at $\mathrm{tf}$}
\label{sec:tf}

We make use of the limiting absorption principle for $\hat{N}_{\mathrm{tf}}$ (the operator defined in \cref{eq:Ntf_form}), involving 
the function spaces 
\begin{equation}
\calA^{\beta,\gamma}(\mathrm{tf}) = \rho_{\mathrm{bf}}^\beta \rho_{\mathrm{zf}}^\gamma \calA^{0,0}(\mathrm{tf}), \qquad \calA^{0,0}(\mathrm{tf}) = \{v\in C^\infty(\mathrm{tf}^\circ) : Lv \in L^\infty\text{ for all }L\in \operatorname{Diff}_{\mathrm{b}}(\mathrm{tf}) \},
\end{equation}
where $\operatorname{Diff}_{\mathrm{b}}(\mathrm{tf}) = \bigcup_{m\in \bbN} \operatorname{Diff}_{\mathrm{b}}^{m,0,0}(\mathrm{tf})$ is the algebra of differential operators generated over $C^\infty(\mathrm{tf})$ by $\hat{r}\partial_{\hat{r}}$ and the differential operators in the angular directions. In $\smash{\operatorname{Diff}_{\mathrm{b}}^{m,\ell,s}(\mathrm{tf})}$, the $m$ is the differential order, $\ell$ is the order at $\mathrm{bf}$, and $s$ is the order at $\mathrm{zf}$. 

To begin our discussion, we recall:
\begin{propositionp}[{\cite[Prop.\ 5.4]{VasyLagrangian}}, {\cite[Thm.\ 2.22]{HintzPrice}}]
	If $\beta<(d-1)/2$ and $\gamma \in (2-d,0)$, then $\hat{N}_{\mathrm{tf}} : \calA^{\beta,\gamma}(\mathrm{tf}) \to \calA^{\beta+1,\gamma-2}(\mathrm{tf})$ is an isomorphism. 
	Moreover, letting 
	\begin{align*}
	b_j &= 2^{-1}(2-d + ((d-2)^2+4\lambda_j)^{1/2}) \\ 
	c_j &= 2^{-1}(d-2 + ((d-2)^2+4\lambda_j)^{1/2}),
	\end{align*}
	as in the previous section, 
	then $\hat{N}_{\mathrm{tf}} : \calA^{\beta,\gamma}(\mathrm{tf};\ell) \to \calA^{\beta+1,\gamma-2}(\mathrm{tf};\ell)$ is an isomorphism under the weakened hypothesis that $\gamma \in (-c_\ell,b_\ell)$. 
	\label{prop:Vasy_tf}
\end{propositionp}

Here, $\calA^{\beta,\gamma}(\mathrm{tf};\ell)$ are the elements of $\calA^{\beta,\gamma}$ orthogonal to $\calY_0,\dots,\calY_{\ell-1}$.

We will not repeat the proof of this result, but it is not difficult to understand the restrictions on $\beta,\gamma$ therein:

\begin{proof}[Proof sketch]
	Since $\mathrm{tf}$ is an exact cone, we can separate variables: for $u(\hat{r},\theta) = v(\hat{r})Y_j(\theta)$ 
	and $f(\hat{r},\theta) = g(\hat{r})Y_j(\theta)$, $Y_j\in \calY_j$, $u$ solves $\hat{N}_{\mathrm{tf}} u=f$ if and only if, letting 
	\begin{equation}
	L_j=- \frac{\mathrm{d}^2}{\mathrm{d} \hat{r}^2} - \Big(\frac{d-1}{\hat{r}}+2i\Big)\frac{\mathrm{d}}{\mathrm{d} \hat{r}} + \frac{\lambda_j }{\hat{r}^2} - \frac{i(d-1)}{\hat{r}},
	\end{equation}
	we have $L_j v = g$. Here, as in previous sections, $\lambda_j$ is the $j$th eigenvalue of the boundary Laplacian (counted with or without multiplicity, depending on our conventions, so that $\triangle_{\partial X} Y_j = \lambda_j Y_j$). Note that 
	\begin{equation}
	L_j : \calA^{\beta,\gamma}([0,\infty]_{\hat{r}}) \to \calA^{\beta+1,\gamma-2}([0,\infty]_{\hat{r}}).
	\end{equation}
	We will assume, for the sake of this discussion, that the only obstructions to the invertibility of $\hat{N}_{\mathrm{tf}}: \calA^{\beta,\gamma}(\mathrm{tf}) \to \calA^{\beta+1,\gamma-2}(\mathrm{tf})$ can be seen already as an obstruction to the invertibility of $L_j$ for some $j\in \bbN$. This is in fact true, as can be proven using the Mellin transform. 
	The ordinary differential operator $L_j$ is regular singular at $\hat{r}=0$ and irregular singular at $\hat{r}=\infty$. It has, up to normalization, a unique element of its kernel which is non-oscillatory at infinity and, up to normalization, a unique element which is recessive (which, we will see, means bounded) at $\hat{r}=0$. Indeed, $L_j$ is a conjugated form of the differential operator appearing in Bessel's ODE, which yields
	\begin{equation}
	\ker L_j = e^{-i\hat{r}} \hat{r}^{-(d-2)/2}\operatorname{span}_\bbC\big\{ J_{\nu}(\hat{r}), Y_\nu(\hat{r}) \big\},
	\end{equation}
	where $\nu = 2^{-1} ((d-2)^2 + 4 \lambda_j)^{1/2}$ and $J_\nu,Y_\nu$ are the usual Bessel functions. We will be able to read the restrictions on $\beta,\gamma$ imposed in \Cref{prop:Vasy_tf} off of the $\hat{r}\to 0$ and $\hat{r}\to\infty$ asymptotics of the elements of $\ker L_j$. 
	
	The element of this kernel which is non-oscillatory at infinity is $e^{-i \hat{r}} \hat{r}^{-(d-2)/2} H_\nu^+(\hat{r})$, where $H_\nu^+(\hat{r})$ is the Hankel function with $e^{i\hat{r}}$-type asymptotics. 
	Since conormality excludes oscillatory asymptotics, the only element of $\ker L_j$ that can lie in $\calA^{\beta,\gamma}(\mathrm{tf})$ is the non-oscillatory solution $e^{-i \hat{r}} \hat{r}^{-(d-2)/2} H_\nu^+(\hat{r})$ and its scalar multiples, no matter what $\beta,\gamma$ are. The small-$\hat{r}$ asymptotics of the Hankel functions show that 
	\begin{equation}
	e^{-i \hat{r}} \hat{r}^{-(d-2)/2} H_\nu^+(\hat{r})  \propto ( 1+o(1)) \hat{r}^{-c_j}
	\end{equation}
	as $\hat{r}\to 0^+$, where $c_j = (d-2)/2+\nu$ is as in the previous subsection. Thus, a sufficient condition on $\beta,\gamma$ for 
	\begin{equation}
		\ker \hat{N}_{\mathrm{tf}} \cap \calA^{\beta,\gamma}(\mathrm{tf})=\{0\}
	\end{equation}
	is that $\gamma>-c_j$. For this to hold for all $j$ means $\gamma>-c_0 = 2-d$. 
	
	Solving $L_j v=g$ for $v$ non-oscillatory, even if $g\in C_{\mathrm{c}}^\infty\smash{(\bbR^+_{\hat{r}})}$ we cannot expect $v$ to decay more rapidly as $\hat{r}\to\infty$ than $\smash{e^{-i \hat{r}} \hat{r}^{-(d-2)/2}} H_\nu^+(\hat{r})$. This is especially clear if $g\in C_{\mathrm{c}}^\infty(\bbR^+_{\hat{r}})$, since then 
	\begin{equation} 
		v(\hat{r}) \propto e^{-i \hat{r}} \hat{r}^{-(d-2)/2} H_\nu^+(\hat{r})
	\end{equation} 
	for $\hat{r}$ sufficiently large.  
	Since 
	\begin{equation}
	e^{-i \hat{r}} \hat{r}^{-(d-2)/2} H_\nu^+(\hat{r}) \propto (1+o(1)) \hat{r}^{-(d-1)/2}
	\end{equation}
	as $\hat{r}\to\infty$, this means that for $\hat{N}_{\mathrm{tf}} : \calA^{\beta,\gamma}(\mathrm{tf}) \to \calA^{\beta+1,\gamma-2}(\mathrm{tf})$ to be surjective it had better be the case that $\beta\leq (d-1)/2$.
	On the other hand, such $v$ will typically have no better decay as $\hat{r}\to 0^+$ than the recessive element of $\ker L_j$, which is $e^{-i \hat{r}} \hat{r}^{-(d-2)/2} J_\nu(\hat{r})$, and this satisfies 
	\begin{equation}
	e^{-i \hat{r}} \hat{r}^{-(d-2)/2} J_\nu(\hat{r}) \propto (1+o(1)) \hat{r}^{\nu - (d-2)/2}
	\end{equation}
	as $\hat{r} \to 0^+$. So, for $\hat{N}_{\mathrm{tf}} : \calA^{\beta,\gamma}(\mathrm{tf}) \to \calA^{\beta+1,\gamma-2}(\mathrm{tf})$ to be surjective it had also better be the case that $\gamma \leq \nu-(d-2)/2=b_j$.  In order for this to hold for all $j$, it just needs to be that $\gamma \leq 0$.

	Conversely, we get surjectivity if both of the conditions $\beta  <(d-1)/2$ and $\gamma < 0$ are met. (Excluding the edge cases is necessary to avoid logarithmic terms when inverting $\hat{N}_{\mathrm{tf}}$.) One method of proving this is an explicit construction of the Greens function for $L_j$ using $H_\nu^+$, $J_\nu$, and their (nonzero) Wronskian. This is just Duhamel's principle.
	
	So, in summary, if $\gamma \in (2-d,0)$ and $\beta<(d-1)/2$, then $\hat{N}_{\mathrm{tf}} : \calA^{\beta,\gamma}(\mathrm{tf}) \to \calA^{\beta+1,\gamma-2}(\mathrm{tf})$ is bijective. 
	If we restrict attention to functions orthogonal to $\calY_0,\dots,\calY_{\ell-1}$, then we can ignore the constraints above coming from $j=0,\dots,\ell-1$, so the interval of allowed $\gamma$ gets larger. 
	The constraint on $\beta$ does not change (since that constraint was the same for all $j$).
\end{proof}

The main proposition of this section is a refinement of the previous result involving the function spaces 
\begin{equation}
\calA^{(\calE,\beta),(\calF,\gamma) }(\mathrm{tf};\ell) + \sum_{j=0}^{\ell-1} \calA^{(\calE_j,\beta), (\calF_j,\gamma_j)} ([0,\infty]_{\hat{r}})  \calY_j ,
\end{equation}
where $\calE,\calE_j$ are index sets, $\calF,\calF_j$ are pre-index sets and $\beta,\gamma,\beta_j,\gamma_j\in \bbR$. This function space consists of all functions $\mathrm{tf}^\circ\to \bbC$ of the form 
\begin{equation}
u = u_{\mathrm{rem}} + \sum_{j=0}^{\ell-1} u_j(\hat{r}) 
\end{equation}
for $u_{\mathrm{rem}} \in \calA^{(\calE,\beta),(\calF,\gamma) }(\mathrm{tf};\ell)$ and $ u_j \in \calA^{(\calE_j,\beta_j),(\calF_j,\gamma_j) }([0,\infty]_{\hat{r}} )\calY_j$. Here, the pre-index sets $\calE,\calE_j$ refer to $\hat{r}\to\infty$ asymptotics and the pre-index sets $\calF,\calF_j$ refer to $\hat{r}\to 0^+$ asymptotics.

 Let $\calB_{\geq \ell}$ denote the smallest pre-index set  containing $(b_j,0)\in \bbC\times \bbN$ for all $j\geq \ell$. 

\begin{proposition}
	Fix $\ell\in \bbN$. 
	Suppose that  $\min \{\Pi \calF, \gamma\}>-c_\ell$ and $\min\{\Pi \calF_j,\gamma_j\}>-c_j$ for each $j\leq \ell-1$.
	If
	\begin{equation} 
	f\in \calA^{(1+\calE,1+\beta),(-2+\calF,-2+\gamma) }(\mathrm{tf};\ell) + \sum_{j=0}^{\ell-1} \calA^{(1+\calE_j,1+\beta_j), (-2+\calF_j,-2+\gamma_j)} ([0,\infty]_{\hat{r}}) \calY_j,
	\end{equation} 
	then $u=\hat{N}_{\mathrm{tf}}^{-1} f$ is well-defined, and $u$ satisfies 
	\begin{equation}
	u\in \calA^{((2^{-1}(d-1)\uplus \calE),\beta-),(\tilde{\calI}[\calF,\ell],\gamma-) }(\mathrm{tf};\ell) + \sum_{j=0}^{\ell-1} \calA^{((2^{-1}(d-1)\uplus \calE_j),\beta_j-), (\tilde{\calI}_j[\calF_j],\gamma_j-)}([0,\infty]_{\hat{r}})  \calY_j,
	\end{equation}
	where $\uplus$ is as in \S\ref{sec:bf} and $\tilde{\calI},\tilde{\calI}_j$ are pre-index sets defined as follows: 
	\begin{itemize}
		\item $\tilde{\calI}_j = \{(b_j,0)\} \uplus \calF_j$, 
		\item and, similarly, $\tilde{\calI}[\calF,\ell]$ is the smallest pre-index set containing $\calF$ and $\{(b_j,0)\}\uplus (1+\tilde{\calI}')$ whenever $\tilde{\calI}'\subseteq \tilde{\calI}$ and $j\geq \ell$. 
	\end{itemize}
	Here, we are abbreviating $(2^{-1}(d-1),0)\uplus \calE= 2^{-1} (d-1)\uplus \calE$.
	\label{prop:Ntf_main}
\end{proposition}
We also write $\tilde{\calI}[\calF,\ell]$ as $\calB_{\geq \ell}[ \calF]$. This can be computed explicitly, but we will not do so. 

\begin{proof}
	That the requirements on $\calF,\calF_j,\gamma,\gamma_j$ suffice for the well-definition of $\hat{N}_{\mathrm{tf}}^{-1} u$ follows from \Cref{prop:Vasy_tf}.

	Since we already know that $u\in C^\infty(\mathrm{tf}^\circ)$, it suffices to check $u$ near $\hat{r}=\infty$ and near $\hat{r}=0$.

	The partial polyhomogeneity near $\hat{r}=\infty$ can be proven using the exact same argument as that used to prove \Cref{prop:bf_2}. Here are the details: the least decaying term in $\hat{N}_{\mathrm{tf}}$ is 
	\begin{equation}
	N_{\mathrm{tf}\cap\mathrm{bf}} = -2i \frac{\partial}{\partial \hat{r}} - \frac{i(d-1)}{\hat{r}}. 
	\end{equation}
	The sense in which this is least decaying is that $\hat{N}_{\mathrm{tf}}-N_{\mathrm{tf}\cap\mathrm{bf}} \in \hat{r}^{-2}\operatorname{Diff}_{\mathrm{b}}(\mathrm{tf}\backslash \mathrm{zf})$, whereas $N_{\mathrm{tf}\cap\mathrm{bf}} \in \hat{r}^{-1}\operatorname{Diff}_{\mathrm{b}}(\mathrm{tf}\backslash \mathrm{zf})$. So, we rearrange the definition $\hat{N}_{\mathrm{tf}} u=f$ of $u$ to get 
	\begin{equation} 
		N_{\mathrm{tf}\cap\mathrm{zf} } u = f - (\hat{N}_{\mathrm{tf}} - N_{\mathrm{tf}\cap \mathrm{bf}}) u.
	\end{equation}
	A one-sided inverse of $N_{\mathrm{tf}\cap\mathrm{bf}}$ can be constructed just by integrating: 
	\begin{equation}
	N_{\mathrm{tf} \cap \mathrm{bf}}^{-1} = \Big[ -2i \frac{\partial}{\partial \hat{r}} - \frac{i(d-1)}{\hat{r}} \Big]^{-1} : f(\hat{r}) \mapsto \frac{i}{2\hat{r}^{(d-1)/2}}\int_1^{\hat{r}}f(x)x^{(d-1)/2} \dd x .
	\end{equation} 
	Applying this to both sides of $N_{\mathrm{tf}\cap\mathrm{zf} } u = f - (\hat{N}_{\mathrm{tf}} - N_{\mathrm{tf}\cap \mathrm{bf}}) u$, we get
	\begin{equation}
	u = \frac{1}{\hat{r}^{(d-1)/2}} \Big( C(\theta) +  \frac{i}{2} \int_1^{\hat{r}} (f(x) - (\hat{N}_{\mathrm{tf}} - N_{\mathrm{tf}\cap \mathrm{bf}}) u(x) ) x^{(d-1)/2} \dd x \Big),
	\label{eq:misc_544}
	\end{equation}
	where $C(\theta)$ is an undetermined function of $\theta$ (that has to do with $\hat{r}\to 0^+$ behavior). 
	
	We abbreviate $\calJ[\calE]=2^{-1}(d-1)\uplus \calE$.
	Now, let $S$ denote the set of $B\in \bbR$ such that $u\in \calA^{\calJ[\calE],\beta\wedge B}(\mathrm{tf}\backslash \mathrm{zf};\ell) + \sum_{j=0}^{\ell-1} \calA^{\calJ[\calE_j],\beta_j\wedge B } ((0,\infty]_{\hat{r}}) \calY_j$. We already know, from \Cref{prop:Vasy_tf}, that $B\in S$ if it is sufficiently negative, so $S$ is nonempty. Note that 
	\begin{equation}
	(\hat{N}_{\mathrm{tf}} - N_{\mathrm{tf}\cap \mathrm{bf}}) u(\hat{r}) \in \calA^{2+\calJ[\calE],2+\beta\wedge B}(\mathrm{tf}\backslash \mathrm{zf};\ell) + \sum_{j=0}^{\ell-1} \calA^{2+\calJ[\calE_j],2+\beta_j\wedge B } ((0,\infty]_{\hat{r}}) \calY_j, 
	\end{equation}
	so the factor in the integrand $f_1 = f - (\hat{N}_{\mathrm{tf}} - N_{\mathrm{tf}\cap \mathrm{bf}}) u$ in \cref{eq:misc_544} satisfies 
	\begin{equation}
	f_1 \in \calA^{(1+\calE)\cup (2+\calJ[\calE]),(1+\beta)\wedge (2+ B) }(\mathrm{tf}\backslash \mathrm{zf};\ell) + \sum_{j=0}^{\ell-1} \calA^{(1+\calE_j) \cup (2+\calJ[\calE_j]) ,(1+\beta_j)\wedge(2+B) } ((0,\infty]_{\hat{r}}) \calY_j.
	\end{equation}
	Consequently, \cref{eq:misc_544} yields 
	\begin{multline}
	u \in \calA^{(2^{-1}(d-1),0) \uplus (\calE\cup (1+\calJ[\calE])),\beta\wedge (1+ B)- }(\mathrm{tf}\backslash \mathrm{zf};\ell) \\ + \sum_{j=0}^{\ell-1} \calA^{(2^{-1}(d-1),0) \uplus (\calE_j\cup (1+\calJ[\calE_j])),\beta_j\wedge (1+ B)- } ((0,\infty]_{\hat{r}}) \calY_j.
	\label{eq:misc_5h5}
	\end{multline}
	By construction, $(2^{-1}(d-1),0) \uplus (\calE\cup (1+\calJ[\calE]))\subseteq \calJ[\calE]$, and likewise for the other index sets. Consequently, \cref{eq:misc_5h5} says that $B+1-\varepsilon \in S$ for any $\varepsilon>0$. So, proceeding inductively, we can conclude that $S=\bbR$. 
	
	Consider now the $\hat{r}\to 0^+$ behavior. 
	Let $S$ denote the set of $\Gamma\in \bbR$ such that 
	\begin{equation}
	u \in \calA^{(\tilde{\calI},\gamma\wedge \Gamma-)}(\mathrm{tf}\backslash \mathrm{bf};\ell ) + \sum_{j=0}^{\ell-1} \calA^{\tilde{\calI}_j, \gamma_j\wedge \Gamma-}([0,\infty)_{\hat{r}} )\calY_j.
	\label{eq:misc_222_640}
	\end{equation}
	We want to prove that $S=\bbR$. We already know from \Cref{prop:Vasy_tf} that $S$ is nonempty. 
	Given that \cref{eq:misc_222_640} holds, we show that it holds also for a larger $\Gamma$. To do this, we use 
	\begin{equation} 
	N_{\mathrm{zf}\cap\mathrm{tf}} = - \frac{\partial^2}{\partial \hat{r}^2 } - \frac{d-1}{\hat{r}}\frac{\partial}{\partial \hat{r}}+ \frac{1}{\hat{r}^2}\triangle_{\partial X}.
	\end{equation} 
	Note that $\hat{N}_{\mathrm{tf}} - N_{\mathrm{zf}\cap\mathrm{tf}} \in \operatorname{Diff}_{\mathrm{b}}^{1,-1,1}(\mathrm{tf})$, so $N_{\mathrm{zf}\cap \mathrm{tf}} u = \hat{N}_{\mathrm{tf}}u - (\hat{N}_{\mathrm{tf}}- N_{\mathrm{zf}\cap \mathrm{tf}}) u = f  - ((\hat{N}_{\mathrm{tf}}- N_{\mathrm{zf}\cap \mathrm{tf}} ) u$
	lies in 
	\begin{multline} 
	\calA^{(-2+\calF) \cup (-1+\tilde{\calI}), \min\{-2+\gamma,-1+\gamma\wedge \Gamma\}- }  (\mathrm{tf}\backslash \mathrm{bf} )  \\ + \sum_{j=0}^{\ell-1}  \calA^{ (-2+\calF_j) \cup (-1+\tilde{\calI}_j)  , \min\{-2+\gamma_j,-1+\gamma_j\wedge \Gamma\}-  } ([0,\infty)_{\hat{r}} )\calY_j.
	\end{multline}
	The indicial roots of $N_{\mathrm{zf}\cap\mathrm{tf}}$ are
	readily computed to be $b_j,-c_j$. Indeed, in the previous section we saw that $c_j,-b_j$ were the indicial roots of the normal operator of $N_{\mathrm{zf}}(P)$ at $r\to\infty$, which is $\mathrm{tf}$. That is, $N_{\mathrm{tf}}( N_{\mathrm{zf}}(P) )$ has indicial roots $c_j,-b_j$. But $N_{\mathrm{tf}}( N_{\mathrm{zf}}(P) ) \propto N_{\mathrm{tf}\cap\mathrm{zf}}$, where the proportionality involves a factor of a boundary-defining-function. However, before we were computing indicial roots as $r\to\infty$, whereas now we are computing indicial roots as $\hat{r}\to 0^+$, which accounts for why they are now negative. 
	
	So, 
	inverting $N_{\mathrm{zf}\cap \mathrm{tf}}$ and citing the analogue of \Cref{lem:Lap_basic} (which just has $c_j$ switched with $b_j$, but is otherwise analogous), we conclude 
	\begin{equation}
	u \in \calA^{\tilde{\calI}[\calF \cup (1+\tilde{\calI}[\calF,\ell]),\ell],\min\{\gamma,1+\Gamma\}- }(\mathrm{tf}\backslash \mathrm{bf})  + \sum_{j=0}^{\ell-1} \calA^{(b_j,0)\uplus \calF_j, \min\{\gamma_j,1+\Gamma\}-}([0,\infty )_{\hat{r}} ) \calY_j.
	\end{equation}

	Given the definition of $\tilde{\calI}$, we have $\tilde{\calI}[\calF \cup (1+\tilde{\calI}[\calF,\ell]),\ell] \subseteq \tilde{\calI}[\calF,\ell]$. So, we conclude that $1+\Gamma \in S$. Induction then establishes that $S=\bbR$. 
\end{proof}

In the previous proposition, we required different things of the different $\calF_\bullet,\gamma_\bullet$, whereas in \S\ref{sec:0} we required $u\in \calA^{0+}(X)$, not discriminating between different harmonics.
This is because, in proving the previous proposition, we made use of separation of variables. The model problem at $\mathrm{tf}$ permits separation of variables, whereas that at $\mathrm{zf}$ does not. If the $\beth_0$ in \S\ref{sec:0} is big, then we have approximate separability, so we can weaken the $u\in \calA^{0+}(X)$ requirement on finitely many harmonics. But this does not seem worth doing.

\section{Some reminders about the Mellin transform}
\label{sec:Mellin}

Our conventions on the Mellin transform are as follows: for $\calX$ a Banach space and $f\in C^\infty([0,1);\calX)$, let 
\begin{equation}
	\calM f (c) = \int_0^1 f(\rho) \frac{\dd \rho}{\rho^{c+1}}
	\label{eq:Mellin}
\end{equation}
for $c\in \bbC$ such that the Bochner integral above is absolutely convergent.

If $f\in \calA_{\mathrm{c}}^\gamma(\dot{X})$ for some $\gamma\in \bbR$, then we can consider $f$ as being a function of $\rho$ valued in $C^k(\partial X_\theta)$ for every $k\in \bbN$. Specifically, 
\begin{equation}
	f(\rho,\theta) \in \rho^\gamma L^\infty_{\mathrm{c}}([0,1)_\rho; C^k(\partial X_\theta)).
\end{equation}
Then, $\calM f(c) \in C^\infty (\partial X)$ exists for every $c\in \bbC$ with $\Re c<\gamma$, and so we can write $\calM f(c,\theta)$ to denote this function of $c$ and $\theta\in \partial X$. 
Moreover, $\calM f(c)$ depends analytically on $c$ in this domain. More precisely, we have uniform bounds in strips:
\begin{equation}
	\calM f (c,\theta) \in \langle \Im c \rangle^\kappa \scrA(\{c\in \bbC: \gamma_1< \Re c < \gamma_0\} ; C^k(\partial X)  ) 
	\label{eq:Mellin_dec}
\end{equation}
for any $\gamma_0,\gamma_1$ satisfying $\gamma_1<\gamma_0<\gamma$, $m\in \bbR$, and $\kappa \in \bbR$, where $\scrA(\Omega;\calX)$ is the set of uniformly bounded analytic functions on $\Omega\subset \bbC$ valued in the Banach space $\calX$. The decay as $\Im c\to \pm \infty$ follows from the definition of the Mellin transform via an integration-by-parts argument: for $c\neq 0$ such that $\Re c<\gamma$, 
\begin{align}
	\begin{split} 
	\calM f(c) =\lim_{\epsilon \to 0^+} \int_\epsilon^1 f(\rho) \Big( - \frac{1}{c} \frac{\mathrm{d}}{\mathrm{d} \rho} \frac{1}{\rho^{c}} \Big) \dd \rho  &= \lim_{\epsilon \to 0^+} \Big[\frac{ f(\epsilon)}{c\epsilon^c } +\frac{1}{c}\int_\epsilon^1 \rho f'(\rho) \frac{\dd \rho}{\rho^{c+1}}  \Big] \\ 
	&= \frac{1}{c} \int_0^1 \rho f'(\rho) \frac{\dd \rho}{\rho^{c+1}} =  \frac{1}{c} \calM_{\rho\to c}(\rho f'(\rho)).
	\end{split} 
\end{align} 
Since $\rho f'(\rho)$ is in $\calA_{\mathrm{c}}^\gamma(\dot{X})$ as well, the Mellin transform $\calM_{\rho\to c}(\rho f'(\rho))$ is also defined for $\Re c<\gamma$ and satisfies an $L^\infty$ bound in all vertical strips thereof. This gives \cref{eq:Mellin_dec} for $\kappa={-1}$, and an inductive argument extends the conclusion to all $\kappa$.

One form of the \emph{Mellin inversion formula} states that any $f \in \calA_{\mathrm{c}}^\gamma(\dot{X})$ can be recovered from its Mellin transform via the absolutely convergent contour integral 
\begin{equation}
	f(\rho,\theta) = \calM^{-1}\{\calM f\} =  \frac{1}{2\pi i} \int_{\gamma_0-i\infty}^{\gamma_0+i\infty} \rho^c \calM f(c) \dd c ,
	\label{eq:Mellin_inversion}
\end{equation}
for any $\gamma_0< \gamma$. Indeed, this can be reduced to the one-dimensional Fourier inversion formula. Letting $x=-\log \rho$ and $\xi = ic$, 
\begin{equation}
	\calM f = \int_{-\infty}^\infty e^{-i\xi x} f(e^{-x}) \dd x, \quad (\calM^{-1} M)(\rho,\theta) = \frac{1}{2\pi} \int_{i\gamma_0-\infty}^{i\gamma_0+\infty} e^{i\xi x} M \dd \xi,
\end{equation} 
so the inverse Mellin transform $\calM^{-1}$ is just the inverse Fourier transform rewritten using a logarithmic change of variables. 

When applying the Mellin inversion formula in order to extract asymptotics, we need the following: suppose that $M\in \langle \Im c \rangle^{\kappa}\scrA(\{c\in \bbC: \gamma_1< \Re c < \gamma_0\} ; \calX) $ for some Banach space $\calX$, for every $\kappa\in \bbR$. Then, the inverse Mellin transform $\calM^{-1} M$
satisfies $\calA^\gamma( [0,\infty)_\rho  ;\calX )$ for any $\gamma \in (\gamma_1,\gamma_0)$. 
Indeed, it follows immediately from the assumed estimates that $f \in \rho^\gamma L^\infty_{\mathrm{loc}}([0,\infty);\calX)$, and we can differentiate under the integral sign to see that   
\begin{equation}
	(\rho\partial_\rho)^k \int_{\gamma-i\infty}^{\gamma+i\infty} \rho^c M(c) \dd c =  \int_{\gamma-i\infty}^{\gamma+i\infty} c^k \rho^c M(c) \dd c  \in \rho^\gamma L^\infty_{\mathrm{loc}}([0,\infty);\calX)
\end{equation}
for every $k\in \bbN$. Applying this for $\calX$ the spaces $C^m(\partial X)$, we get that if $M \in \langle \Im c \rangle^{\kappa}\scrA(\{c\in \bbC: \gamma_1< \Re c < \gamma_0\} ; C^m(\partial X))$ for every $\kappa,m\in \bbR$, then $\calM^{-1} M\in \bigcap_{m\in \bbR}\calA^\gamma([0,\infty)_\rho; C^m(\partial X)) = \calA^\gamma(\dot{X})$.

If $\calE$ is an index set and $f\in \calA^{\calE}_{\mathrm{c}}(\dot{X})$, then 
\begin{itemize}
	\item $\calM f(c,\theta)$, which is only initially defined if $\Re c < \min \calE$, can be extended to a meromorphic function on the whole complex plane, and
	\item 
	for any $j\in \bbC$, the pole at $j$ is at worst of order $k+1$, where $k$ is the largest number such that $(j,k)\in \calE$. 
\end{itemize}
Indeed, under these assumptions, for any $\alpha\in \bbR$ there exists some $f_\alpha \in \calA^\alpha_{\mathrm{c}}(\dot{X})$ such that we can write 
\begin{equation}
	f(\rho,\theta) = \sum_{(j,k) \in \calE, \Re j\leq \alpha}  f_{j,k}(\theta)  \rho^j \log(\rho)^k  + f_\alpha  
\end{equation}
for some $f_{j,k}(\theta)\in C^\infty(\partial X_\theta)$. Thus, 
\begin{align}
	\calM f(c,\theta) = \int_0^{1} f(c,\theta) \frac{\dd \rho }{\rho^{c+1}} = \calM f_\alpha(c,\theta)+ \sum_{(j,k) \in \calE, \Re j\leq \alpha}  f_{j,k}(\theta) \int_0^1 \rho^{j-c-1} (\log\rho)^k \dd \rho  . 
\end{align}
We already know that $ \calM f_\alpha(c,\theta)$ is well-defined and analytic on $\Re c < \alpha$, and each of the functions in the summand is a meromorphic function of $c$ with poles of the claimed form, as follows from an explicit computation of the integral.

Conversely, suppose that $M$ is a meromorphic $\calX$-valued function on $\{c\in \bbC: \gamma_1 < \Re c < \gamma_0\}$ with finitely many poles $c_1,\ldots,c_N \in \{c\in \bbC: \gamma_1 < \Re c < \gamma_0\}$, and let $\calE$ denote the set of $(j,k)\in \bbC\times \bbN$ such that $j\in \{c_1,\ldots,c_N\}$ and the order of the pole of $M$ at $j$ is at least $k+1$. This is a pre-index set in the sense that we are using the term. Suppose moreover that there exists some $R>0$ such that $M \in \langle \Im c \rangle^\kappa \scrA(\{c\in \bbC: \gamma_1< \Re c < \gamma_0, |\Im z|>R\} ; \calX) $ for every $\kappa \in \bbR$. Then,
\begin{equation}
	\calM^{-1} M = \frac{1}{2\pi i} \int_{\gamma-i\infty}^{\gamma+i\infty} \rho^c M(c) \dd c 
\end{equation}
is well-defined for $\gamma\in (\gamma_1, \min \{\Re c_1,\ldots,\Re c_N\})$, and, as discussed above, $\calM^{-1} M \in \calA^\gamma([0,\infty)_\rho;\calX)$. The contour can now be shifted to the right, through the possible poles. Specifically, 
\begin{equation}
	\calM^{-1} M  = \frac{1}{2\pi i} \int_{\gamma-i\infty}^{\gamma+i\infty} \rho^c M(c) \dd c + \sum_{\substack{c\in \{c_1,\ldots,c_N\} \\ \Re c < \gamma }} \frac{1}{2\pi i} \oint_{c} \rho^{\zeta} M( \zeta) \dd \zeta 
	\label{eq:Mel_inv_1}
\end{equation}
for any $\gamma\in (\gamma_1,\gamma_0) \backslash \{\Re c_1,\ldots,\Re c_N\}$, where $\oint_c$ denotes a clockwise-oriented circular contour containing the pole $c$ but no other members of $\{c_1,\ldots,c_N\}$. The first term on the right-hand side of \cref{eq:Mel_inv_1} is in $\calA^\gamma([0,\infty)_\rho;\calX)$. On the other hand, 
\begin{equation}
	\frac{1}{2\pi i} \oint_{c} \rho^{\zeta} M( \zeta) \dd \zeta = \operatorname{Res}_{\zeta=c} \rho^{\zeta} M(\zeta) .
\end{equation}
Since $\rho^{\zeta} = \rho^c e^{(\zeta-c)\log \rho} = \sum_{\kappa=0}^\infty \rho^{c} (\log \rho)^\kappa (\zeta-c)^\kappa /\kappa!$, if $M(\zeta) = \sum_{\kappa=-(k+1)}^\infty M_\kappa (\zeta-c)^{\kappa}$ denotes the Laurent series of $M(\zeta)$ around $\zeta=c$, then 
\begin{equation}
	\operatorname{Res}_{\zeta=c} \rho^{\zeta} M(\zeta) = \sum_{\kappa=0}^{k} \frac{M_{-1-\kappa} \rho^{c} (\log \rho)^\kappa }{ \kappa!} \in \calA^{(c,k) }([0,\infty)_\rho;\calX ) . 
\end{equation}
So, $\calM^{-1} M \in \calA^{(c,k),\gamma}([0,\infty)_\rho ; \calX)\subseteq \calA^{\calE,\gamma}([0,\infty)_\rho ; \calX)$. Since $\gamma<\gamma_0$ was arbitrary, we get $\calM^{-1} M \in \calA^{\calE,\gamma_0-}([0,\infty)_\rho ; \calX)$. 
Applied to the case where $\calX= C^m(\partial X)$, then, assuming that the above holds for every $m\in \bbR$, we get $\calM^{-1} M \in \calA^{\calE,\gamma_0-}(\dot{X})$.

\section{Proofs omitted from \S\ref{sec:op}}
\label{sec:tedium}

\begin{proof}[Proof of \Cref{prop:metric}]
	Working in some local coordinate chart on $\partial X$, the metric $g$ can be written as a matrix in block form. Doing so, $g-g_0 = \delta g$ for 
	\begin{equation}
	\delta g \in
	\begin{pmatrix}
	\rho^{-4}(F +  \rho^{1+\gimel_0} C^\infty(\dot{X}))  & \rho^{-2+\gimel_0} C^\infty(\dot{X})  \\ 
	\rho^{-2+\gimel_0} C^\infty(\dot{X})   & \rho^{-1+\gimel_0} C^\infty(\dot{X})	\end{pmatrix}, 
	\label{eq:dg_form}
	\end{equation}
	for $F\in  \rho^{1+\gimel} C^\infty([0,1)_\rho)$, 
	near $\dot{X}$, 
	where the bottom-right entries stand for all angular-angular components of the metric and the upper-left entry is the coefficient of $\mathrm{d} \rho^2$. The other entries are the cross terms. 
	
	In the chosen local coordinate chart, $\rho\in (0,1),\theta\in \partial X$, the exactly conic metric $g_0$ has the form 
	\begin{equation}
	g_0 = 
	\begin{pmatrix}
	\rho^{-4} & 0 \\ 
	0 & \rho^{-2} g_{\partial X}
	\end{pmatrix}, \quad \det g_0 = \rho^{-2(d+1)} \det g_{\partial X}, \quad g_0^{-1} = 
	\begin{pmatrix}
	\rho^4 & 0 \\
	0 & \rho^2 g_{\partial X}^{-1} 
	\end{pmatrix}.
	\end{equation}
	Therefore,  
	\begin{equation}
	g_0^{-1} \delta g \in 
	\begin{pmatrix}
	F +  \rho^{1+\gimel_0} C^\infty(\dot{X})  & \rho^{2+\gimel_0} C^\infty(\dot{X})  \\ 
	\rho^{\gimel_0} C^\infty(\dot{X})   & \rho^{1+\gimel_0} C^\infty(\dot{X})	\end{pmatrix}.
	\end{equation}
	One can then compute that 
	\begin{equation}
	\det (1 + g_0^{-1} \delta g) =1 +  F  + \rho^{1+\gimel_0}C^\infty(\dot{X}), \quad 	\det (1 + g_0^{-1} \delta g)^{-1} = (1 +  F)^{-1}  + \rho^{1+\gimel_0}C^\infty(\dot{X}).
	\end{equation} 	
	So,
	\begin{align} 
	\begin{split}
	\det (g) = \det (g_0)  \det (1+g_0^{-1} \delta g) &= \rho^{-2(d+1)} \det(g_{\partial X})\det (1+g_0^{-1} \delta g) \\ 
	&=\rho^{-2(d+1)} \det (g_{\partial X}) (1+ F) + \rho^{-2d-1+\gimel_0}C^\infty(\dot{X}). \end{split}
	\end{align}
	Note that $\det(g) - \det(g_0) \in \rho^{-2(d+1)} F + \rho^{-2d-1+\gimel_0}C^\infty(\dot{X})$.
	
	Via Cramer's formula for $(1+g_0^{-1} \delta g)^{-1}$, we have 
	\begin{multline}
	(1+g_0^{-1} \delta g)^{-1}  \in ((1 +  F)^{-1}  + \rho^{1+\gimel_0}C^\infty(\dot{X})) \\
	\times \begin{pmatrix}
	1+ \rho^{1+\gimel_0} C^\infty(\dot{X})  & \rho^{2+\gimel_0}C^\infty(\dot{X}) \\ 
	\rho^{\gimel_0} C^\infty(\dot{X}) & (1+F+\rho^{1+\gimel_0} C^\infty(\dot{X}) )\operatorname{Adj}(I_{d-1}+\rho^{1+\gimel_0 } C^\infty(\dot{X}) )+\rho^{2+2\gimel_0} C^\infty(\dot{X})
	\end{pmatrix}
	\\ \subseteq \begin{pmatrix}
	(1+F)^{-1} + \rho^{1+\gimel_0} C^\infty(\dot{X})& \rho^{2+\gimel_0} C^\infty(\dot{X}) \\
	\rho^{\gimel_0} C^\infty(\dot{X})& I_{d-1} + \rho^{1+\gimel_0} C^\infty(\dot{X})
	\end{pmatrix}.
	\end{multline}
	So, 
	\begin{equation}
	(1+g_0^{-1} \delta g)^{-1} - I_d\in  \begin{pmatrix}
	F_1+\rho^{1+\gimel_0} C^\infty(\dot{X})& \rho^{2+\gimel_0} C^\infty(\dot{X}) \\
	\rho^{\gimel_0} C^\infty(\dot{X})&  \rho^{1+\gimel_0} C^\infty(\dot{X})
	\end{pmatrix}
	\end{equation}
	for $F_1 = (1+F)^{-1}-1\in  \rho^{1+\gimel} C^\infty([0,1)_\rho)$.
	Since $g^{-1} -g_0^{-1}= ((1+g_0^{-1} \delta g)^{-1} - I_d) g_0^{-1}$, it follows that 
	\begin{equation}
	g^{-1} - g_0^{-1} \in \begin{pmatrix}
	\rho^4 F_1+\rho^{5+\gimel_0} C^\infty(\dot{X})& \rho^{4+\gimel_0} C^\infty(\dot{X}) \\
	\rho^{4+\gimel_0} C^\infty(\dot{X})&  \rho^{3+\gimel_0} C^\infty(\dot{X})
	\end{pmatrix}.
	\end{equation}
	
	Now turning to the formula for the Laplace--Beltrami operator in local coordinates, 
	\begin{equation}
	\triangle_g = \sum_{i,j=1}^d  \frac{1}{\sqrt{|g|}} \frac{\partial}{\partial x_i}\Big( \sqrt{|g|}g^{ij} \frac{\partial}{\partial x_j} \Big) = \sum_{i,j=1}^d \Big(g^{ij} \frac{\partial^2}{\partial x_i x_j}  + \frac{\partial g^{ij}}{\partial x_i} \frac{\partial}{\partial x_j} + \frac{1}{2|g|} \frac{\partial |g|}{\partial x_i} g^{ij}\frac{\partial}{\partial x_j}\Big),
	\end{equation}
	where $|g|=\det(g)$. Here, $x_1=\rho$, and $x_2,\dots,x_d$ are a local coordinate chart for $\partial X$, not Cartesian coordinates.
	So, $\triangle_g$ differs from $\triangle_{g_0}$ by a sum of five terms:
	\begin{itemize}
		\item 
		\begin{multline}
		\sum_{i,j=1}^d (g^{ij} - g_0^{ij} )\frac{\partial }{\partial x_i \partial x_j} \in (\rho^4 F_1+\rho^{5+\gimel_0} C^\infty(\dot{X}))\frac{\partial^2}{\partial \rho^2} + \rho^{4+\gimel_0} C^\infty(\dot{X}) \frac{\partial}{\partial \rho }\calV(\partial X) \\ 
		+ \rho^{3+\gimel_0} C^\infty(\dot{X})\calV(\partial X)^2,
		\label{eq:horrid_1}
		\end{multline}
		\item 
		\begin{multline}
		\sum_{i,j=1}^d \Big(\frac{\partial }{\partial x_i}(g^{ij}- g_0^{ij}) \Big) \frac{\partial}{\partial x_j} \in 
		\begin{pmatrix}
		\rho^3 F_2+\rho^{4+\gimel_0} C^\infty(\dot{X}) \\ \rho^{3+\gimel_0} C^\infty(\dot{X}) 
		\end{pmatrix}^\intercal \nabla \\ 
		\subseteq (\rho^3 F_2 + \rho^{4+\gimel_0} C^\infty(\dot{X}) )\frac{\partial}{\partial \rho} + \rho^{3+\gimel_0} C^\infty(\dot{X})
		\operatorname{Diff}^1(\partial X)
		\label{eq:horrid_2}
		\end{multline}
		for some $F_2 \in \rho^{1+\gimel} C^\infty([0,1)_\rho)$;
		\item 
		\begin{multline}
		\sum_{i,j=1}^d  \frac{g^{ij}-g_0^{ij}}{2|g|} \frac{\partial |g|}{\partial x_i}  \frac{\partial}{\partial x_j}  \in 
		\begin{pmatrix}
		\rho^{-1}F_3 + \rho^{\gimel_0}C^\infty(\dot{X}) \\ C^\infty(\dot{X})
		\end{pmatrix}^\intercal\begin{pmatrix}
		\rho^4 F_1+\rho^{5+\gimel_0} C^\infty(\dot{X})& \rho^{4+\gimel_0} C^\infty(\dot{X}) \\
		\rho^{4+\gimel_0} C^\infty(\dot{X})&  \rho^{3+\gimel_0} C^\infty(\dot{X})
		\end{pmatrix} \nabla  \\ 
		\subseteq 	\begin{pmatrix}
		\rho^{-1}F_3 + \rho^{\gimel_0}C^\infty(\dot{X}) \\ C^\infty(\dot{X})
		\end{pmatrix}^\intercal
		\begin{pmatrix}
		(\rho^4 F_1 + \rho^{5+\gimel_0} C^\infty(\dot{X}) ) \partial_\rho + \rho^{4+\gimel_0} C^\infty(\dot{X})\calV(\partial X) \\ 
		\rho^{4+\gimel_0} C^\infty(\dot{X}) \partial_\rho + \rho^{3+\gimel_0} C^\infty(\dot{X}) \calV(\partial X)
		\end{pmatrix} \\ 
		\subseteq (\rho^3F_1F_3 + \rho^{4+\gimel_0} C^\infty(\dot{X}))\partial_\rho + \rho^{3+\gimel_0} C^\infty(\dot{X}) \calV(\partial X), 
		\label{eq:horrid_3}
		\end{multline}
		where the inclusion follows from the observation that 
		\begin{equation}
		\frac{1}{|g|} \frac{\partial |g|}{\partial x_i} \in 
		\begin{cases}
		\rho^{-1}F_3 + \rho^{\gimel_0}C^\infty(\dot{X}) & (x_i=\rho), \\ 
		C^\infty(\dot{X}) & (\text{otherwise}),
		\end{cases}
		\end{equation}
		for some $F_3 \in  C^\infty([0,1)_\rho)$.
		\item 
		\begin{multline}
		\sum_{i,j=1}^d \frac{g_0^{ij}}{2|g|} \frac{\partial}{\partial x_i} (|g|-|g_0|)  \frac{\partial}{\partial x_j} = \frac{g_0^{11}}{2|g|} \frac{\partial}{\partial \rho }(|g|-|g_0|)  \frac{\partial}{\partial \rho} + \sum_{i,j=2}^{d}  \frac{g_0^{ij}}{2|g|} \frac{\partial}{\partial \theta_i} (|g|-|g_0|)  \frac{\partial}{\partial \theta_j} \\
		\in \frac{\rho^{2d+6}}{2} \Big( \frac{1}{\det (g_{\partial X}) (1+F)} + \rho^{1+\gimel_0}C^\infty(\dot{X}) \Big)(\rho^{-2d-3} \det (g_{\partial X})F_4 + \rho^{-2d-2+\gimel_0}C^\infty(\dot{X})) \frac{\partial}{\partial \rho} 
		\\ + \frac{1}{2}  \sum_{i,j=2}^d g_0^{ij} \Big(  \frac{1}{\det (g_{\partial X}) (1+F)} + \rho^{1+\gimel_0}C^\infty(\dot{X}) \Big) \Big( F\frac{\partial \det (g_{\partial X})}{\partial \theta_i} +\rho^{1+\gimel_0} C^\infty(\dot{X})\Big)  \frac{\partial}{\partial \theta_j}
		\\
		\subseteq  \Big( \frac{\rho^3 F_4}{2(1+F)} + \rho^{4+\gimel_0} C^\infty(\dot{X}) \Big)\frac{\partial}{\partial \rho} + \rho^{3+\gimel_0} \calV(\partial X) + \calB, 
		\label{eq:horrid_4}
		\end{multline}
		where $\theta_j = x_j$ for $j\geq 2$,  $F_4 \in \rho^{1+\gimel} C^\infty([0,1)_\rho)$, and 
		\begin{equation}
		\calB = \frac{1}{2} \frac{F}{\det (g_{\partial X} )(1+F) }\sum_{i,j=2}^d g_0^{ij}  \frac{\partial \det (g_{\partial X})}{\partial \theta_i} \frac{\partial}{\partial \theta_j}.
		\end{equation}
		\item 
		\begin{multline}
		\sum_{i,j=1}^d \frac{g_0^{ij}}{2} \frac{\partial |g_0|}{\partial x_i} \Big(\frac{1}{|g|} - \frac{1}{|g_0|} \Big) \frac{\partial}{\partial x_j} =  \Big(\frac{1}{|g|} - \frac{1}{|g_0|} \Big) \Big[  \frac{\rho^4}{2} \frac{\partial |g_0|}{\partial \rho} \frac{\partial}{\partial \rho} + \sum_{i,j=2}^d \frac{g_0^{ij}}{2} \frac{\partial |g_0|}{\partial \theta_i} \frac{\partial}{\partial \theta_j}\Big] \\ 
		\subseteq \rho^{2d+2}\Big( \frac{F_5}{ \det(g_{\partial X}) }  +  \rho^{1+\gimel_0}C^\infty(\dot{X})  \Big) \Big[ -(d+1) \rho^{-2d+1} \det (g_{\partial X})\frac{\partial}{\partial \rho}  + \sum_{i,j=2}^d \frac{g_0^{ij}}{2} \frac{\partial |g_0|}{\partial \theta_i} \frac{\partial}{\partial \theta_j} \Big] \\
		\subseteq (-(d-1)\rho^{3} F_5 + \rho^{4+\gimel} C^\infty(\dot{X})) \frac{\partial}{\partial \rho} + \rho^{3+\gimel_0} C^\infty(\dot{X}) \calV(\partial X) +\calB' 
		\label{eq:horrid_5}
		\end{multline}
		for some $F_5 \in \rho^{1+\gimel} C^\infty([0,1)_\rho)$, where 
		\begin{equation}
		\calB' = \frac{1}{2}\frac{F_5}{\det (g_{\partial X})} \sum_{i,j=2}^d  g_0^{ij} \frac{\partial \det(g_{\partial X})}{\partial \theta_i} \frac{\partial}{\partial\theta_j} .
		\end{equation}
		Indeed, this follows from 
		\begin{equation}
		\frac{1}{|g|} - \frac{1}{|g_0|} \in \frac{\rho^{2d+2} }{\det(g_{\partial X}) } \Big( \frac{1}{(1+F) } - 1 \Big) + \rho^{2d+3+\gimel_0} C^\infty(\dot{X}) 
		\end{equation}
		and $g_0^{ij} (\partial_{\theta_i} |g_0|) \partial_{\theta_j} \in \rho^{-2d} \calV(\partial X)$ for $i,j\geq 2$, from which we also get that $F_5 = -F/(1+F)$, so $\calB'=-\calB$.
	\end{itemize}
	All of the terms on the right-hand sides of \cref{eq:horrid_1}, \cref{eq:horrid_2}, \cref{eq:horrid_3}, \cref{eq:horrid_4}, \cref{eq:horrid_5} 
	lie in $\rho^{3+\gimel} \operatorname{Diff}^2_{\mathrm{b}}([0,1)_\rho) +  \rho^{3+\gimel_0} \operatorname{Diff}^2_{\mathrm{b}}(X)$, with the exception of $\calB,\calB' = -\calB$. Fortunately, when we combine these exceptional contributions to $\triangle_g-\triangle_{g_0}$, they cancel, so we are left only with terms in $\rho^{3+\gimel} \operatorname{Diff}^2_{\mathrm{b}}([0,1)_\rho) +  \rho^{3+\gimel_0} \operatorname{Diff}^2_{\mathrm{b}}(X)$.
\end{proof}

\begin{proof}[Proof of \Cref{prop:Schrodinger_form}]
	We prove the two parts of the proposition in turn. In both cases, it is immediate that the operators considered have the desired forms in the interior of $X$. So, we only describe the situation near $\partial X$. For convenience, we repeat here the definitions (\cref{eq:misc_025})
	\begin{equation}
		\hat{O}(\sigma) = e^{i\sigma r} r^{i\sigma \frakm/2} \circ O(\sigma)\circ e^{-i\sigma r} r^{-i\sigma \frakm/2}, \qquad \check{O}(\sigma) = e^{-i\sigma r} r^{-i\sigma \frakm/2} \circ O(\sigma)  \circ e^{i\sigma r} r^{i\sigma \frakm/2}, 
	\end{equation}
	of $\hat{O},\check{O}$ for $O$ a differential operator.
	\begin{enumerate}[label=(\Roman*)]
		\item Suppose that $P(\sigma)$ has the form described in and below \cref{eq:Pform}. We then compute 
		\begin{multline}
		\check{P}(\sigma) = \check{\triangle}_g  + 2i \sigma(1-\chi) \frac{\partial}{\partial r}-2 \sigma^2(1-\chi) +\frac{i\sigma(d-1)}{r} + \check{L} + \sigma \check{Q} + \sigma^2 R  -  \frac{\sigma^2 \frakm (1-\chi)}{r}.
		\end{multline}
		(Recall that $R$ is scalar, so $\check{R}=R$.)
		
		Starting with the first order terms $\check{L},\check{Q}$: note that if $A\in \operatorname{Diff}^1_{\mathrm{b}}(X)$, then $\check{A} - A\in \sigma \rho^{-1} C^\infty(X)$, and if $A\in \operatorname{Diff}^1_{\mathrm{b}}([0,1)_\rho)$, then we can improve this to $\check{A}-A \in \sigma \rho^{-1} C^\infty([0,1)_\rho)$. So, 
		\begin{align}
		\check{L} &= L + \sigma L_1\;\text{ for }L_1 \in \rho^{2+\beth}C^\infty([0,1)_\rho) + \rho^{2+\beth_0}C^\infty(X), \\ 
		\check{Q} &= Q + \sigma Q_1\text{ for }Q_1 \in \rho^{1+\beth_1}C^\infty([0,1)_\rho) + \rho^{1+\beth_2}C^\infty(X).
		\end{align}

		Likewise, writing $\check{\triangle}_g = \check{\triangle}_{g_0} + r^{-1} \frakm\check{\partial_r^2} + \check{O}$ for 
		\begin{equation} 
			O = \triangle_g- \triangle_{g_0} - r^{-1} \frakm \partial_r^2 \in \rho^{\max\{4,3+\gimel\}} \operatorname{Diff}^2_{\mathrm{b}}([0,1)_\rho) +  \rho^{3+\gimel_0} \operatorname{Diff}^2_{\mathrm{b}}(X)
		\end{equation} 
		(this being \cref{eq:misc_018}) we compute
		\begin{equation}
			\check{\triangle}_{g_0} = \triangle_{g_0} - i\sigma \Big( 2 +\frac{\frakm}{r} \Big)\frac{\partial}{\partial r} -\frac{i\sigma(d-1)}{r} + \frac{i \sigma \frakm}{r^2} \Big(1- \frac{d}{2} \Big) + \sigma^2\Big(1 + \frac{\frakm}{r} + \frac{\frakm^2}{4r^2}   \Big),
			\label{eq:misc_276}
		\end{equation}
		\begin{equation}
			r^{-1} \frakm\check{\partial_r^2} = r^{-1} \frakm \frac{\partial^2}{\partial r^2} + \frac{i\sigma \frakm}{r}\Big(2 + \frac{\frakm}{r} \Big) \frac{\partial}{\partial r}  -\frac{\sigma^2 \frakm^3}{4r^3} - \frac{i \sigma \frakm^2}{2r^3} - \frac{ \sigma^2 \frakm}{r}- \frac{\sigma^2 \frakm^2}{r^2}
			\label{eq:misc_277}
		\end{equation}
		near $\partial X$ and 
		\begin{multline}
			\qquad\qquad\qquad\qquad\check{O}- O \in \sigma\rho^{\max\{3,2+\gimel\}} \operatorname{Diff}^1_{\mathrm{b}}([0,1)_\rho ) + \sigma \rho^{2+\gimel_0} \operatorname{Diff}^1_{\mathrm{b}}(X) \\ + \sigma^2  \rho^{\max\{2,1+\gimel\}} C^\infty([0,1)_\rho) + \sigma^2 \rho^{1+\gimel_0} C^\infty(X).\qquad\qquad
			\label{eq:misc_144}
		\end{multline}
		By the definition of $\gimel$, we have $\gimel=0$ if $\frakm\neq 0$. 
		So, $\check{\triangle}_g = \triangle_g - 2i \sigma \partial_r + \sigma^2- i \sigma(d-1) r^{-1}+ \sigma O_1 + \sigma^2 O_2$ for 
		\begin{align}
		O_1 &\in  \rho^{2+\gimel} \operatorname{Diff}^1_{\mathrm{b}}([0,1)_\rho ) + \rho^{2+\gimel_0} \operatorname{Diff}^1_{\mathrm{b}}(X),\\ 
		O_2 &\in \rho^{1+\gimel} C^\infty([0,1)_\rho) +  \rho^{1+\gimel_0} C^\infty(X).
		\end{align}
		Here, we used that, if $O\in \operatorname{Diff}^2_{\mathrm{b}}(X)$, we have $\check{O}  -  O \in \sigma^2 \rho^{-2} C^\infty(X) + \sigma \rho^{-1} \operatorname{Diff}^1_{\mathrm{b}}(X)$. If $O \in \operatorname{Diff}^2_{\mathrm{b}}([0,1)_\rho)$, then we can improve this to $\check{O}  - O \in \sigma^2 \rho^{-2} C^\infty([0,1)_\rho) + \sigma \rho^{-1} \operatorname{Diff}^1_{\mathrm{b}}([0,1)_\rho)$.
		
		Combining these ingredients, we see that
		\begin{equation}
		\check{P}(\sigma) = \triangle_{g}  - \sigma^2 + L + \sigma\underbrace{\Big(- 2i \chi \frac{\partial}{\partial r}+L_1 + Q + O_1\Big)}_{P_1} + \sigma^2\underbrace{(2\chi +Q_1+R + O_2-\frakm(1-\chi)r^{-1})}_{P_2}
		\end{equation}
		near $\partial X$. So, the stated form \cref{eq:Pcheck_form} of $\check{P}(\sigma)$ holds, with $P_1,P_2$ as underlined. Thus, 
		\begin{align}
		P_1 &\in \rho^{2+\min\{\gimel,\beth,\beth_1\} }\operatorname{Diff}^1_{\mathrm{b}}([0,1)_\rho ) + \rho^{2+\min\{\gimel_0,\beth_0,\beth_2\}} \operatorname{Diff}^1_{\mathrm{b}}(X),  \\
		P_2 &\in \rho^{1+\min\{\gimel,\beth_1,1+\beth_3\} }C^\infty([0,1)_\rho ) + \rho^{1+\min\{\gimel_0,\beth_2,1+\beth_4\}} C^\infty(X),
		\end{align}
		which is what was claimed in the proposition.
		\item The converse direction is similar. Given $P_0(\sigma) = P_0 -\sigma^2$ for $P_0 = \triangle_g + L$, we have 
		\begin{equation}
		\hat{P}_0(\sigma) = \hat{\triangle}_g+\hat{L} - \sigma^2. 
		\end{equation}
		As in the previous part, $\hat{L} = L + \sigma L_1$ for $L_1 \in \rho^{2+\beth}C^\infty([0,1)_\rho) + \rho^{2+\beth_0}C^\infty(X)$, and
		\begin{equation} 
			\hat{\triangle}_g = \hat{\triangle}_{g_0}  + \frakm r^{-1} \hat{\partial_r^2} + \hat{O}
		\end{equation} 
		for $O=\triangle_g - \triangle_{g_0} - \frakm r^{-1}\partial_r^2$. The difference $\hat{O}-O$ is in the set on the right-hand side of \cref{eq:misc_144}.  
	 	The conjugate of the computation leading to \cref{eq:misc_276}, \cref{eq:misc_277} yields 
		\begin{equation}
		\hat{\triangle}_{g_0} = \triangle_{g_0} + i\sigma \Big( 2 +\frac{\frakm}{r} \Big)\frac{\partial}{\partial r} +\frac{i\sigma(d-1)}{r} - \frac{i \sigma \frakm}{r^2} \Big(1- \frac{d}{2} \Big) + \sigma^2\Big(1 + \frac{\frakm}{r} + \frac{\frakm^2}{4r^2}   \Big),
		\label{eq:misc_286}
		\end{equation}
		\begin{equation}
				r^{-1} \frakm\hat{\partial_r^2} = r^{-1} \frakm \frac{\partial^2}{\partial r^2} - \frac{i\sigma \frakm}{r}\Big(2 + \frac{\frakm}{r} \Big) \frac{\partial}{\partial r}  -\frac{\sigma^2 \frakm^3}{4r^3} + \frac{i \sigma \frakm^2}{2r^3} - \frac{ \sigma^2 \frakm}{r}- \frac{\sigma^2 \frakm^2}{r^2}.
				\label{eq:misc_287}
		\end{equation}
		So, altogether, we have 
		\begin{multline}
		P(\sigma) = \triangle_{g_0} + 2i \sigma \frac{\partial}{\partial r} + \frac{i\sigma(d-1)}{r} + L  \\ + \sigma \underbrace{(\rho^{2+\min\{\gimel,\beth\}} \operatorname{Diff}^1_{\mathrm{b}} ([0,1)_\rho ) + \rho^{2+\min\{\gimel_0,\beth_0\}} \operatorname{Diff}^1_{\mathrm{b}}(X))}_{Q}  \\ + \sigma^2 \underbrace{( \rho^{\max\{2,1+\gimel \}}  C^\infty([0,1)_\rho) + \rho^{1+\gimel_0 } C^\infty(X)  )}_R.
		\end{multline}
		The one key cancellation that occurs is between the $\sigma^2\frakm/r$ terms in \cref{eq:misc_286}, \cref{eq:misc_287}. Without this cancellation, we would only have $R\in \rho^{1+\gimel }  C^\infty([0,1)_\rho) + \rho^{1+\gimel_0 } C^\infty(X)$.

		The operators $Q,R$ defined here have the desired forms if we take $\beth,\cdots,\beth_4$ as in the proposition statement.
	\end{enumerate}
\end{proof}

\begin{proof}[Proof of \Cref{prop:Vasy_absorption}]
	Our goal will be to deduce this from \cite[Theorem 1.1]{VasyLA}. 
	There are two reasons why \Cref{prop:Vasy_absorption} is not a special case of \cite[Theorem 1.1]{VasyLA}:
	\begin{itemize}
		\item Vasy does not conjugate by $r^{\pm i \sigma \frakm/2}$, 
		\item depending on $Q,R$, it may not be the case that $\check{P}$ is the spectral family of a Schr\"odinger operator.
	\end{itemize}
	The former is actually irrelevant, because $r^{\pm i \sigma \frakm/2}$ is an isomorphism on all of the Sobolev spaces in \cite[Theorem 1.1]{VasyLA}, so Vasy's theorem applies with our conjugations. 
	
	Regardless of $Q,R$, for each individual $\sigma>0$, we have 
	\begin{equation}
	\check{P}(\sigma) = \check{P}(\sigma;\sigma),
	\end{equation}
	where, for each fixed $\omega\in \bbR$, $\check{P}(\sigma;\omega)$ is the spectral family of a Schr\"odinger operator $\check{P}(0;\omega)$, which depends on $\omega$. To be concrete, we can take, in terms of the operators in \cref{eq:Pcheck_form}, 
	\begin{equation}
	\check{P}(\sigma;\omega) = P_0 - \sigma^2 + \omega P_1 + \omega^2 P_2. 
	\end{equation}
	So, \cite[Theorem 1.1]{VasyLA} gives the most of the proposition. (Note that what we call $P$ Vasy calls $\hat{P}$, and what we call $\check{P}$ is what Vasy just calls $P$.)  
	
	What remains to be proven is that, when our original operator is symmetric, it is the case that, for each $\omega\in \bbR$, 
	\begin{equation} 
		\check{P}(0;\omega) =\check{P}(0;\omega)^*,
	\end{equation} 
	i.e.\ that $P_0 +\omega P_1 + \omega^2 P_2$ is symmetric on $C_{\mathrm{c}}^\infty(X^\circ)$ with respect to the $L^2(X,g)$ inner product. Indeed, it follows from the symmetry of $P$ that $\check{P}(\sigma) = \check{P}(\sigma)^*$ for all $\sigma \in \bbR$. We can deduce the symmetry of each of $P_0,P_1,P_2$ from this fact:
	\begin{itemize}
		\item Since $\check{P}(0) = P_0$, the symmetry of $P_0$ is immediate. 
		\item 
			\begin{equation}
		P_1 = \frac{\partial \check{P}(\sigma)}{\partial \sigma} \Big|_{\sigma =0} =\frac{\partial \check{P}(\sigma)^*}{\partial \sigma} \Big|_{\sigma =0}  = P_1^*
		\end{equation}
		\item 
		
		\begin{equation}
		P_2 = 1 +  \frac{1}{2}\frac{\partial^2 \check{P}(\sigma)}{\partial \sigma^2} \Big|_{\sigma =0} = 1 +  \frac{1}{2}\frac{\partial^2 \check{P}(\sigma)^*}{\partial \sigma^2} \Big|_{\sigma =0} = P_2^*. 
		\end{equation}
		
	\end{itemize}

	So, $\check{P}(\sigma;\omega)$ satisfies the assumptions of \cite[\S3]{VasyLA}, specifically with $\Im \alpha_\pm = 0$. 
	We may therefore apply \cite[Theorem 1.1]{VasyLA} to conclude the proposition. 
\end{proof}

\begin{proof}[Proof of \Cref{prop:Vasy_low}]
	First note that the operator $P(\sigma)$ above falls into the framework discussed in \cite[\S2, eq. 2.1]{VasyLagrangian}. So, we can appeal to \cite[Thm.\ 2.5]{VasyLagrangian}, which, in the case at hand, gives the same conclusion as \cite[Thm.\ 2.1]{VasyLagrangian}. (The latter theorem is stated in a slightly less general case than needed here, but the additional terms in $P(\sigma)$ do not change the statement of \cite[Thm.\ 2.5]{VasyLagrangian}.)
	
	In order to apply this theorem, we need that $P(0)$ has trivial kernel acting on $H_{\mathrm{b}}^{\infty,(d-4)/2}(X)$; \cref{eq:Sobolev} says that
	\begin{equation}
	H_{\mathrm{b}}^{\infty,(d-4)/2}(X) \subseteq \calA^{d-2}(X),
	\end{equation}
	so our assumption that $\ker_{\calA^{d-2}(X)} P(0) = \{0\}$ is sufficient to verify this hypothesis.

	So, if $r>-1/2$ and $\ell,\nu$ are as in the proposition statement, then \cite[Thm.\ 2.1]{VasyLagrangian} says that, for any $s\in \bbR$, there exists some $\sigma_0>0$ and $C>0$ such that 
	\begin{equation}
	\lVert (\rho+\sigma)^\nu u \rVert_{H_{\mathrm{sc,b,res}}^{s,r,\ell} } \leq C \lVert (\rho+\sigma)^{\nu-1} P(\sigma) u \rVert_{H_{\mathrm{sc,b,res}}^{s-1,r+1,\ell+1} }
	\label{eq:misc_0h7}
	\end{equation}
	holds for all $u,\sigma$ such that the left side is finite, 
	where the norms $\lVert \bullet \rVert_{H_{\mathrm{sc,b,res}}^{s,r,\ell}}$ are defined in \cite{VasyLA}.

	In order to put this in a more useful form, we can use the equivalence
	\begin{equation}
	\lVert \bullet\rVert_{H_{\mathrm{b}}^{s,\ell}} \approx \lVert \bullet\rVert_{H_{\mathrm{sc,b,res}}^{s,s+\ell,\ell}} ,
	\end{equation}
	which holds for all $s,\ell\in \bbR$ (with the $\sigma$-\emph{in}dependent constants, notationally suppressed by the ``$\approx$,'' depending $s,\ell$); see \cite[Eq. 3.5]{VasyLA}. \Cref{eq:misc_0h7} therefore gives, for $s,\ell,\nu$ as in the proposition statement, 
	\begin{multline}
	\lVert (\rho+\sigma)^\nu u \rVert_{H_{\mathrm{b}}^{s,\ell} } \approx \lVert (\rho+\sigma)^\nu u \rVert_{H_{\mathrm{sc,b,res}}^{s,s+\ell,\ell} } \lesssim \lVert (\rho+\sigma)^{\nu-1} P(\sigma)u \rVert_{H_{\mathrm{sc,b,res}}^{s-1,s+\ell+1,\ell+1} } \\ \lesssim \lVert (\rho+\sigma)^{\nu-1} P(\sigma) u \rVert_{H_{\mathrm{sc,b,res}}^{s,s+\ell+1,\ell+1} }  \approx \lVert (\rho+\sigma)^{\nu-1} P(\sigma) u \rVert_{H_{\mathrm{b} }^{s,\ell+1} }.
	\label{eq:misc_1h7}
	\end{multline}
\end{proof}

\section{An ODE example with $\frakm\neq 0$}
\label{sec:example}

Consider the ordinary differential operator $P(\sigma)$ on $\bbR^+_r$ given by 
\begin{equation}
	P(\sigma) =  \Big(-1+\frac{\frakm}{r} \Big) \frac{\partial^2}{\partial r^2}  -\sigma^2. 
\end{equation}
The ODE $Pu=0$ is a form of the confluent hypergeometric equation, and a basis of solutions is given by 
\begin{equation}
	u_1= e^{i(r-\frakm) \sigma} U(-2^{-1} i \frakm\sigma,0,2i(\frakm-r) \sigma) , \quad u_2 = e^{-i(r-\frakm)} L(2^{-1} i \frakm\sigma,-1,2i (\frakm-r) \sigma), 
\end{equation}
where $U(a,b,z)$ denotes Tricomi's confluent hypergeometric function and $L(n,a,z)$ denotes the generalized Laguerre function of order $n$.

The large-$z$ asymptotics of $U(a,b,z)$ are given by $U(a,b,z)\sim z^{-a}$ in the relevant region of the complex plane \cite[\href{http://dlmf.nist.gov/13.2E6}{13.2.6}]{NIST}, so 
\begin{equation}
	u_1 \sim e^{i (r-\frakm)\sigma} (2i (\frakm-r) \sigma)^{i \sigma \frakm/2 } 
\end{equation}
as $r\to\infty$ for $\sigma$ fixed. This shows how $\frakm$ being nonzero results in an additional oscillation $r^{i\sigma \frakm/2} = \exp(2^{-1} i \sigma \frakm \log r)$ being present in solutions.

\section{Index of notation}
\label{sec:index}

\subsection*{Index sets}
 
A \emph{pre-index set} is a subset $\calE\subset \bbC\times \bbN$ such that $\calE\cap \{z\in \bbC:\Re z<\alpha\}$ is finite for each $\alpha\in \bbR$ and such that $(j,k+1)\in \calE\Rightarrow (j,k)\in \calE$ for all $j\in \bbC$ and $k\in \bbN$.\footnote{We use $\bbN=\{0,1,2,3,\dots\}$, $\bbN_+=\{1,2,3,\dots\}$. } Then, if $\calE$ also satisfies $(j,k)\in \calE\Rightarrow (j+1,k)\in \calE$, then $\calE$ is called an \emph{index set}.

We use the caligraphic typeface, e.g.\ $\calE,\calF,\calI,\calJ,\calK$, to denote index or pre-index sets. (Exceptions are `$\calA,\calB,\calC,\calM$', which are reserved for other uses.)
\begin{itemize}
	\item We use $(j,k)$ as an abbreviation for the index set $\{(j+\delta,\kappa):\delta\in \bbN,\kappa \leq k\}$ the point $(j,k)\in \bbC\times \bbN$ generates.  Thus, $(j,k)$ always denotes either a point or the index set that point generates. Which will be clear from context. 
	\item $\{(j,k)\}\subset \bbC\times \bbN$ denotes the singleton consisting of the point $(j,k) \in \bbC\times \bbN$.
	\item If $\calE$ is a pre-index set, then $\dot{\calE}$ denotes the smallest index set containing it (the index set ``it generates''). That is, 
	\begin{equation*}
		\dot{\calE} = \{(j+\delta,k) :(j,k)\in \calE,\delta\in \bbN\}.
	\end{equation*}
	\item  If $\calE$ is a pre-index set, then \begin{equation}
		\{(j,k)\}\uplus \calE = \{(j,k)\}\cup \calE \cup \{(j+\delta,k'+\kappa+1) : (j,\kappa)\in \calE, k'\in \{0,\dots,k\}\}, 
	\end{equation}
	which consists of the singleton $\{(j,k)\}$, the terms in $\calE$, and possibly a few more logarithmic terms. It is closely related to the notion of the ``extended union'' of the given index sets, but it can be smaller. 
	\item  Similarly, 
	\begin{equation}
		(j,k)\uplus \calE = (j,k)\cup \calE \cup \{(j+\delta,k'+\kappa+1) : (j,\kappa)\in \calE, \delta\in \bbN, k'\in \{0,\dots,k\}\}.
		\label{eq:uplus}
	\end{equation}
	This consists of $(j+\delta,k)$ for each $\delta\in \bbN$, the terms in $\calE$, and, if $(j,0)\in \calE$, then it consists of a few more logarithmic terms.
	Note that this is a pre-index set and is an index set if $\calE$ is.

	We will abbreviate $(j,0)\uplus \calE=j\uplus \calE$ in some places.
	\item If $\calE$ is a pre-index set, $\Pi \calE=\{\Re j:(j,k)\in \calE\}$. 
	\item $\calC_\ell \uplus \calE = \cup_{j\geq \ell} \{(c_j,0)\}\uplus \calE $. (See below for the definition of $c_j$.)
	\item $\calC_\ell \dot{\uplus} \calE = \dot{\calI}$ for $\calI= \cup_{j\geq \ell} \{(c_j,0)\}\uplus \calE $.
	\item $c+\calE = \{ (c+j,k):(j,k)\in \calE\}$, and $(c,\kappa)+\calE = \{ (c+j,k+\varkappa) : (j,k)\in \calE, \varkappa \in \{0,\dots,\kappa\} \}$.
	\item In \S\ref{sec:0},  $\calI_j[\calF_\bullet]$ and $\calI[\calF_\bullet,\ell]$ denote pre-index sets defined in \S\ref{subsec:0main}. Their definition is via a complicated recursive procedure, so we do not provide closed form expressions (except for a few cases, see \Cref{ex:spherical_breaking}, \Cref{ex:sym}, \Cref{ex:Coulomb}). In \S\ref{sec:1} and \S\ref{sec:2}, $\calI_j[\calF_\bullet]$ and $\calI[\calF_\bullet,\ell]$ denote the index sets generated by the pre-index sets above.
	\item $\calB_\ell[\calG]$ is the smallest pre-index set (in the appendix) or index set (in the body) containing $(b_j,0)\uplus \calG$ for every $j\geq \ell$.
\end{itemize}

\subsection*{Manifolds-with-corners (mwc)}
\begin{itemize}
	\item $X$ is our manifold-with-boundary, with dimension $d\geq 3$. The asymptotically Euclidean case involves 
	\begin{equation*} 
		X=\overline{\bbR^d}=\bbR^d\sqcup \infty \bbS^{d-1},
	\end{equation*} 
	where $\infty \bbS^{d-1}$ is a ``sphere at infinity.'' More precisely, 
	\begin{equation*}
		\overline{\bbR^d} = \bbR^d\cup ([0,\infty)_{1/r}\times \bbS^{d-1}_\theta), 
	\end{equation*}
	where on the right-hand side we are implicitly identifying nonzero $x\in \bbR^d$ with $(1/|x|,x/|x|)$. 
	This is called the radial compactification of $\bbR^d$. 
	\item $\dot{X}\subset X$ is a boundary collar of $X$, a neighborhood of $\partial X$ which we identify with $[0,1)_{\rho}\times \partial X$ using a fixed diffeomorphism.
	\item  $X^+_{\mathrm{res}}$ is the result of blowing up the corner of $X\times [0,\infty)_\sigma$. This is the mwc on which we do low-energy scattering theory. See \Cref{fig}.
	\item In \S\ref{sec:2}, $\tilde{X}^+_{\mathrm{res}}$ is the result of blowing up the small-$\hat{r}$ corner of $[0,\infty]_{\hat{r}}\times \partial X\times [0,\infty)_\sigma$. 
	\item $\mathrm{bf},\mathrm{tf},\mathrm{zf}$ are the three boundary hypersurfaces of $X^+_{\mathrm{res}},\dot{X}^+_{\mathrm{res}}$, and $\tilde{X}^+_{\mathrm{res}}$ relevant to low-energy analysis. The names stand for ``boundary face,'' ``transition face,'' and ``zero (energy) face,'' respectively. The face $\mathrm{bf}$ can be thought of as corresponding to the large-$r$, $r\gg 1/\sigma$ regime. The face $\mathrm{zf}$ can be thought of as corresponding to the small-$\sigma$, $r\ll 1/\sigma$ regime. As the name indicates, $\mathrm{tf}$ is the transitional regime, and which $r\to \infty$ and $\sigma\to 0$ inversely. Note that 
	\begin{align*}
		\mathrm{bf}&\cong \partial X \times [0,\infty)_\sigma \\ 
		\mathrm{tf} &\cong \partial X \times [0,\infty]_{\hat{r}} \\ 
		\mathrm{zf}&\cong X, 
	\end{align*}
	naturally, where $\hat{r}=\sigma r$. 
	 
	\item $\rho$ is a boundary-defining-function (bdf) of $\partial X$ in $X$. 
	\item $\rho_{\mathrm{f}}$ is a bdf of $\mathrm{f}$ in $X^+_{\mathrm{res}}$. Concretely, 
	\begin{equation*}
		\rho_{\mathrm{zf}} = \frac{\sigma}{\sigma+\rho},\quad \rho_{\mathrm{tf}}=\sigma+\rho,\quad \rho_{\mathrm{bf}} = \frac{\rho}{\sigma+\rho} 
	\end{equation*}
	work.
\end{itemize}
Analogous notation is used with $\dot{X},[0,1)$ in place of $X$. 
\subsection*{Function spaces}
\begin{itemize}
	\item $H_{\mathrm{b}}^{m,s}(X)=\{u\in \calS':\Psi_{\mathrm{b}}^{-m,-s}u\in L^2(X,g)\}$ is the b-Sobolev space on $X$, with $m$ differential orders and $s$ decay orders. (Higher $s$ means more decay.)
	
	 Notice that $H_{\mathrm{b}}^{0,0}= L^2(X,g)\neq L^2(X,g_{\mathrm{b}})$ for $g_{\mathrm{b}}$ a b-metric (i.e.\ asymptotically cylindrical metric).
\end{itemize}
Our main function spaces are the partially polyhomogeneous (phg) spaces $\calA^\bullet(\bullet)$. 
\begin{itemize}
	\item $\calA^{\calE,\calF,\calG}(X^+_{\mathrm{res}})$ are the phg spaces with index set $\calE$ at $\mathrm{bf}$, $\calF$ at $\mathrm{tf}$, and $\calG$ at $\mathrm{zf}$. 
	\item $\calA^{\calE}(\mathrm{zf})$ is the phg space on $\mathrm{zf}$ with index set $\calE$ at $\mathrm{zf}\cap\mathrm{tf}$. 
	\item $\calA^{\calE,\calF}(\mathrm{tf})$ is the phg space on $\mathrm{tf}$ with index set $\calE$ at $\mathrm{tf}\cap\mathrm{bf}$ and index set $\calF$ at $\mathrm{zf}\cap\mathrm{tf}$. 
	\item $\calY_0,\calY_1,\calY_2,\dots$ denote subspaces of eigenspaces of the boundary Laplacian $\triangle_{\partial X}$ with non-decreasing eigenvalues $\lambda_j$. 
	\item $\calA^{\bullet}(\bullet;\ell)$ denotes a phg space on some boundary collar in which functions are orthogonal to the elements of $\calY_0,\dots,\calY_{\ell-1}$. 
	\item If we write something like $\calA^{\alpha,\bullet}$ for $\alpha$ a \emph{number} instead of index set, what is meant is, near the face $\mathrm{f}$ to which $\alpha$ corresponds, a conormal function space in $\rho_{\mathrm{f}}$ valued in $\calA^\bullet(\mathrm{f})$.
	\item  For $\calE$ and index set and $\alpha$ a number, 
	$\calA^{(\calE,\alpha),\bullet}= \calA^{\calE,\bullet}+\calA^{\alpha,\bullet}$.
\end{itemize}

\subsection*{Operators}

\begin{itemize}
	\item $\triangle_g$ denotes the positive semidefinite Laplace--Beltrami operator associated to a Riemannian metric $g$ (the negative of the usual definition).
	\item $P$ is our main operator, described in \S\ref{sec:op}.
	\item $N_{\mathrm{f}}=N_{\mathrm{f}}(P)$ is the normal operator (``model operator,'' ``leading terms'') in $P$ at the face $\mathrm{f}\in \{\mathrm{zf},\mathrm{tf},\mathrm{bf}\}$. This means that $P-N_{\mathrm{f}}$ has more decay than $P$ at $\mathrm{f}$. 
	\item $\hat{N}_{\mathrm{tf}}= \sigma^{-2}N_{\mathrm{tf}}$.
	\item $R(\sigma^2 \pm i 0)$ is the limiting resolvent, defined in \S\ref{subsec:Sommerfeld}
	\item $\beth_\bullet\in \bbN$ are various terms telling us at which order (in terms of $1/r$) below subleading various terms in the PDE enter. For example, $\beth$ tells us the decay order of the difference between the metric $g$ and the Minkowski metric --- if $\beth=0$, then this means $1/r$ decay. The term $\beth_0$ is the same, but restricting attention to spherically asymmetric terms. So, on spherically symmetric spacetimes, $\beth_0=\infty$, even if $\beth=0$. On asymptotically Kerr spacetimes, $\beth=0$ and $\beth_0=1$. 
	
	The terms $\beth_j$ for $j\geq 1$ are similar, except about other terms in the PDE. See \S\ref{subsec:op} for the details.
	\item In \S\ref{sec:op}, \S\ref{sec:tedium}, we use $\gimel,\gimel_0$ to refer to orders of subleading terms in the metric. These will become $\beth,\beth_0$.
	\item In most of the paper, we use $\frakm$ to denote a coefficient in the immediately subleading term in $g^{00}$. In \S\ref{sec:Price}, we use it to denote the immediately subleading term in the potential, which enters the low-energy theory in the same way.
	\item In \S\ref{sec:0}, we use $\aleph,\aleph_1$ to denote the orders of conormal errors.
\end{itemize}

\subsection*{Miscellaneous}

\begin{itemize}
	\item We use $b_j,c_j$ to denote the numbers 
	\begin{align*}
		b_j &= 2^{-1}(2-d + ((d-2)^2+4\lambda_j)^{1/2}) \\ 
		c_j &= 2^{-1}(d-2 + ((d-2)^2+4\lambda_j)^{1/2}).
	\end{align*}
	\item In \S\ref{sec:Mellin}, $\calM$ denotes the Mellin transform.
	\item $a\wedge b = \operatorname{min}\{a,b\}$.
	\item When we write `$-$' at the end of an order, e.g.\ $\calA^{\alpha-}$, what we mean is an intersection, e.g.\ $\cap_{\epsilon>0} \calA^{\alpha-\epsilon}$. 
	\item Similarly, $\calA^{\alpha+}=\cup_{\epsilon>0} \calA^{\alpha+\epsilon}$. 
	\item $\hat{r}=\sigma r$.
\end{itemize}

\printbibliography

\end{document}